\documentclass{article}
\usepackage{amsmath}
\usepackage{amssymb}
\usepackage{amsthm}
\usepackage{cite}
\usepackage{euscript}
\usepackage[arrow,curve,matrix,arc,2cell]{xy}
\UseAllTwocells
\DeclareFontFamily{U}{rsfs}{} \DeclareFontShape{U}{rsfs}{n}{it}{<->
rsfs10}{} \DeclareSymbolFont{mscr}{U}{rsfs}{n}{it}
\DeclareSymbolFontAlphabet{\scr}{mscr}
\def\mathscr{\scr}
\begin{document}
\def\e#1\e{\begin{equation}#1\end{equation}}
\def\ea#1\ea{\begin{align}#1\end{align}}
\def\eq#1{{\rm(\ref{#1})}}
\theoremstyle{plain}
\newtheorem{thm}{Theorem}[section]
\newtheorem{qthm}[thm]{``Theorem''}
\newtheorem{lem}[thm]{Lemma}
\newtheorem{prop}[thm]{Proposition}
\newtheorem{cor}[thm]{Corollary}
\newtheorem{conj}[thm]{Conjecture}
\newtheorem{quest}[thm]{Question}
\newtheorem{prob}[thm]{Problem}
\theoremstyle{definition}
\newtheorem{dfn}[thm]{Definition}
\newtheorem{ex}[thm]{Example}
\newtheorem{rem}[thm]{Remark}
\newtheorem{conv}[thm]{Convention}
\newtheorem{property}[thm]{Property}
\numberwithin{equation}{section}
\numberwithin{figure}{section}
\def\Crit{\mathop{\rm Crit}}
\def\Ker{\mathop{\rm Ker}}
\def\Coker{\mathop{\rm Coker}}
\def\Hess{\mathop{\rm Hess}\nolimits}
\def\GL{\mathop{\rm GL}}
\def\Im{\mathop{\rm Im}}
\def\inc{\mathop{\rm inc}}
\def\Hom{\mathop{\rm Hom}\nolimits}
\def\Ext{\mathop{\rm Ext}\nolimits}
\def\qcoh{\mathop{\rm qcoh}}
\def\Spec{\mathop{\rm Spec}}
\def\Sets{{\mathop{\bf Sets}}}
\def\SSets{{\mathop{\bf SSets}}}
\def\id{{\mathop{\rm id}\nolimits}}
\def\Sch{{\mathop{\rm Sch}\nolimits}}
\def\Iso{{\mathop{\rm Iso}\nolimits}}
\def\fIso{{\mathop{\mathfrak{Iso}}\nolimits}}
\def\Sh{{\mathop{\rm Sh}\nolimits}}
\def\Art{{\mathop{\rm Art}\nolimits}}
\def\Le{{\mathop{\text{\rm Lis-\'et}}\nolimits}}
\def\Kalg{{\text{\rm $\K$-alg}}}
\def\Kvect{{\text{\rm $\K$-vect}}}
\def\red{{\rm red}}
\def\bs{\boldsymbol}
\def\ge{\geqslant}
\def\pr{\preceq}
\def\le{\leqslant\nobreak}
\def\O{{\mathcal O}}
\def\bX{{\bs X}}
\def\bY{{\bs Y}}
\def\bA{{\mathbin{\mathbb A}}}
\def\bG{{\mathbin{\mathbb G}}}
\def\K{{\mathbin{\mathbb K}}}
\def\bL{{\mathbin{\mathbb L}}}
\def\Z{{\mathbin{\mathbb Z}}}
\def\Q{{\mathbin{\mathbb Q}}}
\def\C{{\mathbin{\mathbb C}}}
\def\CP{{\mathbin{\mathbb{CP}}}}
\def\cA{{\mathbin{\cal A}}}
\def\cB{{\mathbin{\cal B}}}
\def\cC{{\mathbin{\cal C}}}
\def\cD{{\mathbin{\cal D}}}
\def\cE{{\mathbin{\cal E}}}
\def\cF{{\mathbin{\cal F}}}
\def\cG{{\mathbin{\cal G}}}
\def\cK{{\mathbin{\cal K}}}
\def\cL{{\mathbin{\cal L}}}
\def\cM{{\mathbin{\cal M}}}
\def\oM{{\mathbin{\smash{\,\,\overline{\!\!\mathcal M\!}\,}}}}
\def\cS{{\mathbin{\cal S}\kern -0.1em}}
\def\cSz{{\mathbin{\cal S}\kern -0.1em}^{\kern .1em 0}}
\def\al{\alpha}
\def\be{\beta}
\def\ga{\gamma}
\def\de{\delta}
\def\io{\iota}
\def\ep{\epsilon}
\def\la{\lambda}
\def\ka{\kappa}
\def\th{\theta}
\def\ze{\zeta}
\def\up{\upsilon}
\def\vp{\varphi}
\def\si{\sigma}
\def\om{\omega}
\def\De{\Delta}
\def\La{\Lambda}
\def\Om{\Omega}
\def\Up{\Upsilon}
\def\Ga{\Gamma}
\def\Si{\Sigma}
\def\Th{\Theta}
\def\pd{\partial}
\def\ts{\textstyle}
\def\st{\scriptstyle}
\def\sst{\scriptscriptstyle}
\def\w{\wedge}
\def\sm{\setminus}
\def\lt{\ltimes}
\def\bu{\bullet}
\def\sh{\sharp}
\def\op{\oplus}
\def\od{\odot}
\def\op{\oplus}
\def\ot{\otimes}
\def\ov{\overline}
\def\bigop{\bigoplus}
\def\bigot{\bigotimes}
\def\boxt{\boxtimes}
\def\iy{\infty}
\def\es{\emptyset}
\def\ra{\rightarrow}
\def\rra{\rightrightarrows}
\def\Ra{\Rightarrow}
\def\Longra{\Longrightarrow}
\def\ab{\allowbreak}
\def\longra{\longrightarrow}
\def\hookra{\hookrightarrow}
\def\dashra{\dashrightarrow}
\def\t{\times}
\def\ci{\circ}
\def\ti{\tilde}
\def\d{{\rm d}}
\def\ha{{\ts\frac{1}{2}}}
\def\md#1{\vert #1 \vert}
\def\ms#1{\vert #1 \vert^2}
\def\ban#1{\bigl\langle #1 \bigr\rangle}
\def\an#1{\langle #1 \rangle}
\title{A classical model for derived critical loci}
\author{Dominic Joyce}
\date{}
\maketitle
\begin{abstract} Let $f:U\ra\bA^1$ be a regular function on a smooth
scheme $U$ over a field $\K$. Pantev, To\"en, Vaqui\'e and Vezzosi
\cite{PTVV,Vezz} define the `derived critical locus' $\bs\Crit(f)$,
an example of a new class of spaces in derived algebraic geometry,
which they call `$-1$-shifted symplectic derived schemes'.

They show that intersections of algebraic Lagrangians in a smooth
symplectic $\K$-scheme, and stable moduli schemes of coherent
sheaves on a Calabi--Yau 3-fold over $\K$, are also $-1$-shifted
symplectic derived schemes. Thus, their theory may have applications
in algebraic symplectic geometry, and in Donaldson--Thomas theory of
Calabi--Yau 3-folds.

This paper defines and studies a new class of spaces we call
`algebraic d-critical loci', which should be regarded as classical
truncations of the $-1$-shifted symplectic derived schemes of
\cite{PTVV}. They are simpler than their derived analogues. We also
give a complex analytic version of the theory, and an extension to
Artin stacks.

In the sequels \cite{Buss,BBBJ,BBJ,BBDJS,BJM} we will apply
d-critical loci to motivic and categorified Donaldson--Thomas
theory, and to intersections of complex Lagrangians in complex
symplectic manifolds. We will show that the important structures one
wants to associate to a derived critical locus --- virtual cycles,
perverse sheaves, $\cD$-modules, and mixed Hodge modules of
vanishing cycles, and motivic Milnor fibres --- can be defined for
oriented d-critical loci.
\end{abstract}

\tableofcontents

\section{Introduction}
\label{dc1}

Pantev, To\"en, Vaqui\'e and Vezzosi \cite{PTVV,Vezz} defined a new
notion of {\it derived critical locus}. It is set in the context of
To\"en and Vezzosi's theory of derived algebraic geometry
\cite{Toen,ToVe1,ToVe2}, and consists of a quasi-smooth derived
scheme $\bX$ equipped with a $-1$-shifted symplectic structure
$\om$. In fact Pantev et al.~\cite{PTVV} define $k$-shifted
symplectic structures on derived stacks for $k\in\Z$, but the case
relevant to this paper is $k=-1$, and derived schemes rather than
derived stacks.

The following are examples of $-1$-shifted symplectic derived
schemes:
\begin{itemize}
\setlength{\itemsep}{0pt}
\setlength{\parsep}{0pt}
\item[(a)] The critical locus $\Crit(f)$ of a regular function
$f:U\ra\bA^1$ on a smooth $\K$-scheme $U$.
\item[(b)] The intersection $L\cap M$ of smooth Lagrangians $L,M$ in
an algebraic symplectic manifold $(S,\om)$.
\item[(c)] A moduli scheme $\cM$ of stable coherent sheaves on
a Calabi--Yau 3-fold.
\end{itemize}
Parts (b),(c) are the beginning of applications of these structures
in symplectic geometry and in Donaldson--Thomas theory of
Calabi--Yau 3-folds.

This paper will define and study a new class of geometric objects we
call {\it d-critical loci} $(X,s)$. They are much simpler than
$-1$-shifted symplectic derived schemes, and are entirely
`classical', by which we mean they are defined up to isomorphism in
an ordinary category using classical algebraic geometry in the style
of Hartshorne \cite{Hart}, rather than being defined up to
equivalence in an $\iy$-category using homotopy theory and derived
algebraic geometry as in~\cite{Toen,ToVe1,ToVe2}.

In fact we give two versions of the theory, {\it complex analytic
d-critical loci} $(X,s)$ in which $X$ is a complex analytic space,
and {\it algebraic d-critical loci\/} $(X,s)$ in which $X$ is a
scheme over a field $\K$. In both cases $s\in H^0(\cSz_X)$ is a
global section of a certain sheaf $\cSz_X$ on $X$, satisfying some
local conditions. When we can we give results and/or proofs for both
complex analytic and algebraic versions simultaneously, or just
briefly indicate the differences between the two.

In the algebraic case there are several topologies we could work
with --- the Zariski topology, the \'etale topology, and for an
embedding $X\hookra U$ of a $\K$-scheme $X$ into a smooth
$\K$-scheme $U$, it may be natural to consider the formal completion
$\hat U$ of $U$ along $X$, and work Zariski or \'etale locally on
$\hat U$. Whenever we can, we will use the Zariski topology. One
reason is that Theorem \ref{dc1thm4} below, proved in \cite{BJM},
requires the Zariski rather than the \'etale topology, as distinct
motives in $\oM^{\hat\mu}_X[\bL^{-1/2}]$ can become equal on an
\'etale open cover of~$X$.

To persuade the reader that d-critical loci are a useful idea, we
quote four results from the sequels \cite{Buss,BBBJ,BBJ,BBDJS,BJM}
to this paper:

\begin{thm}[Bussi, Brav and Joyce \cite{BBJ}] Suppose\/ $(\bX,\om)$
is a $-1$-shifted symplectic derived scheme in the sense of Pantev
et al.\ {\rm\cite{PTVV}} over an algebraically closed field\/ $\K$
of characteristic zero, and let\/ $X=t_0(\bX)$ be the associated
classical\/ $\K$-scheme of\/ $\bX$. Then $X$ extends naturally to an
algebraic d-critical locus\/ $(X,s)$. The canonical bundle $K_{X,s}$
from\/ {\rm\S\ref{dc24}} is naturally isomorphic to the determinant
line bundle $\det(\bL_\bX)\vert_{X^\red}$ of the cotangent complex\/
$\bL_\bX$ of\/ $\bX$. Zariski locally on $X$ one can reconstruct\/
$(\bX,\om)$ from $(X,s)$ up to equivalence, but this may not be
possible globally.
\label{dc1thm1}
\end{thm}

That is, there is a (non-full) truncation functor from $-1$-shifted
symplectic derived schemes to algebraic d-critical loci. The theorem
implies that examples (a)--(c) above have the structure of
d-critical loci.

\begin{thm}{\bf(a) \rm(Bussi, Brav and Joyce \cite[Cor.~6.8]{BBJ})}
Suppose $(S,\om)$ is an algebraic symplectic manifold over $\K,$
and\/ $L,M$ are smooth algebraic Lagrangians in $S$. Then the
intersection $X=L\cap M,$ as a $\K$-subscheme of\/ $S,$ extends
naturally to an algebraic d-critical locus\/ $(X,s)$. The canonical
bundle $K_{X,s}$ from\/ {\rm\S\ref{dc24}} is isomorphic to\/
$K_L\vert_{X^\red}\ot K_M\vert_{X^\red}$.
\smallskip

\noindent{\bf(b) \rm(Bussi \cite[\S 3]{Buss})} Suppose $(S,\om)$ is
a complex symplectic manifold, and\/ $L,M$ are complex Lagrangian
submanifolds in $S$. Then the intersection $X=L\cap M,$ as a complex
analytic subspace of\/ $S,$ extends naturally to a complex analytic
d-critical locus\/ $(X,s),$ with canonical bundle $K_{X,s}\cong
K_L\vert_{X^\red}\ot K_M\vert_{X^\red}$.
\label{dc1thm2}
\end{thm}

Theorem \ref{dc1thm2}(a) is a corollary of Theorem \ref{dc1thm1} and
\cite[Th.~2.10]{PTVV}, but Theorem \ref{dc1thm2}(b) is proved
directly, without going via derived algebraic geometry.

In \S\ref{dc24}, for a d-critical locus $(X,s)$ we construct a line
bundle $K_{X,s}$ on the reduced complex analytic subspace or
subscheme $X^\red$, called the {\it canonical bundle\/} of $(X,s)$,
and in \S\ref{dc25} we define an {\it orientation\/} on $(X,s)$ to
be a choice of square root $K_{X,s}^{1/2}$ of $K_{X,s}$ on $X^\red$.
If $X=\Crit(f)$ for $U$ a complex manifold and $f:U\ra\C$
holomorphic then $K_{X,s}\cong K_U^{\ot^2}\vert_{X^\red}$, so there
is a natural square root $K_U\vert_{X^\red}$ for $K_{X,s}$. Examples
in \S\ref{dc25} show that orientations need not exist, or be unique.
Here are two results on oriented d-critical loci:

\begin{thm}[Bussi, Brav, Dupont, Joyce and Szendr\H oi
\cite{BBDJS}] Let\/ $(X,s)$ be a complex analytic d-critical locus
with orientation $K_{X,s}^{1/2}$. Then we construct a $\Z$-perverse
sheaf\/ $P_{X,s}^\bu,$ a $\scr D$-module $D_{X,s},$ and a mixed
Hodge module $H_{X,s}^\bu$ on $X$. If\/
$(X,s),\smash{K_{X,s}^{1/2}}$ are locally modelled on
$\Crit(f),K_U\vert_{\Crit(f)^\red}$ for $U$ a complex manifold and\/
$f:U\ra\C$ holomorphic, then $P_{X,s}^\bu,D_{X,s},H_{X,s}^\bu$ are
modelled on the perverse sheaf, $\scr D$-module and mixed Hodge
module of vanishing cycles of\/~$U,f$.

Analogues hold for oriented algebraic d-critical loci $(X,s),$
yielding an algebraic $\Z$-perverse sheaf\/ $P_{X,s}^\bu,$ $\scr
D$-module $D_{X,s},$ and mixed Hodge module $H_{X,s}^\bu$ on $X$
if\/ $X$ is a $\C$-scheme, and a $\Z_l$-perverse sheaf\/
$P_{X,s}^\bu$ and a $\scr D$-module $D_{X,s}$ on $X$ if\/ $X$ is a
$\K$-scheme and\/ $l\ne \mathop{\rm char}\K$ a prime.
\label{dc1thm3}
\end{thm}

\begin{thm}[Bussi, Joyce and Meinhardt \cite{BJM}] Let\/ $(X,s)$ be
an algebraic d-critical locus over $\K$ with an orientation
$K_{X,s}^{1/2}$. Then we construct a motive $MF_{X,s}$ in a certain
ring of motives $\oM^{\hat\mu}_X$ over $X$. If\/ $(X,s),
\smash{K_{X,s}^{ 1/2}}$ are Zariski locally modelled on
$\Crit(f),K_U\vert_{\Crit(f)^\red}$ for $U$ a smooth\/ $\K$-scheme
and\/ $f:U\ra\bA^1$ regular, then $MF_{X,s}$ is locally modelled on
$\bL^{-\dim U/2}\bigl([X]-MF^{\rm mot}_{U,f}\bigr),$ where $MF^{\rm
mot}_{U,f}$ is the motivic Milnor fibre of\/~$f$.
\label{dc1thm4}
\end{thm}

In \cite{BBBJ} we will generalize Theorems \ref{dc1thm1},
\ref{dc1thm3} and \ref{dc1thm4} from $\K$-schemes to Artin
$\K$-stacks. Theorems \ref{dc1thm3} and \ref{dc1thm4} (and their
extension to stacks \cite{BBBJ}) also have applications to
extensions of Donaldson--Thomas theory of Calabi--Yau 3-folds, as in
\cite{Thom,JoSo,KoSo1,KoSo2}. Theorem \ref{dc1thm3} is important for
{\it categorification\/} of Donaldson--Thomas invariants ---
defining a graded vector space (the hypercohomology ${\mathbb
H}^*(P_{X,s}^\bu)$ of the perverse sheaf) whose dimension is the
Donaldson--Thomas invariant, as proposed by Dimca and Szendr\H oi
\cite{DiSz} --- and hence for constructing {\it cohomological Hall
algebras}, following Kontsevich and Soibelman \cite{KoSo2}. Theorem
\ref{dc1thm4} is helpful for defining {\it motivic\/}
Donaldson--Thomas invariants, as in~\cite{KoSo1}.

We have explained that d-critical loci are classical truncations of
$-1$-shifted symplectic derived schemes in \cite{PTVV}. There is
another geometric structure which is a semiclassical truncation of
$-1$-shifted symplectic derived schemes: Behrend's schemes with
symmetric obstruction theories \cite{Behr}, which we now define.

\begin{dfn} Let $X$ be a $\K$-scheme. A {\it perfect obstruction
theory\/} on $X$ in the sense of Behrend and Fantechi \cite{BeFa} is
a morphism $\phi:\cE^\bu\ra \bL_X$ in the derived category
$D(\mathop{\rm qcoh}(X))$, where $\bL_X$ is the cotangent complex of
$X$, satisfying:
\begin{itemize}
\setlength{\itemsep}{0pt}
\setlength{\parsep}{0pt}
\item[(i)] $\cE^\bu$ is quasi-isomorphic locally on $X$ to a
complex $[\cF^{-1}\ra\cF^0]$ of vector bundles in
degrees~$-1,0$;
\item[(ii)] $h^0(\phi):h^0(\cE^\bu)\ra h^0(\bL_X)$ is an isomorphism; and
\item[(iii)] $h^{-1}(\phi):h^{-1}(\cE^\bu)\ra h^{-1}(\bL_X)$ is surjective.
\end{itemize}
Following Behrend \cite{Behr}, we call $\phi:\cE^\bu\ra\bL_X$ a {\it
symmetric obstruction theory\/} if we also are given an isomorphism
$\th:\cE^\bu\ra \cE^{\bu\vee}[1]$ with~$\th^{\vee}[1]=\th$.
\label{dc1def1}
\end{dfn}

If $U$ is a smooth $\K$-scheme and $f:U\ra\bA^1$ a regular function
then $X=\Crit(f)$ has a natural symmetric obstruction theory
$\phi:\cE^\bu\ra\bL_X$ with
\e
\smash{\xymatrix@C=35pt{\cE^\bu=\bigl[TU\vert_X
\ar[r]^(0.52){\pd^2f\vert_X} & T^*U\vert_X\bigr].}}
\label{dc1eq1}
\e
But Pandharipande and Thomas \cite{PaTh} give examples of schemes
$X$ with symmetric obstruction theories with $X$ not locally
isomorphic to a critical locus. Schemes with symmetric obstruction
theories are the basis of Joyce and Song's theory of
Donaldson--Thomas invariants of Calabi--Yau 3-folds \cite{JoSo}. If
$(\bX,\om)$ is a $-1$-shifted symplectic derived scheme in the sense
of Pantev et al.\ \cite{PTVV}, then the classical scheme
$X=t_0(\bX)$ has a symmetric obstruction theory
$\phi:\cE^\bu\ra\bL_X$ with $\cE^\bu=i^*(\bL_\bX)$ and
$\th=i^*(\om_0)$, where $i:X\hookra\bX$ is the inclusion.

We illustrate the relations between these structures in Figure
\ref{dc1fig1}. The two dotted arrows `$\dashra$' indicate a
construction which works locally, but not globally. That is, given
an algebraic d-critical locus $(X,s)$, then Zariski locally on $X$
we can construct both a $-1$-shifted symplectic derived scheme
$(\bX,\om)$, and a symmetric obstruction theory
$\phi:\cE^\bu\ra\bL_X$, $\th$, uniquely up to equivalence, but we
cannot combine these local models to make $(\bX,\om)$ or
$\cE^\bu,\phi,\th$ globally on $X$ because of difficulties with
gluing `derived' objects on open covers.

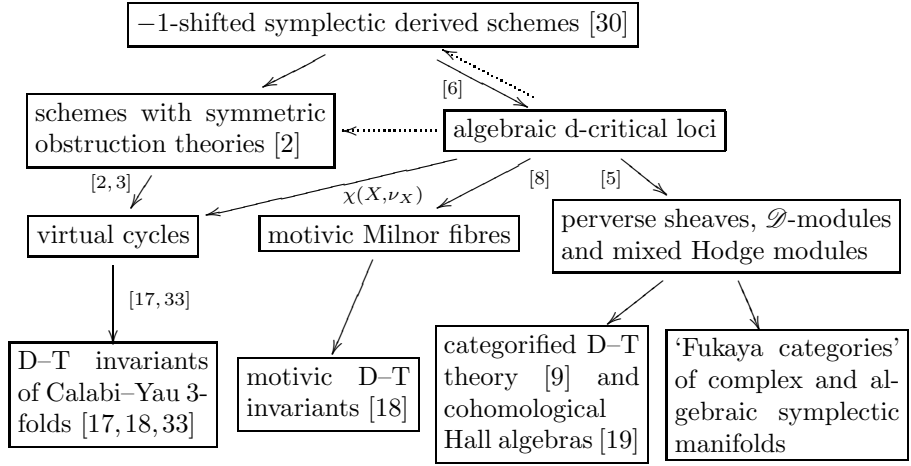
\begin{figure}[htb]
\xymatrix@!0@R=20pt@C=25.5pt{
&&&& \fbox{$-1$-shifted symplectic derived schemes \cite{PTVV}} \ar[ddlll]
\ar@<-.5ex>[ddrrr]_(0.4){\text{\cite{BBJ}}} \\ \\
& \fbox{\parbox{110pt}{schemes with symmetric obstruction theories \cite{Behr}}}
\ar[ddl]_(0.6){\text{\cite{Behr,BeFa}}} &&&&&&
\fbox{algebraic d-critical loci} \ar@<-.5ex>@{.>}[uulll] \ar@{.>}[llllll]
\ar[ddlll]^(0.3){\text{\cite{BJM}}} \ar[ddrr]_(0.3){\text{\cite{BBDJS}}}
\ar[ddlllllll]^(0.45){\chi(X,\nu_X)} \\ \\
\fbox{virtual cycles} \ar[ddd]^(0.4){\text{\cite{Thom,JoSo}}} &&&&
{}\,\,\,\framebox{motivic Milnor fibres} \ar[dddl]
&&&&& {}\,\,\,\fbox{\parbox{125pt}{perverse sheaves, $\scr D$-modules
and mixed Hodge modules}}
\ar[dddlll] \ar[dddr] \\ \\ \\
\fbox{\parbox{72pt}{D--T invariants of Calabi--Yau 3-folds
\cite{Thom,JoSo,KoSo1}}} &&&
{}\,\,\,\,\,\,\fbox{\parbox{62pt}{motivic D--T invariants \cite{KoSo1}}} &&&
{}\,\,\,\,\,\,\,\,\,\,\,\fbox{\parbox{74pt}{categorified D--T theory \cite{DiSz}
and cohomological Hall algebras \cite{KoSo2}}} &&&&
\fbox{\parbox{85pt}{`Fukaya categories' of complex and
algebraic symplectic manifolds}} }
\caption{Relations between different structures, and applications}
\label{dc1fig1}
\end{figure}

If $X$ is a proper $\K$-scheme with obstruction theory
$\phi:\cE^\bu\ra\bL_X$ then Behrend and Fantechi define a {\it
virtual cycle\/} $[X]^{\rm vir}$ in Chow homology $A_*(X)$. If the
obstruction theory is symmetric, and $\K$ algebraically closed of
characteristic zero, then Behrend \cite{Behr} (see also \cite[\S
4]{JoSo}) shows that $[X]^{\rm vir}\in A_0(X)$, and
\e
\ts\int_{[X]^{\rm vir}}1=\chi(X,\nu_X),
\label{dc1eq2}
\e
where $\nu_X$ is a $\Z$-valued constructible function on $X$ called
the {\it Behrend function}, which depends only on $X$ as a
$\K$-scheme. In particular, $[X]^{\rm vir}$ is independent of the
choice of symmetric obstruction theory on~$X$.

If $(X,s)$ is a proper algebraic d-critical locus, we define the
{\it virtual cycle\/} of $X$ to be $\chi(X,\nu_X)\in\Z$, as in
\eq{dc1eq2}. Although we will not do it in this paper, one can
define a notion of family of d-critical loci over a base $Y$, and
show that the virtual cycles of a proper family of d-critical loci
are locally constant on~$Y$.

Example \ref{dc2ex4} below shows that locally, schemes with
symmetric obstruction theories can contain strictly less information
than algebraic d-critical loci. On the other hand, Example
\ref{dc2ex5} shows that schemes with (symmetric) obstruction
theories can contain global, nonlocal information (in the form of a
class in $\Ext^2(T^*X,T^*X^\vee)$) which is forgotten by algebraic
d-critical loci.

The author and his collaborators tried for some time to construct
perverse sheaves, and motivic Milnor fibres, from a scheme with
symmetric obstruction theory, but failed, and the author now
believes this is not possible. So, one moral of Figure \ref{dc1fig1}
is that d-critical loci are more useful than schemes with symmetric
obstruction theories for various applications.
\medskip

\noindent{\bf Conventions.} Throughout $\K$ will be an algebraically
closed field with $\mathop{\rm char}\K\ne 2$. As in Theorem
\ref{dc1thm1}, the sequel \cite{BBJ} and those parts of
\cite{BBBJ,BBDJS,BJM} which depend on \cite{BBJ} also require
$\mathop{\rm char}\K=0$, but this paper does not need $\mathop{\rm
char}\K=0$. All complex analytic spaces, $\K$-schemes and Artin
$\K$-stacks $X$ will be assumed to be locally of finite type, as
this is necessary for the existence of local embeddings $X\hookra U$
with $U$ a complex manifold or smooth $\K$-scheme.

\medskip

\noindent{\bf Acknowledgements.} I would like to thank Oren
Ben-Bassat, Dennis Borisov, Chris Brav, Tom Bridgeland, Vittoria
Bussi, Stephane Guillermou, Frances Kirwan, Davesh Maulik, Sven Meinhardt, Anatoly Preygel, Pierre Schapira, Edward Segal, Bal\'azs Szendr\H oi, and Bertrand To\"en for helpful conversations. This research was supported by EPSRC Programme Grant EP/I033343/1.

\section{The main results}
\label{dc2}

This section, the heart of the paper, gives our central definitions,
the main results, and some examples. The proofs of results stated in
\S\S\ref{dc21}, \ref{dc22}, \ref{dc23}, \ref{dc24}, \ref{dc26}, and \ref{dc28}
will be deferred until sections \ref{dc3}--\ref{dc8}, respectively.

Sections \ref{dc21}--\ref{dc26} concern d-critical structures on
complex analytic spaces and $\K$-schemes. Some good background
references on complex analytic spaces and analytic coherent sheaves
upon them are Gunning and Rossi \cite{GuRo} and Grauert and Remmert
\cite{GrRe}. A good book on $\K$-schemes and sheaves in algebraic
geometry is Hartshorne \cite{Hart}. The relationship between
$\C$-schemes and complex analytic spaces is discussed in Hartshorne
\cite[App.~B]{Hart} and Serre~\cite{Serr}.

After some background material on Artin stacks and sheaves upon them
in \S\ref{dc27}, section \ref{dc28} extends parts of
\S\ref{dc21}--\S\ref{dc26} from $\K$-schemes to Artin $\K$-stacks.
Given an Artin $\K$-stack $X$, the main idea is to consider smooth
1-morphisms $t:T\ra X$ from $\K$-schemes $T$, and apply the results
of \S\ref{dc21}--\S\ref{dc26} on~$T$.

\subsection{The sheaves $\cS_X,\cSz_X$ and their properties}
\label{dc21}

The next theorem, which will be proved in
\S\ref{dc31}--\S\ref{dc33}, associates a sheaf $\cS_X$ to each
complex analytic space (or $\K$-scheme) $X$, such that (very
roughly) sections of $\cS_X$ parametrize different ways of writing
$X$ as $\Crit(f)$ for $U$ a complex manifold (or smooth $\K$-scheme)
and $f:U\ra\C$ holomorphic (or $f:U\ra\bA^1$ regular). This will be
needed in the definition of d-critical loci in~\S\ref{dc22}.

Note our convention from \S\ref{dc1} that all complex analytic
spaces and $\K$-schemes $X$ in this paper are {\it locally of finite
type}, which is necessary for the existence of embeddings
$i:X\hookra U$ for $U$ a complex manifold or smooth $\K$-scheme.

\begin{thm} Let\/ $X$ be a complex analytic space. Then there
exists a sheaf\/ $\cS_X$ of commutative\/ $\C$-algebras on $X,$
unique up to canonical isomorphism, which is uniquely characterized
by the following two properties:
\begin{itemize}
\setlength{\itemsep}{0pt}
\setlength{\parsep}{0pt}
\item[{\bf(i)}] Suppose\/ $U$ is a complex manifold, $R$ is an open
subset in\/ $X,$ and\/ $i:R\hookra U$ is an embedding of\/ $R$
as a closed complex analytic subspace of\/ $U$. Then we have an
exact sequence of sheaves of\/ $\C$-vector spaces on $R\!:$
\e
\smash{\xymatrix@C=30pt{ 0 \ar[r] & I_{R,U} \ar[r] &
i^{-1}(\O_U) \ar[r]^{i^\sh} & \O_X\vert_R \ar[r] & 0, }}
\label{dc2eq1}
\e
where $\O_X,\O_U$ are the sheaves of holomorphic functions on
$X,U,$ and\/ $i^\sh$ is the morphism of sheaves of\/
$\C$-algebras on $R$ induced by $i,$ which is surjective as $i$
is an embedding, and\/ $I_{R,U}=\Ker(i^\sh)$ is the sheaf of
ideals in $i^{-1}(\O_U)$ of functions on $U$ near $i(R)$ which
vanish on $i(R)$.

There is an exact sequence of sheaves of\/ $\C$-vector spaces on
$R\!:$
\e
\xymatrix@C=20pt{ 0 \ar[r] & \cS_X\vert_R
\ar[rr]^(0.4){\io_{R,U}} &&
\displaystyle\frac{i^{-1}(\O_U)}{I_{R,U}^2} \ar[rr]^(0.4)\d &&
\displaystyle\frac{i^{-1}(T^*U)}{I_{R,U}\cdot i^{-1}(T^*U)}\,, }
\label{dc2eq2}
\e
where $\d$ maps $f+I_{R,U}^2\mapsto \d f+I_{R,U}\cdot
i^{-1}(T^*U),$ and\/ $\io_{R,U}$ is a morphism of sheaves of
commutative\/ $\C$-algebras.
\item[{\bf(ii)}] Let\/ $R,U,i,\io_{R,U}$ and\/
$S,V,j,\io_{S,V}$ be as in {\bf(i)} with\/ $R\subseteq
S\subseteq X,$ and suppose $\Phi:U\ra V$ is holomorphic with\/
$\Phi\ci i=j\vert_{R}$ as a morphism of complex analytic spaces
$R\ra V$. Then the following diagram of sheaves on $R$ commutes:
\e
\begin{gathered}
\xymatrix@C=12pt{ 0 \ar[r] & \cS_X\vert_{R} \ar[d]^\id
\ar[rrr]^(0.4){\io_{S,V}\vert_{R}} &&&
\displaystyle\frac{j^{-1}(\O_{V})}{I_{S,V}^2}\Big\vert_{R}
\ar@<-2ex>[d]^{i^{-1}(\Phi^\sh)} \ar[rr]^(0.4)\d &&
\displaystyle\frac{j^{-1}(T^*V)}{I_{S,V}\cdot
j^{-1}(T^*V)}\Big\vert_{R} \ar@<-2ex>[d]^{i^{-1}(\d\Phi)} \\
 0 \ar[r] & \cS_X\vert_{R} \ar[rrr]^(0.4){\io_{R,U}} &&&
\displaystyle\frac{i^{-1}(\O_{U})}{I_{R,U}^2} \ar[rr]^(0.4)\d &&
\displaystyle\frac{i^{-1}(T^*U)}{I_{R,U}\cdot i^{-1}(T^*U)}\,.
}\!\!\!\!\!\!\!{}
\end{gathered}
\label{dc2eq3}
\e
Here $\Phi:U\ra V$ induces $\Phi^\sh:\Phi^{-1}(\O_{V})\ra\O_{U}$
on $U,$ so we have
\e
i^{-1}(\Phi^\sh):j^{-1}(\O_{V})\vert_{R}=i^{-1}\ci
\Phi^{-1}(\O_{V})\longra i^{-1}(\O_{U}),
\label{dc2eq4}
\e
a morphism of sheaves of\/ $\C$-algebras on $R$. As $\Phi\ci
i=j\vert_{R},$ equation \eq{dc2eq4} maps $I_{S,V}\vert_{R}\ra
I_{R,U},$ and so maps $I_{S,V}^2\vert_{R}\ra I_{R,U}^2$. Thus
\eq{dc2eq4} induces the morphism of sheaves of\/ $\C$-algebras
in the second column of\/ \eq{dc2eq3}. Similarly,
$\d\Phi:\Phi^{-1}(T^*V)\ra T^*U$ induces the third column
of\/~\eq{dc2eq3}.
\end{itemize}

These sheaves $\cS_X$ also satisfy:
\begin{itemize}
\setlength{\itemsep}{0pt}
\setlength{\parsep}{0pt}
\item[{\bf(a)}] There is a natural decomposition\/
$\cS_X=\C_X\op\cSz_X,$ where\/ $\C_X$ is the constant sheaf on
$X$ with fibre $\C,$ as a sheaf of\/ $\C$-subalgebras in
$\cS_X,$ and\/ $\cSz_X\subset\cS_X$ is a sheaf of ideals in
$\cS_X,$ the kernel of the composition of morphisms of sheaves
of commutative\/ $\C$-algebras
\e
\xymatrix@C=40pt{ \cS_X \ar[r]^{\be_X} & \O_X
\ar[r]^(0.47){i_X^\sh} & \O_{X^\red},  }
\label{dc2eq5}
\e
with\/ $X^\red$ the reduced complex analytic subspace of\/ $X,$
and\/ $i_X:X^\red\hookra X$ the inclusion.
\item[{\bf(b)}] There are natural exact sequences of sheaves of\/
$\C$-vector spaces on $X\!:$
\ea
&\xymatrix@C=13.5pt{ 0 \ar[r] & h^{-1}(\bL_X)
\ar[rr]_(0.55){\al_X} && \cS_X \ar[rr]_{\be_X} && \O_X
\ar[rr]_(0.33)\d && T^*X\cong h^0(\bL_X), }
\label{dc2eq6}\\
&\xymatrix@C=13.5pt{ 0 \ar[r] & h^{-1}(\bL_X)
\ar[rr]^(0.55){\al_X^0} && \cSz_X \ar[rr]^{\be_X^0} && \O_X
\ar[rr]^(0.33){\d\op i_X^\sh} && T^*X\op\O_{X^\red}, }
\label{dc2eq7}
\ea
where $\bL_X$ is the cotangent complex and\/ $T^*X$ the
cotangent sheaf of\/~$X$.

\item[{\bf(c)}] The sheaf\/ $\cSz_X$ is canonically isomorphic
to the cohomology of the complex
\e
\xymatrix@C=40pt{ I_{R,U}^2 \ar[r]^(0.35){\d} & I_{R,U}\cdot
i^{-1}(T^*U) \ar[r]^{\d} & i^{-1}(\La^2T^*U).  }
\label{dc2eq8}
\e
\end{itemize}

With the exception of\/ {\bf(c)\rm,} the analogue of all the above
also holds for schemes over a field\/ $\K$ in algebraic geometry,
taking $X$ to be a $\K$-scheme with structure sheaf\/ $\O_X$ and
reduced\/ $\K$-subscheme $X^\red,$ and\/ $\cS_X$ a sheaf of
commutative\/ $\K$-algebras on $X$ in either the Zariski or the
\'etale topology, and\/ $R\subseteq X$ a Zariski open
$\K$-subscheme, and\/ $U$ a smooth $\K$-scheme, and replacing $\C_X$
by $\K_X$. For {\bf(c)\rm,} we must replace $U$ by the formal
completion $\hat U$ of\/ $U$ along $i(R),$ so that the analogue of\/
\eq{dc2eq8} is
\e
\begin{gathered}
\xymatrix@C=40pt{ I_{R,\smash{\hat U}}^2 \ar[r]^(0.35){\d} &
I_{\smash{R,\hat U}}\cdot i^{-1}(T^*\hat U) \ar[r]^{\d} &
i^{-1}(\La^2T^*\hat U), }\quad\text{where}\\
i^{-1}(\O_{\smash{\hat U}})=\lim_{n\ra\iy} i^{-1}(\O_U)/I_{R,U}^n,
\quad I_{\smash{R,\hat U}}=\lim_{n\ra\iy} I_{R,U}/I_{R,U}^n\subset
i^{-1}(\O_{\smash{\hat U}}).
\end{gathered}
\label{dc2eq9}
\e
\label{dc2thm1}
\end{thm}

Here in part (b), for ({\it co}){\it tangent complexes\/} of
$\K$-schemes see Illusie \cite{Illu1,Illu2}, and of complex analytic
spaces see Palamodov~\cite{Pala1,Pala2,Pala3,Pala4}.

\begin{rem}{\bf(a)} In this paper and the sequels
\cite{Buss,BBDJS,BBBJ,BBJ,BJM} we will make no use of the fact that
$\cS_X$ is a {\it sheaf of commutative\/ $\C$-algebras}, rather than
just a sheaf of sets. Material in \cite[Th.~6.9]{BBDJS} on Verdier
duality and monodromy isomorphisms $\Si_{X,s},{\rm T}_{X,s}$ for the
perverse sheaves $P_{X,s}^\bu$ in Theorem \ref{dc1thm3} above
depends implicitly on being able to multiply $s\in H^0(\cSz_X)$ by
$-1$ or by ${\rm e}^{i\th}$, but we have not yet found an
application for the additive or multiplicative structures on
$\cS_X,\cSz_X$. Although $(X,\cS_X)$ is a locally ringed space, it
is generally far from being a scheme. The ideals $\cSz_X$ in $\cS_X$
need not be square-zero, so the multiplicative structure on
$\cS_X=\C_X\op\cSz_X$ can be nontrivial.
\smallskip

\noindent{\bf(b)} Equation \eq{dc2eq6} suggests the following
interpretation of the sheaf $\cS_X$: on a complex analytic space or
$\K$-scheme $X$ we have the {\it de Rham differential\/} $\d_{\rm
dR}:\O_X\ra\bL_X$, which is a morphism in
$D\mathop{\text{mod-}\C_X}$ or $D\mathop{\text{mod-}\K_X}$, the
derived category of complexes of sheaves of $\C$- or $\K$-vector
spaces on $X$. Write $\cD^\bu_X$ for the cone on $\d_{\rm dR}$, so
that we have a distinguished triangle
\e
\xymatrix@C=30pt{ \O_X \ar[r]^{\d_{\rm dR}} & \bL_X \ar[r] &
\cD^\bu_X \ar[r] & \O_X[1]. }
\label{dc2eq10}
\e
Comparing \eq{dc2eq6} with the long exact sequence of cohomology
sheaves of \eq{dc2eq10}, we see that~$\cS_X\cong h^{-1}(\cD^\bu_X)$.
\label{dc2rem1}
\end{rem}

There are natural {\it pullback morphisms\/} $\phi^\star$ for the
sheaves~$\cS_X,\cSz_X$:

\begin{prop} Let\/ $\phi:X\ra Y$ be a morphism of complex analytic
spaces, and\/ $\cS_X,\ab\cSz_X,\ab\io_{R,U},\ab
I_{R,U},\ab\cS_Y,\cSz_Y, \io_{S,V},I_{S,V}$ be as in Theorem\/
{\rm\ref{dc2thm1}}. Then there is a unique morphism
$\phi^\star:\phi^{-1}(\cS_Y)\ra \cS_X$ of sheaves of commutative\/
$\C$-algebras on $X,$ which maps $\phi^{-1}(\cSz_Y)\ra \cSz_X,$ such
that if\/ $R\subseteq X,$ $S\subseteq Y$ are open with\/
$\phi(R)\subseteq S,$ $U,V$ are complex manifolds, $i:R\hookra U,$
$j:S\hookra V$ are closed embeddings, and\/ $\Phi:U\ra V$ is
holomorphic with\/ $\Phi\ci i=j\ci\phi\vert_R:R\ra V,$ then as for
\eq{dc2eq3} the following diagram of sheaves on $R$ commutes:
\e
\begin{gathered}
\xymatrix@R=20pt@C=9pt{ 0 \ar[r] & \phi^{-1}(\cS_Y)\vert_R
\ar[rrr]_(0.45){\phi^{-1}(\io_{S,V})\vert_R}
\ar[d]^{\phi^\star\vert_R} &&& {\frac{\ts\phi^{-1}\ci
j^{-1}(\O_V)\vert_R}{\ts\phi^{-1}(I_{S,V}^2)\vert_R}}
\ar[d]^{i^{-1}(\Phi^\sharp)} \ar[rr]_(0.42){\phi^{-1}(\d)} &&
\frac{\ts \phi^{-1}(j^{-1}(T^*V))\vert_R}{\ts \phi^{-1}(I_{S,V}\cdot
j^{-1}(T^*V))\vert_R}
\ar[d]_{i^{-1}(\d\Phi)} \\
0 \ar[r] & \cS_X\vert_R \ar[rrr]^(0.45){\io_{R,U}} &&& {\frac{\ts
i^{-1}(\O_U)}{\ts I_{R,U}^2}} \ar[rr]^(0.42)\d && \frac{\ts
i^{-1}(T^*U)}{\ts I_{R,U}\cdot i^{-1}(T^*U)}\,.
}\!\!\!\!\!\!\!\!\!\!\!\!\!\!\!{}
\end{gathered}
\label{dc2eq11}
\e

If\/ $\psi:Y\ra Z$ is another morphism of complex analytic spaces,
then
\e
(\psi\ci\phi)^\star=\phi^\star\ci\phi^{-1}(\psi^\star):
(\psi\ci\phi)^{-1}(\cS_Z)=\phi^{-1}\ci\psi^{-1}(\cS_Z)\longra \cS_X.
\label{dc2eq12}
\e
If\/ $\phi:X\ra Y$ is $\id_X:X\ra X$ then
$\id_X^\star=\id_{\cS_X}:\id_X^{-1}(\cS_X)=\cS_X\ra\cS_X$.

If\/ $\phi:X\ra Y$ is an \'etale morphism of complex analytic
spaces, then $\phi^\star:\phi^{-1}(\cS_Y)\ra\cS_X$ is an isomorphism
of sheaves of commutative $\C$-algebras.

With the exception of the last part, the analogue of all the above
holds for schemes over a field\/ $\K$ in algebraic geometry, taking
$\phi:X\ra Y$ to be a morphism of\/ $\K$-schemes, $R\subseteq X,$
$S\subseteq Y$ to be Zariski open, and\/ $U,V$ to be smooth\/
$\K$-schemes, and taking\/ $\cS_X,\cS_Y$ to be sheaves of
commutative\/ $\K$-algebras on $X,Y$ in either the Zariski or the
\'etale topology, as in Theorem\/~{\rm\ref{dc2thm1}}.

For the last part, if\/ $\cS_X,\cS_Y$ are sheaves in the Zariski
topology, then\/ $\phi^\star$ is an isomorphism if $\phi:X\ra Y$ is
a Zariski open inclusion, and if\/ $\cS_X,\cS_Y$ are sheaves in the
\'etale topology, then\/ $\phi^\star$ is an isomorphism if\/
$\phi:X\ra Y$ is \'etale.

\label{dc2prop1}
\end{prop}

The next example, which is central to our theory, shows the point of
$\cS_X$.

\begin{ex} Let $U$ be a complex manifold, $f:U\ra\C$ be
holomorphic, and $X=\Crit(f)$, as a closed complex analytic subspace
of $U$. Write $i:X\hookra U$ for the inclusion, and
$I_{X,U}\subseteq i^{-1}(\O_U)$ for the sheaf of ideals vanishing on
$X\subseteq U$. Then $i^{-1}(f)\in H^0\bigl(i^{-1}(\O_U)\bigr)$ with
$\d\bigl(i^{-1}(f)\bigr)\in H^0\bigl(I_{X,U}\cdot
i^{-1}(T^*U)\bigr)\subseteq H^0\bigl(i^{-1}(T^*U)\bigr)$, so
$i^{-1}(f)+I_{X,U}^2\in H^0\bigl(i^{-1}(\O_U)/I_{X,U}^2\bigr)$ with
$\d\bigl(i^{-1}(f)+I_{X,U}^2\bigr)=0$ in
$H^0\bigl(i^{-1}(T^*U)/I_{X,U} \cdot i^{-1}(T^*U)\bigr)$. Thus by
equation \eq{dc2eq2} with $R=X$, we see there is a unique section
$s\in H^0(\cS_X)$ with~$\io_{X,U}(s)=i^{-1}(f)+I_{X,U}^2$.

Thus, if we can write $X=\Crit(f)$ for $f:U\ra\C$ holomorphic, then
we obtain a natural section $s\in H^0(\cS_X)$. Essentially
$s=f+I_{\d f}^2$, where $I_{\d f}\subseteq\O_U$ is the ideal
generated by $\d f$. Note that $f\vert_X=f+I_{\d f}$, so $s$
determines $f\vert_X$. Basically, $s$ remembers all of the
information about $f$ which makes sense intrinsically on $X$, rather
than on the ambient space~$U$.

We can also explain the decomposition $\cS_X=\C_X\op\cSz_X$ in this
example. We will see in Example \ref{dc2ex2} that if $X=\Crit(f)$
then $f$ need not be locally constant on $X$, but $f$ is locally
constant on the reduced complex analytic space $X^\red$. Since
locally constant functions on $X^\red\subseteq U$ extend uniquely to
locally constant functions on $U$ near $X^\red$, after shrinking $U$
we can uniquely write $f=c+f^0$, where $c:U\ra\C$ is locally
constant and $f^0:U\ra\C$ has $f^0\vert_{X^\red}=0$. Then $c,f^0$
correspond to the components of $s$ in $H^0(\C_X),H^0(\cSz_X)$,
and~$X=\Crit(f^0)$.

The analogue also holds in the algebraic case, with $U$ a smooth
$\K$-scheme and $f:U\ra\bA^1$ a regular function.
\label{dc2ex1}
\end{ex}

\subsection{The definition of d-critical loci, and some examples}
\label{dc22}

We can now define d-critical loci:

\begin{dfn} A ({\it complex analytic\/}) {\it d-critical locus\/}
is a pair $(X,s)$, where $X$ is a complex analytic space, and $s\in
H^0(\cSz_X)$ for $\cSz_X$ as in Theorem \ref{dc2thm1}, satisfying
the condition that for each $x\in X$, there exists an open
neighbourhood $R$ of $x$ in $X$, a complex manifold $U$, a
holomorphic function $f:U\ra\C$, and an embedding $i:R\hookra U$ of
$R$ as a closed complex analytic subspace of $U$, such that
$i(R)=\Crit(f)$ as complex analytic subspaces of $U$,
and~$\io_{R,U}(s\vert_R)=i^{-1}(f)+I_{R,U}^2$.

Similarly, for $\K$-schemes we define an ({\it algebraic\/}) {\it
d-critical locus\/} to be a pair $(X,s)$, where $X$ is a
$\K$-scheme, and $s\in H^0(\cSz_X)$ for $\cS_X$ as in Theorem
\ref{dc2thm1}, such that $X$ may be covered by Zariski open sets
$R\subseteq X$ with a closed embedding $i:R\hookra U$ into a smooth
$\K$-scheme $U$ and a regular function $f:U\ra\bA^1=\K$, such that
$i(R)=\Crit(f)$ as $\K$-subschemes of $U$,
and~$\io_{R,U}(s\vert_R)=i^{-1}(f)+I_{R,U}^2$.

In both cases we call the quadruple $(R,U,f,i)$ a {\it critical
chart\/} on~$(X,s)$.

A {\it morphism\/} $\phi:(X,s)\ra (Y,t)$ of d-critical loci
$(X,s),(Y,t)$ (either complex analytic or algebraic) is a morphism
$\phi:X\ra Y$ (of complex analytic spaces or $\K$-schemes) such that
$\phi^\star(t)=s$, for $\phi^\star$ as in Proposition
\ref{dc2prop1}. If $\phi:(X,s)\ra (Y,t)$, $\psi:(Y,t)\ra(Z,u)$ are
morphisms then equation \eq{dc2eq12} implies that
$\psi\ci\phi:(X,s)\ra(Z,u)$ is a morphism, and the last part of
Proposition \ref{dc2prop1} shows that $\id_X:(X,s)\ra (X,s)$ is a
morphism. Thus, (complex analytic or algebraic) {\it d-critical loci
form a category}.
\label{dc2def1}
\end{dfn}

\begin{rem}{\bf(a)} In Definition \ref{dc2def1}, we could instead
have defined a d-critical locus $(X,s)$ to have $s\in H^0(\cS_X)$
rather than $s\in H^0(\cSz_X)$, but with the rest of the definition
the same. The difference is this: as in Example \ref{dc2ex1}, if
$X=\Crit(f)$ for holomorphic $f:U\ra\C$, then $f\vert_{X^\red}$ is
locally constant, and we can write $f=f^0+c$ uniquely near $X$ in
$U$ for $f^0:U\ra\C$ holomorphic with $\Crit(f^0)=X=\Crit(f)$,
$f^0\vert_{X^\red}=0$, and $c:U\ra\C$ locally constant
with~$c\vert_{X^\red}=f\vert_{X^\red}$.

Defining d-critical loci using $s\in H^0(\cSz_X)$, as we have done,
corresponds to remembering only the function $f^0$ near $X$ in $U$,
and forgetting the locally constant function
$f\vert_{X^\red}:X^\red\ra\C$. Equivalently, it corresponds to
remembering the closed 1-form $\d f=\d f^0$ on $U$ near $X=(\d
f)^{-1}(0)$. In the applications the author has in mind
\cite{BBBJ,Buss,BBDJS,BBJ,BJM}, taking $s$ in $H^0(\cSz_X)$ rather
than $H^0(\cS_X)$ is more natural, as there is no canonical value
for $f\vert_{X^\red}$ other than $f\vert_{X^\red}=0$. Also
\eq{dc2eq8}--\eq{dc2eq9} give an alternative description for
$\cSz_X$ rather than~$\cS_X$.
\smallskip

\noindent{\bf(b)} As in Theorem \ref{dc1thm1}, in
\cite[Th.~6.6]{BBJ} we define a truncation functor from $-1$-shifted
symplectic derived $\K$-schemes $(\bX,\om)$ in the sense of Pantev
et al.\ \cite{PTVV} to algebraic d-critical loci $(X,s)$, so that
algebraic d-critical loci may be regarded as classical truncations
of $-1$-shifted symplectic derived $\K$-schemes.

If we define a morphism $\bs\phi:(\bX,\om)\ra(\bY,\om')$ of
$-1$-shifted symplectic derived $\K$-schemes to be a morphism
$\bs\phi:\bX\ra\bY$ of derived $\K$-schemes with
$\bs\phi^*(\om')\simeq\om$, this forces $\bs\phi$ to be \'etale.
However, the notion of morphism $\phi:(X,s)\ra (Y,t)$ of d-critical
loci in Definition \ref{dc2def1} is more general, e.g.\ $\phi:X\ra
Y$ can be smooth of positive dimension, as in
Proposition~\ref{dc2prop3}.
\smallskip

\noindent{\bf(c)} For $(X,s)$ to be a (complex analytic or
algebraic) d-critical locus places strong local restrictions on the
singularities of $X$. For example, Behrend \cite{Behr} notes that if
$X$ has reduced local complete intersection singularities then
locally it cannot be the zeroes of an almost closed 1-form on a
smooth space, and hence not locally a critical locus, and
Pandharipande and Thomas \cite{PaTh} give examples which are zeroes
of almost closed 1-forms, but are not locally critical loci.
\label{dc2rem2}
\end{rem}

On a d-critical locus $(X,s)$, any closed embedding $X\supseteq
R\,{\buildrel i\over\longra}\,U$ with $U$ smooth can be made into a
critical chart $(R',U',f',i')$, after shrinking~$R,U$.

\begin{prop} Suppose $(X,s)$ is a complex analytic d-critical
locus, $R\subseteq X$ is open, and\/ $i:R\hookra U$ is a closed
embedding, where $U$ is a complex manifold. Then for each\/ $x\in
R,$ there exist open $x\in R'\subseteq R$ and\/ $i(R')\subseteq
U'\subseteq U$ and a holomorphic function $f':U'\ra\C$ such that\/
$(R',U',f',i')$ is a critical chart on $(X,s),$ where
$i'=i\vert_{R'}:R'\hookra U'$.

Suppose also that\/ $\dim U=\dim T_xX,$ so that\/ $\d
i\vert_x:T_xX\ra T_{i(x)}U$ is an isomorphism, and\/ $f:U\ra\C$ is
holomorphic with $\io_{R,U}(s\vert_R)=i^{-1}(f)+I_{R,U}^2$. Then we
may take $f'=f\vert_{U'}$ in the critical chart\/~$(R',U',f',i')$.

The analogue holds for algebraic d-critical loci, with\/ $U$ a
smooth\/ $\K$-scheme, $R'\subseteq R\subseteq X,$ $U'\subseteq U$
Zariski open, and\/ $f:U\ra\bA^1,$ $f':U'\ra\bA^1$ regular.
\label{dc2prop2}
\end{prop}

The next result will be useful in~\S\ref{dc28}.

\begin{prop} Let\/ $\phi:X\ra Y$ be a smooth morphism of complex
analytic spaces or $\K$-schemes. Suppose $t\in H^0(\cSz_Y),$ and
set\/ $s:=\phi^\star(t)\in H^0(\cSz_X),$ for $\phi^\star$ as in
Proposition\/ {\rm\ref{dc2prop1}}. If\/ $(Y,t)$ is a d-critical
locus, then\/ $(X,s)$ is a d-critical locus, and\/
$\phi:(X,s)\ra(Y,t)$ is a morphism of d-critical loci.

Conversely, if also $\phi:X\ra Y$ is surjective, then $(X,s)$ a
d-critical locus implies $(Y,t)$ is a d-critical locus.
\label{dc2prop3}
\end{prop}

As in Hartshorne \cite[App.~B]{Hart} and Serre \cite{Serr}, there is
an analytification functor from algebraic $\C$-schemes $X$ to
complex analytic spaces $X^{\rm an}$, where the points of $X^{\rm
an}$ are the $\C$-points of $X$. It is easy to show that this
extends to d-critical loci, and we leave the details to the reader:

\begin{prop} Let\/ $(X,s)$ be an algebraic d-critical locus over
the field\/ $\K=\C$. Then the complex analytic space $X^{\rm an}$
associated to the $\C$-scheme $X$ extends naturally to a complex
analytic d-critical locus $(X^{\rm an},s^{\rm an})$.
\label{dc2prop4}
\end{prop}

The proofs of the following lemma and proposition are also
more-or-less immediate, and we leave them as exercises.

\begin{lem} Let\/ $(X,s)$ be a d-critical locus, and\/ $0\ne c\in\C$
or\/ $0\ne c\in\K$. Then $(X,c\cdot s)$ is also a d-critical locus,
and if\/ $(R,U,f,i)$ is a critical chart on $(X,s)$ then
$(R,U,c\cdot f,i)$ is a critical chart on $(X,c\cdot s)$.
\label{dc2lem1}
\end{lem}

\begin{prop} Let\/ $(X,s),(Y,t)$ be d-critical loci. Write
$\pi_X:X\t Y\ra X,$ $\pi_Y:X\t Y\ra Y$ for the projections, and
define $s\boxplus t:=\pi_X^\star(s)+\pi_Y^\star(t)$ in
$H^0(\cSz_{X\t Y}),$ for $\pi_X^\star,\pi_Y^\star$ as in
Proposition\/ {\rm\ref{dc2prop1}}. Then $(X\t Y,s\boxplus t)$ is a
d-critical locus, and if\/ $(R,U,f,i),$ $(S,V,g,j)$ are critical
charts on $(X,s),$ $(Y,t)$ respectively then $(R\t S,U\t V,f\boxplus
g,i\t j)$ is a critical chart on\/~$(X\t Y,s\boxplus t)$.
\label{dc2prop5}
\end{prop}

\begin{rem} Note that in Proposition \ref{dc2prop5}, $\pi_X,\pi_Y$
are in general {\it not\/} morphisms of d-critical loci $(X\t
Y,s\boxplus t)\ra(X,s)$, $(X\t Y,s\boxplus t)\ra(Y,t)$ in the sense
of Definition \ref{dc2def1}, since $\pi_X^\star(s)\ne s\boxplus
t\ne\pi_Y^\star(t)$. Also $(X\t Y,s\boxplus t)$ is {\it not\/} a
product $(X,s)\t(Y,t)$ in the category of d-critical loci, in the
sense of category theory. Nonetheless, we will call $(X\t
Y,s\boxplus t)$ the {\it product\/} of~$(X,s),(Y,t)$.
\label{dc2rem3}
\end{rem}

Let $U$ be a complex manifold, $f:U\ra\C$ be holomorphic, and
$X=\Crit(f)$, as a complex analytic space. Then $f\vert_X:X\ra\C$ is
holomorphic, and $\d(f\vert_X)=0$ in $H^0(T^*X)$. Experience with
calculus on manifolds suggests that if $g:X\ra\C$ is holomorphic
with $\d g=0$ in $H^0(T^*X)$ then $g$ is locally constant on $X$.
However, this is true only for {\it reduced\/} complex analytic
spaces or $\K$-schemes $X$. Here is an example of a non-reduced
critical locus with $f\vert_X$ not locally constant:

\begin{ex} Define $f:\C^2\ra\C$ by $f(x,y)=x^5+x^2y^2+y^5$, and let
$X=\Crit(f)$, as a complex analytic space. Then $f\vert_X\in
H^0(\O_X)$. We have $\d(f\vert_X)=0\in H^0(T^*X)$, since $X=\d
f^{-1}(0)$. Suppose for a contradiction that $f$ is constant on $X$
near $(0,0)$. Then we may write
\begin{equation*}
f(x,y)=\ts A+\frac{\pd f}{\pd x}\,B(x,y)+\frac{\pd f}{\pd y}\,C(x,y)
\end{equation*}
on $\C^2$ near $(0,0)$, for some holomorphic functions $B,C$ defined
near $(0,0)$ in $\C^2$. That is, we have
\begin{align*}
x^5+x^2y^2+y^5=\ts A+(5x^4+2xy^2)\sum\limits_{i,j\ge 0}B_{i,j}x^iy^j
+(2x^2y+5y^4)\sum\limits_{i,j\ge 0}C_{i,j}x^iy^j.
\end{align*}
Comparing coefficients of $x^5,y^5,x^2y^2$ give the equations
\begin{equation*}
1=5B_{1,0},\quad 1=5C_{0,1},\quad 1=2B_{1,0}+2C_{0,1},
\end{equation*}
which have no solution. Thus in this case, $f\vert_X$ is not locally
constant on $X$.
\label{dc2ex2}
\end{ex}

If $X=\Crit(f)$ and $X^\red$ is the reduced complex analytic
subspace of $X$, then $f\vert_{X^\red}$ is always locally constant.
This is why we defined $\cSz_X$ using restriction to $X^\red$ in
Theorem \ref{dc2thm1}(a). Combining Theorem \ref{dc2thm1}(a),(b) we
deduce:

\begin{cor} Suppose $X$ is a complex analytic space, and the
following sequence of sheaves of\/ $\C$-vector spaces on $X$ is
exact:
\e
\xymatrix@C=30pt{ 0 \ar[r] & \C_X \ar[r]^{\inc} & \O_X \ar[r]^\d &
T^*X, }
\label{dc2eq13}
\e
where $\inc:\C_X\hookra\O_X$ is the inclusion of the constant
functions into the holomorphic functions. Then $\cSz_X\cong
h^{-1}(\bL_X),$ so that\/ $\cSz_X$ is a coherent sheaf on $X,$ and\/
$\cS_X\cong\C_X\op h^{-1}(\bL_X)$. The analogue also holds for
$\K$-schemes.
\label{dc2cor1}
\end{cor}

Now \eq{dc2eq13} is exact if for $g:X\ra\C$ a locally defined
holomorphic function, $\d g=0$ implies $g$ is locally constant. In
Example \ref{dc2ex2} this fails, so in this example \eq{dc2eq13} is
not exact, and $\cSz_X\not\cong h^{-1}(\bL_X)$, and $\cSz_X$ is not
a coherent sheaf on $X$. Next we consider smooth complex analytic
spaces and $\K$-schemes:

\begin{ex} Suppose $X$ is a complex manifold, considered as a
complex analytic space. Then in Theorem \ref{dc2thm1} we see that
$\cS_X\cong\C_X$, the constant sheaf, and $\cSz_X=0$, the zero
sheaf. To see this, take $R=X=U$ and $i=\id_X:X\ra X$ in Theorem
\ref{dc2thm1}. Then $I_{X,X}=0$ by \eq{dc2eq1}, and
$\cS_X\cong\Ker\bigl(\d:\O_X\ra T^*X\bigr)\cong\C_X$ by \eq{dc2eq2}.
As $\cSz_X=0$ there is a unique global section $s=0\in H^0(\cSz_X)$,
and $(X,0)$ is a complex analytic d-critical locus, as it
is~$\Crit(0:X\ra\C)$.

Similarly, if $X$ is a smooth $\K$-scheme then $\cS_X\cong\K_X$, and
$\cSz_X=0$, and $(X,0)$ is an algebraic d-critical locus.
\label{dc2ex3}
\end{ex}

Our next two examples compare algebraic d-critical loci with
symmetric obstruction theories on $\K$-schemes, as defined in
Definition~\ref{dc1def1}.

\begin{ex} Let $\K$ be a field of characteristic zero, and define
$X$ to be the $\K$-scheme $X=\Spec\bigl(\K[z]/(z^n)\bigr)$ for $n\ge
2$. Then $X$ has an obvious embedding $i:X\hookra
\bA^1=\Spec\bigl(\K[z]\bigr)$ as the subscheme $z^n=0$ in $\bA^1$.
It is a non-reduced point. Using this embedding $X\hookra\bA^1$,
from Theorem \ref{dc2thm1}(i) we find that
\begin{align*}
H^0(\cS_X)=\bigl\{a_0&+a_{n+1}z^{n+1}+\cdots+
a_{2n-1}z^{2n-1}+(z^{2n}):\\
&a_0,a_{n+1},\ldots,a_{2n-1}\in\K\bigr\}\cong\K^n,
\end{align*}
and $H^0(\cSz_X)\subset H^0(\cS_X)$ is the subspace with $a_0=0$,
isomorphic to~$\K^{n-1}$.

Now let $0\in U\subseteq\bA^1$ be open, and suppose $f:U\ra\bA^1$ is
regular with $f(0)=0$ and $\Crit(f)=i(X)$. Write
$a_k=\frac{1}{k!}\frac{\pd^kf}{\pd z^k}(0)$ for $k=0,1,\ldots.$ Then
$f(0)=0$ gives $a_0=0$, and $\Crit(f)=X$ is equivalent to
$a_1=\cdots=a_n=0$ and $a_{n+1}\ne 0$. The section $s\in
H^0(\cSz_X)$ corresponding to $f$ is $f+(z^{2n})=a_{n+1}z^{n+1}+
\cdots+ a_{2n-1}z^{2n-1}+(z^{2n})$. From this we see that if
$s=a_{n+1}z^{n+1}+\cdots+ a_{2n-1}z^{2n-1}+(z^{2n})\in H^0(\cSz_X)$,
then $(X,s)$ is an algebraic d-critical locus if and only if
$a_{n+1}\ne 0$.

By equation \eq{dc1eq1}, the natural symmetric obstruction theory on
$X=\Crit(f)$ is determined by $\frac{\pd^2f}{\pd z^2}\big\vert_X$,
that is, by $(n+1)na_{n+1}z^{n-1}+\cdots$ modulo $(z^n)$. Thus, in
this example, the d-critical locus $(X,s)$ associated to $\Crit(f)$
records the first $n-1$ coefficients $a_{n+1},a_{n+2},\ldots,
a_{2n-1}$ in the power series expansion of
$f(z)=a_{n+1}z^{n+1}+a_{n+2}z^{n+2}+\cdots$ at 0, but the symmetric
obstruction theory $\phi:\cE^\bu\ra\bL_X$, $\th$ records only the
first coefficient $a_{n+1}$. Hence (at least in this case), {\it the
algebraic d-critical locus remembers more information, locally, than
the symmetric obstruction theory}.
\label{dc2ex4}
\end{ex}

\begin{ex} Let $t:U\ra\bA^1$ be a smooth morphism of
$\K$-schemes of relative dimension 2, whose fibres $U_t$ for $t$ at
and near 0 in $\bA^1$ are $K3$ surfaces. Set $X=U_0\subset U$, and
regard $U,t$ as a 1-parameter family of deformations of $X$.

We wish to compare $\Crit(t^2:U\ra\bA^1)$ and $\Crit(0:X\ra\bA^1)$
in the categories (or higher categories) of:
\begin{itemize}
\setlength{\itemsep}{0pt}
\setlength{\parsep}{0pt}
\item[(i)] classical $\K$-schemes;
\item[(ii)] algebraic d-critical loci;
\item[(iii)] schemes with perfect obstruction theories
\cite{BeFa};
\item[(iv)] schemes with symmetric obstruction theories
\cite{Behr}; and
\item[(v)] $-1$-shifted symplectic derived schemes \cite{PTVV}.
\end{itemize}
We will see that the two are isomorphic in (i),(ii), but not
equivalent in (iii)--(v). For (i), $\Crit(t^2:U\ra\bA^1)$ and
$\Crit(0:X\ra\bA^1)$ are both $X$ as classical schemes. For (ii), as
$X$ is smooth $\cSz_X=0$, so $\Crit(t^2:U\ra\bA^1)$ and
$\Crit(0:X\ra\bA^1)$ are both $(X,0)$ as algebraic d-critical loci.

For (iii), write $\phi:\cE^\bu\ra\bL_X$ and $\psi:\cF^\bu\ra\bL_X$
for the obstruction theories from $\Crit(t^2:U\ra\bA^1)$ and
$\Crit(0:X\ra\bA^1)$. Then \eq{dc1eq1} gives
\begin{equation*}
\smash{\xymatrix@C=16pt{\cE^\bu=\bigl[TU\vert_X
\ar[rr]^(0.52){\pd^2(t^2)\vert_X} && T^*U\vert_X\bigr], &
\cF^\bu=\bigl[TX \ar[rr]^(0.52){0} && T^*X\vert_X\bigr]. }}
\end{equation*}
We want to know whether $\cE^\bu\cong\cF^\bu$ in $D(\qcoh(X))$. Now
$\tau_{\le -1}(\cE^\bu)\cong TX[1]$ and $\tau_{\ge 0}(\cE^\bu)\cong
T^*X$, so we have a distinguished triangle in $D(\qcoh(X))$:
\begin{equation*}
\xymatrix{ \cdots \ar[r] & T^*X[-1] \ar[r]^\al  & TX[1] \ar[r]
& \cE^\bu \ar[r] & T^*X \ar[r] & \cdots. }
\end{equation*}
That is, $\cE^\bu$ is the cone on $\al:T^*X[-1]\ra TX[1]$ in
$D(\qcoh(X))$ for some $\al$ in $\Ext^2(T^*X,TX)$. Hence
$\cE^\bu\cong\cF^\bu$ if and only if~$\al=0$.

The normal bundle $\nu$ of $X$ in $U$, and its dual $\nu^*$, are
both isomorphic to $\O_X$ as $t:U\ra\bA^1$ induces isomorphisms
$\nu\cong t^*(T_0\bA^1)$, $\nu^*\cong t^*(T^*_0\bA^1)$. Hence we
have exact sequences
\begin{equation*}
\xymatrix@C=30pt@R=7pt{ 0 \ar[r] & TX \ar[r] & TU\vert_X \ar[r] &
\O_X \ar[r] & 0, \\
0 \ar[r] & \O_X \ar[r] & T^*U\vert_X \ar[r] & T^*X \ar[r] & 0. }
\end{equation*}
Let these correspond to elements $\be'\in\Ext^1(\O_X,TX)$ and
$\be''\in\Ext^1(T^*X,\O_X)$. Then $\al=\be'\ci\be''\in
\Ext^2(T^*X,TX)\cong H^2(TX\ot TX)$.

Under the isomorphisms $\Ext^1(\O_X,TX)\cong H^1(TX)\cong
\Ext^1(T^*X,\O_X)$, we see that $\be',\be''$ are both identified
with $\be\in H^1(TX)$, which parametrizes the infinitesimal
deformation of $\{U_t:t\in\bA^1\}$ at $t=0$, so that informally
$\be\sim \frac{\d}{\d t}U_t\vert_{t=0}$. The projection of $\al$
from $H^2(TX\ot TX)$ to $H^2(\La^2TX)\cong\K$ is~$\be^2$.

Let us choose the deformation $t:U\ra\bA^1$ of $X=U_0$ so that the
infinitesimal deformation $\be\in H^1(TX)$ at $t=0$ satisfies
$\be^2\ne 0$ in $H^2(\La^2TX)\cong\K$, which is possible by well
known facts about $K3$ surfaces. Then $\al\ne 0$ in
$\Ext^2(T^*X,TX)$, and $\cE^\bu\not\cong\cF^\bu$. Hence
$\Crit(t^2:U\ra\bA^1)$ and $\Crit(0:X\ra\bA^1)$ are not equivalent
as schemes with perfect obstruction theories, as in (iii), and so a
fortiori they are also not equivalent as schemes with symmetric
obstruction theories, as in (iv), or as $-1$-shifted symplectic
derived schemes, as in~(v).

Observe that $\al\in\Ext^2(T^*X,TX)$ which distinguishes the
obstruction theories is {\it global\/} information, which is {\it
locally trivial\/}: if $Y\subset X$ is any affine open subset then
$\al\vert_Y=0$ as $\Ext^2(T^*Y,TY)=0$. Thus (at least in this case),
{\it the (symmetric) obstruction theory remembers global, non-local
information which is forgotten by the algebraic d-critical locus}.

This example shows that the dotted arrows `$\dashra$' in Figure
\ref{dc1fig1}, which indicate local constructions, cannot be made to
work globally.
\label{dc2ex5}
\end{ex}

Using related ideas, the author expects that if $(X,s)$ is an
algebraic d-critical locus, then there is an obstruction in
$\Ext^3\bigl(T^*X,(T^*X)^\vee\bigr)$ to finding a symmetric
obstruction theory $\phi:\cE^\bu\ra\bL_X$, $\th$ on $X$ which is
locally modelled on \eq{dc1eq1} when $(X,s)$ is locally modelled on
$\Crit(f:U\ra\C)$. But the author does not know of an example in
which this obstruction is nonzero. Finding such an example would
show that the truncation functor from $-1$-shifted symplectic
derived schemes to algebraic d-critical loci in \cite{BBJ} is not
essentially surjective.

\subsection{Comparing critical charts $(R,U,f,i)$}
\label{dc23}

In \S\ref{dc22} we defined a d-critical locus $(X,s)$ to admit an
open cover by {\it critical charts\/} $(R,U,f,i)$, which write
$(X,s)$ as $\Crit(f)$ in an open set $R\subset X$. We will treat
critical charts like coordinate charts on a manifold. Our analogues
of transition functions between coordinate charts are called {\it
embeddings}.

\begin{dfn} Let $(X,s)$ be a d-critical locus (either complex
analytic or algebraic), and $(R,U,f,i)$ be a critical chart on
$(X,s)$. Let $U'\subseteq U$ be (Zariski) open, and set
$R'=i^{-1}(U')\subseteq R$, so that $R'\subseteq R\subseteq X$ are
(Zariski) open, and $i'=i\vert_{R'}:R'\hookra U'$, and
$f'=f\vert_{U'}:U'\ra\C$ or $\bA^1$. Then $(R',U',f',i')$ is also a
critical chart on $(X,s)$, and we call it a {\it subchart\/} of
$(R,U,f,i)$. As a shorthand we write~$(R',U',f',i')\subseteq
(R,U,f,i)$.

Let $(R,U,f,i)$ and $(S,V,g,j)$ be critical charts on $(X,s)$, with
$R\subseteq S\subseteq X$. An {\it embedding\/} of $(R,U,f,i)$ in
$(S,V,g,j)$ is a locally closed embedding $\Phi:U\hookra V$ of
complex manifolds or $\K$-schemes such that $\Phi\ci i=j\vert_R:R\ra
V$ and $f=g\ci\Phi:U\ra\C$ or $\bA^1$. As a shorthand we write
$\Phi: (R,U,f,i)\hookra(S,V,g,j)$ to mean $\Phi$ is an embedding of
$(R,U,f,i)$ in~$(S,V,g,j)$.

Clearly, if $\Phi:(R,U,f,i)\hookra(S,V,g,j)$ and
$\Psi:(S,V,g,j)\hookra(T,W,h,k)$ are embeddings, then $\Psi\ci\Phi:
(R,U,f,i)\hookra(T,W,h,k)$ is also an embedding.
\label{dc2def3}
\end{dfn}

If $\Phi:(R,U,f,i)\hookra(S,V,g,k)$ is an embedding then $\dim
U\le\dim V$. Thus, embeddings between critical charts
$(R,U,f,i),(S,V,g,j)$ usually go in only one direction, and do not
have inverses which are embeddings. The author drew some inspiration
for these ideas from the theory of Kuranishi spaces in the work of
Fukaya, Oh, Ohta and Ono \cite[\S A]{FOOO} in symplectic geometry:
critical charts $(R,U,f,i)$ are like Kuranishi neighbourhoods on a
topological space, and embeddings are like coordinate changes
between Kuranishi neighbourhoods.

In the algebraic case, it is sometimes convenient to work with
critical charts $(S,V,g,j)$ in which $V\subseteq\bA^n$ is Zariski
open in an affine space $\bA^n$. Every critical chart $(R,U,f,i)$
locally admits embeddings into such a $(S,V,g,j)$. The proof of the
next proposition in \S\ref{dc51} uses the assumption $\mathop{\rm
char}\K\ne 2$ from~\S\ref{dc1}.

\begin{prop} Let\/ $(R,U,f,i)$ be a critical chart on an algebraic
d-critical locus\/ $(X,s)$. Then for each\/ $x\in R$ there exists a
subchart\/ $(R',U',f',i')\subseteq (R,U,f,i)$ with\/ $x\in R'$ and
an embedding $\Phi:(R',U',f',i')\hookra (S,V,g,j)$ into a critical
chart\/ $(S,V,g,j)$ with\/ $V\subseteq\bA^n$ Zariski open for
some\/~$n\ge 0$.
\label{dc2prop6}
\end{prop}

Given two critical charts $(R,U,f,i),(S,V,g,j)$ on $(X,s)$, there
need not exist embeddings between them (or their subcharts) in
either direction. So to compare $(R,U,f,i),(S,V,g,j)$, we construct
embeddings of subcharts $(R',U',f',i'),\ab(S',V',g',j')$ into a
third critical chart $(T,W,h,k)$ with $\dim W\ge \dim U,\dim V$.

\begin{thm} Let\/ $(X,s)$ be a d-critical locus (either complex
analytic or algebraic), and\/ $(R,U,f,i),(S,V,g,j)$ be critical
charts on $(X,s)$. Then for each\/ $x\in R\cap S\subseteq X$ there
exist subcharts $(R',U',f',i')\subseteq(R,U,f,i),$
$(S',V',g',j')\subseteq (S,V,g,j)$ with\/ $x\in R'\cap S'\subseteq
X,$ a critical chart\/ $(T,W,h,k)$ on $(X,s),$ and embeddings
$\Phi:(R',U',f',i')\hookra (T,W,h,k),$
$\Psi:(S',V',g',j')\hookra(T,W,h,k)$.
\label{dc2thm2}
\end{thm}

\begin{rem} To see the point of the definition and theorem, we
explain how they will be used. Often, given a d-critical locus
$(X,s)$, we want to construct some global object $\cG$ on $X$ by
gluing together local data by isomorphisms. Examples include the
canonical bundle $K_{X,s}$ in \S\ref{dc24}, perverse sheaves, $\scr
D$-modules and mixed Hodge modules on oriented d-critical loci
$(X,s)$ in \cite{BBDJS}, and motivic Milnor fibres on oriented
algebraic d-critical loci $(X,s)$ in \cite{BJM}. For each of these
constructions, we use the following method:
\begin{itemize}
\setlength{\itemsep}{0pt}
\setlength{\parsep}{0pt}
\item[(i)] For each critical chart $(R,U,f,i)$ on $(X,s)$, we
define a geometric structure $\cG_{R,U,f,i}$ on $R$, with
$\cG_{R',U',f',i'}=\cG_{R,U,f,i}\vert_{R'}$
for~$(R',U',f',i')\subseteq(R,U,f,i)$.
\item[(ii)] For each embedding $\Phi:(R,U,f,i)\hookra
(S,V,g,j)$, we define a canonical isomorphism $\Phi_*:
\cG_{R,U,f,i}\ra \cG_{S,V,g,j}\vert_R$. We show $\Phi_*$ is
independent of $\Phi$, that is, if $\Phi,\ti\Phi:(R,U,f,i)
\hookra(S,V,g,j)$ are embeddings then~$\Phi_*=\ti\Phi_*$.
\item[(iii)] For embeddings $\Phi\!:\!(R,U,f,i)\!\hookra\!(S,V,g,j)$,
$\Psi\!:\!(S,V,g,j)\!\hookra\!(T,W,h,k)$, we show that
$(\Psi\ci\Phi)_*=\Psi_*\vert_R\ci\Phi_*:\cG_{R,U,f,i}\ra
\cG_{T,W,h,k}\vert_R$.
\item[(iv)] Choose critical charts $\bigl\{(R_a,U_a,f_a,i_a):
a\in A\bigr\}$ with $\{R_a:a\in A\}$ an open cover of $X$. Using
(ii),(iii) and Theorem \ref{dc2thm2} we obtain isomorphisms
$\io_{ab}:\cG_{R_a,U_a,f_a,i_a} \vert_{R_a\cap
R_b}\ra\cG_{R_b,U_b,f_b,i_b}\vert_{R_a\cap R_b}$ for $a,b\in A$,
with $\io_{aa}=\id$ and $\io_{bc}\ci\io_{ab}\vert_{R_a\cap
R_b\cap R_c}=\io_{ac}\vert_{R_a\cap R_b\cap R_c}$ for $a,b,c\in
A$. Thus, provided the geometric structures concerned form a
sheaf, there exists $\cG$ on $X$, unique up to canonical
isomorphism, with $\cG\vert_{R_a}\cong\cG_{R_a,U_a,f_a,i_a}$ for
all~$a\in A$.
\end{itemize}
\label{dc2rem4}
\end{rem}

Our next three results say roughly that if $\Phi:(R,U,f,i)\hookra
(S,V,g,j)$ is an embedding of critical charts on $(X,s)$, then
locally near $j(R)$ in $V$ we have $V\cong U\t\C^n$ or $V\cong
U\t\bA^n$ and $g\cong f\boxplus z_1^2+\cdots+z_n^2$, where $n=\dim
V-\dim U$. But in the algebraic case we have to be careful about
which topology we mean when we say `locally'. The complex analytic
case is straightforward:

\begin{prop} Let\/ $(X,s)$ be a complex analytic d-critical locus,
and\/ $\Phi:(R,U,f,i)\ab\hookra(S,V,g,j)$ an embedding of critical
charts on $(X,s)$. Then for each\/ $x\in R$ there exist open
neighbourhoods\/ $U',V'$ of\/ $i(x),j(x)$ in $U,V$ with\/
$\Phi(U')\subseteq V',$ and holomorphic\/ $\al:V'\ra U,$
$\be:V'\ra\C^n$ for $n=\dim V-\dim U,$ such that\/ $\al\t\be:V'\ra
U\t\C^n$ is a biholomorphism with an open subset of\/ $U\t\C^n,$
and\/ $\al\ci\Phi\vert_{U'}=\id_{U'},$ $\be\ci\Phi\vert_{U'}=0,$
$g\vert_{V'}=f\ci\al+(z_1^2+\cdots+z_n^2)\ci\be$.
\label{dc2prop7}
\end{prop}

For the algebraic case, we give two statements. The first is a
direct analogue of Proposition \ref{dc2prop7} for the \'etale
topology, regarding $\io:U'\ra U$, $\jmath:V'\ra V$ as \'etale
neighbourhoods of $i(x),j(x)$ in $U,V$, and $\Phi':U'\ra V'$ with
$\Phi\ci\io=\jmath\ci\Phi'$ as the analogue of $\Phi(U')\subseteq
V'$. It will be used in \cite{BBDJS}.

Recall our convention in \S\ref{dc1} that the base field $\K$ of $X$
is algebraically closed with $\mathop{\rm char}\K\ne 2$. Both these
assumptions will be used in the proofs of Propositions
\ref{dc2prop8} and \ref{dc2prop9}, since to define $V'$ we need to
take square roots in $\K$, and we need $\mathop{\rm char}\K\ne 2$ to
diagonalize quadratic forms over~$\K$.

\begin{prop} Let\/ $(X,s)$ be an algebraic d-critical locus,
and\/ $\Phi\!:\!(R,U,f,i)\ab\hookra(S,V,g,j)$ an embedding of
critical charts on $(X,s)$. Then for each\/ $x\in R$ there exist
smooth\/ $\K$-schemes $U',V',$ a point\/ $u'\in U',$ and morphisms
$\io:U'\ra U,$ $\jmath:V'\ra V,$ $\Phi':U'\ra V',$ $\al:V'\ra U,$
and\/ $\be:V'\ra\bA^n$ for $n=\dim V-\dim U,$ such that\/
$\io(u')=i(x),$ and\/ $\io:U'\ra U,$ $\jmath:V'\ra V,$
$\al\t\be:V'\ra U\t\bA^n$ are \'etale, and\/
$\Phi\ci\io=\jmath\ci\Phi',$ $\al\ci\Phi'=\io,$ $\be\ci\Phi'=0,$
and\/~$g\ci\jmath=f\ci\al+(z_1^2+\cdots +z_n^2)\ci\be$.

\label{dc2prop8}
\end{prop}

The second holds with $U'$ a Zariski  open neighbourhood of $x$ in
$U$, at the cost of giving a more general form for
$g\ci\jmath:V'\ra\bA^1$. It will be used in~\cite{BJM}.

\begin{prop} Let\/ $(X,s)$ be an algebraic d-critical locus, and\/
$\Phi\!:\!(R,U,f,i)\ab\hookra(S,V,g,j)$ an embedding of critical
charts on $(X,s)$. Then for each\/ $x\in R$ there exist a Zariski
open neighbourhood\/ $U'$ of\/ $i(x)$ in $U,$ a smooth\/ $\K$-scheme
$V',$ and morphisms $\jmath:V'\ra V,$ $\Phi':U'\ra V',$ $\al:V'\ra
U',$ $\be:V'\ra\bA^n$ and\/ $q_1,\ldots,q_n:U'\ra\bA^1\sm\{0\}$ for
$n=\dim V-\dim U,$ such that\/ $\jmath:V'\ra V$ and\/
$\al\t\be:V'\ra U'\t\bA^n$ are \'etale,
$\Phi\vert_{U'}=\jmath\ci\Phi',$ $\al\ci\Phi'=\id_{U'},$
$\be\ci\Phi'=0,$ and\/
\e
g\ci\jmath=f\ci\al+(q_1\ci\al)\cdot(z_1^2\ci\be)+\cdots+
(q_n\ci\al)\cdot(z_n^2\ci\be).
\label{dc2eq14}
\e

\label{dc2prop9}
\end{prop}

\subsection{Canonical bundles of d-critical loci}
\label{dc24}

Propositions \ref{dc2prop7}--\ref{dc2prop9} locally describe
embeddings $\Phi:(R,U,f,i)\hookra(S,V,g,j)$ of critical charts on
$(X,s)$. We can also associate a piece of global data to $\Phi$, a
nondegenerate quadratic form $q_{\sst UV}$ on the pullback
$i^*(N_{\sst UV})$ of the normal bundle $N_{\sst UV}$ of $\Phi(U)$
in~$V$.

\begin{prop} Let\/ $(X,s)$ be a d-critical locus (either complex
analytic or algebraic), and\/ $\Phi:(R,U,f,i)\hookra(S,V,g,j)$ be an
embedding of critical charts on $(X,s)$. Write\/ $N_{\sst UV}$ for
the normal bundle of\/ $\Phi(U)$ in $V,$ regarded as a (holomorphic
or algebraic) vector bundle on $U$ in the exact sequence
\e
\xymatrix@C=20pt{ 0 \ar[r] & TU \ar[rr]^(0.4){\d\Phi} && \Phi^*(TV)
\ar[rr]^(0.55){\Pi_{\sst UV}} && N_{\sst UV} \ar[r] & 0,}
\label{dc2eq15}
\e
so that $i^*(N_{\sst UV})$ is a vector bundle on $R\subseteq X$.
Then there exists a unique $q_{\sst UV}\in H^0(S^2i^*(N_{\sst
UV}^*))$ which is a nondegenerate quadratic form on $i^*(N_{\sst
UV}),$ with the following property in each case:
\begin{itemize}
\setlength{\itemsep}{0pt}
\setlength{\parsep}{0pt}
\item[{\bf(a)}] If\/ $(X,s)$ is a complex analytic d-critical locus
and\/ $x,U',V',n,\al,\be$ are as in Proposition\/
{\rm\ref{dc2prop7},} writing $\langle\d z_1,\ldots,\d
z_n\rangle_{U'}$ for the trivial vector bundle on $U'$ with
basis $\d z_1,\ldots,\d z_n$ and\/ $R'=i^{-1}(U')\subseteq
R\subseteq X,$ there is a natural isomorphism $\hat\be:\langle\d
z_1,\ldots,\d z_n\rangle_{U'}\ra N_{\sst UV}^*\vert_{U'}$ such
that
\begin{gather}
\begin{split}
\Phi\vert_{U'}^*&(\d\be^*)=\Pi_{\sst UV}\vert_{U'}^*\ci\hat\be:\\
&\Phi\vert_{U'}^*\ci\be^*(T_0^*\C^n)=\langle\d z_1,\ldots, \d
z_n\rangle_{U'}\longra \Phi\vert_{U'}^*(T^*V),
\end{split}
\label{dc2eq16}\\
\text{with\/}\quad q_{\sst
UV}\vert_{R'}=i\vert_{R'}^*\bigl[(S^2\hat\be)(\d z_1\ot\d
z_1+\cdots+\d z_n\ot\d z_n)\bigr].
\label{dc2eq17}
\end{gather}

\item[{\bf(b)}] If\/ $(X,s)$ is an algebraic d-critical locus
and\/ $x,U',V',\io,\jmath,\Phi',\al,\be,n$ are as in
Proposition\/ {\rm\ref{dc2prop8},} then there is an isomorphism
$\hat\be:\langle\d z_1,\ldots,\d z_n\rangle_{U'}\ra\io^*(N_{\sst
UV}^*)$ making the following diagram of vector bundles on $U'$
commute:
\e
\begin{gathered}
\xymatrix@C=110pt@R=15pt{
*+[r]{\io^*(N_{\sst UV}^*)} \ar[r]_(0.3){\io^*(\Pi_{\sst UV}^*)}
& *+[l]{\io^*\ci\Phi^*(T^*V)=\Phi^{\prime *}\ci
\jmath^*(T^*V)} \ar[d]_{\Phi^{\prime *}(\d\jmath^*)} \\
*+[r]{\langle\d z_1,\ldots,\d z_n\rangle_{U'}=\Phi^{\prime *}
\ci\be^*(T_0^*\C^n)} \ar[r]^(0.7){\Phi^{\prime *}(\d\be^*)}
\ar@{.>}[u]_{\hat\be} & *+[l]{\Phi^{\prime *}(T^*V'),}}
\end{gathered}
\label{dc2eq18}
\e
and if\/ $R'=R\t_{i,U,\io}U'$ with projections $\rho:R'\ra R,$
$i':R'\ra U',$ then
\e
\rho^*(q_{\sst UV})=i^{\prime *}\bigl[(S^2\hat\be)(\d z_1\ot\d
z_1+\cdots+\d z_n\ot\d z_n)\bigr].
\label{dc2eq19}
\e

\item[{\bf(c)}] If\/ $(X,s)$ is an algebraic d-critical locus
and\/ $x,U',V',\jmath,\Phi',\al,\be,n,q_a$ are as in
Proposition\/ {\rm\ref{dc2prop9},} then there is an isomorphism
$\hat\be:\langle\d z_1,\ldots,\d
z_n\rangle_{U'}\ab\ra\io^*(N_{\sst UV}^*)$ making the following
commute:
\e
\begin{gathered}
\xymatrix@C=110pt@R=15pt{
*+[r]{N_{\sst UV}^*\vert_{U'}} \ar[r]_(0.3){\Pi_{\sst UV}^*\vert_{U'}}
& *+[l]{\Phi\vert_{U'}^*(T^*V)=\Phi^{\prime *}\ci
\jmath^*(T^*V)} \ar[d]_{\Phi^{\prime *}(\d\jmath^*)} \\
*+[r]{\langle\d z_1,\ldots,\d z_n\rangle_{U'}=\Phi^{\prime *}
\ci\be^*(T_0^*\C^n)} \ar[r]^(0.7){\Phi^{\prime *}(\d\be^*)}
\ar@{.>}[u]_{\hat\be} & *+[l]{\Phi^{\prime *}(T^*V'),}}
\end{gathered}
\label{dc2eq20}
\e
and if\/ $R'=i^{-1}(U')\subseteq R\subseteq X,$ then
\e
{}\!\!\! q_{\sst UV}\vert_{R'}=i\vert_{R'}^*\bigl[q_1\cdot
(S^2\hat\be)(\d z_1\ot\d z_1)+\cdots+q_n\cdot (S^2\hat\be)(\d
z_n\ot\d z_n)\bigr].
\label{dc2eq21}
\e
\end{itemize}

Now suppose $\Psi:(S,V,g,j)\hookra(T,W,h,k)$ is another embedding of
critical charts, so that\/ $\Psi\ci\Phi:(R,U,f,i)\hookra(T,W,h,k)$
is also an embedding, and define\/ $N_{\sst VW},q_{\sst VW}$ and\/
$N_{\sst UW},q_{\sst UW}$ using $\Psi$ and\/ $\Psi\ci\Phi$ as above.
Then there are unique morphisms $\ga_{\sst UVW},\de_{\sst UVW}$
which make the following diagram of vector bundles on $U$ commute,
with straight lines exact:
\e
\begin{gathered}
\xymatrix@!0@C=19pt@R=9pt{ &&&&&&&&&&&&&&& 0 \ar[ddl] \\
&&&&&&&&&&&&&&&& 0 \ar[dll] \\
&&&&&& 0 \ar[dddrr] &&&&&&&& TU \ar[dddllllll]_{\d\Phi}
\ar[ddddddllll]^(0.4){\d(\Psi\ci\Phi)} \\ \\ \\
&&&&&&&& \Phi^*(TV) \ar[dddllllll]_{\Pi_{\sst UV}}
\ar[dddrr]^(0.6){\Phi^*(\d\Psi)} \\ \\
0 \ar[drr] \\
&& N_{\sst UV} \ar[dll]  \ar@{.>}[dddrrrrrr]_{\ga_{\sst UVW}}
&&&&&&&& (\Psi\ci\Phi)^*(TW) \ar[ddddddrrrr]^{\Phi^*(\Pi_{\sst VW})}
\ar[dddll]^(0.4){\Pi_{\sst UW}} \\
0 \\ \\
&&&&&&&& N_{\sst UW}  \ar[dddll] \ar@{.>}[dddrrrrrr]_{\de_{\sst UVW}} \\ \\ \\
&&&&&& 0  &&&&&&&& \Phi^*(N_{\sst VW}) \ar[drr] \ar[ddr] \\
&&&&&&&&&&&&&&&& 0 \\
&&&&&&&&&&&&&&& 0.}\!\!\!\!\!\!\!\!\!\!\!\!\!\!\!\!\!\!\!\!\!{}
\end{gathered}
\label{dc2eq22}
\e

Pulling back by $i^*$ gives an exact sequence of vector bundles on
$R\subseteq X\!:$
\e
\xymatrix@C=14pt{ 0 \ar[r] & i^*(N_{\sst UV})
\ar[rrr]^(0.48){i^*(\ga_{\sst UVW})} &&& i^*(N_{\sst UW})
\ar[rrr]^(0.45){i^*(\de_{\sst UVW})} &&& j^*(N_{\sst VW})\vert_R
\ar[r] & 0.}
\label{dc2eq23}
\e
Then there is a natural isomorphism of vector bundles on $R$
\e
i^*(N_{\sst UW})\cong i^*(N_{\sst UV})\op j^*(N_{\sst VW})\vert_R,
\label{dc2eq24}
\e
compatible with the exact sequence {\rm\eq{dc2eq23},} which
identifies
\e
\begin{split}
q_{\sst UW}&\cong q_{\sst UV}\op q_{\sst VW}\op 0 \qquad\text{under
the splitting}\\
S^2(i^*(N_{\sst UW}^*))&\!\cong\! S^2(i^*(N_{\sst UV}^*))\!\op
\!S^2(j^*(N_{\sst VW}^*))\vert_R\!\op\! i^*(N_{\sst UV}^*)\!\ot\!
j^*(N_{\sst VW}^*)\vert_R.
\end{split}
\label{dc2eq25}
\e
\label{dc2prop10}
\end{prop}

Using $N_{\sst UV},q_{\sst UV}$ we define an isomorphism of line
bundles $J_\Phi$ on $R^\red$:

\begin{dfn} Let $\Phi:(R,U,f,i)\hookra(S,V,g,j)$ be an embedding of
critical charts on a d-critical locus $(X,s)$. Define $N_{\sst
UV},\ab q_{\sst UV}$ as in Proposition \ref{dc2prop10}, and set
$n=\dim V-\dim U$. Write $R^\red$ for the reduced complex analytic
subspace or reduced $\K$-subscheme of $R$. Taking top exterior
powers in the dual of \eq{dc2eq15} and pulling back to $R^\red$
using $i^*$ gives an isomorphism of line bundles
\e
\rho_{\sst UV}:\bigl(i^*(K_U)\ot i^*(\La^nN_{\sst
UV}^*)\bigr)\vert_{R^\red}
\,{\buildrel\cong\over\longra}\,j^*(K_V)\vert_{R^\red}.
\label{dc2eq26}
\e
As $q_{\sst UV}$ is a nondegenerate quadratic form on $i^*(N_{\sst
UV}),$ its determinant $\det(q_{\sst UV})$ is a nonvanishing section
of $i^*(\La^nN_{\sst UV}^*)^{\ot^2}$. Define an isomorphism of line
bundles $J_\Phi:i^*\bigl(K_U^{\ot^2}\bigr) \vert_{R^\red}\ra
j^*\bigl(K_V^{\ot^2}\bigr)\vert_{R^\red}$ on $R^\red$ by the
commutative diagram
\e
\begin{gathered}
\xymatrix@C=180pt@R=20pt{
*+[r]{i^*\bigl(K_U^{\ot^2}\bigr)\vert_{R^\red}}
\ar@/_1pc/[dr]_{J_\Phi}
\ar[r]_(0.37){\id_{i^*(K_U^2)}\ot\det(q_{\sst UV})\vert_{R^\red}} &
*+[l]{\bigl(i^*\bigl(K_U^{\ot^2}\bigr)\!\ot\!i^*\bigl(\La^{\rm top}N_{\sst
UV}^*\bigr){}^{\ot^2}\bigr)\big\vert_{R^\red}}
\ar[d]_(0.45){\rho_{\sst UV}^{\ot^2}\vert_{R^\red}} \\
& *+[l]{j^*\bigl(K_V^{\ot^2}\bigr)\vert_{R^\red}.\!\!{}} }
\end{gathered}
\label{dc2eq27}
\e

\label{dc2def4}
\end{dfn}

Here are some useful properties of the $J_\Phi$. The proof that
$J_\Phi$ is independent of $\Phi$ needs $R^\red$ reduced, which is
why we restricted to $R^\red$ in Definition~\ref{dc2def4}.

\begin{prop} In Definition\/ {\rm\ref{dc2def4},} the isomorphism
$J_\Phi$ is independent of the choice of\/ $\Phi$. That is, if\/
$\Phi,\ti\Phi:(R,U,f,i)\hookra(S,V,g,j)$ are embeddings of critical
charts then\/~$J_\Phi=J_{\smash{\ti\Phi}}$.

If\/ $\Psi:(S,V,g,j)\hookra(T,W,h,k)$ is another embedding of
critical charts then
\e
J_\Psi\vert_{R^\red}\ci J_\Phi=J_{\Psi\ci\Phi}.
\label{dc2eq28}
\e

\label{dc2prop11}
\end{prop}

We can now define the {\it canonical bundle\/} of a d-critical
locus.

\begin{thm} Let\/ $(X,s)$ be a d-critical locus (either complex
analytic or algebraic), and\/ $X^\red\subseteq X$ the associated
reduced complex analytic space or reduced\/ $\K$-scheme. Then there
exists a (holomorphic or algebraic) line bundle $K_{X,s}$ on
$X^\red$ which we call the \begin{bfseries}canonical
bundle\end{bfseries} of\/ $(X,s),$ which is natural up to canonical
isomorphism, and is characterized by the following properties:
\begin{itemize}
\setlength{\itemsep}{0pt}
\setlength{\parsep}{0pt}
\item[{\bf(i)}] If\/ $(R,U,f,i)$ is a critical chart on
$(X,s),$ there is a natural isomorphism
\e
\io_{R,U,f,i}:K_{X,s}\vert_{R^\red}\longra
i^*\bigl(K_U^{\ot^2}\bigr)\vert_{R^\red},
\label{dc2eq29}
\e
where $K_U=\La^{\dim U}T^*U$ is the canonical bundle of\/ $U$ in
the usual sense.
\item[{\bf(ii)}] Let\/ $\Phi:(R,U,f,i)\hookra(S,V,g,j)$ be an
embedding of critical charts on $(X,s),$ and let\/ $J_\Phi$ be
as in \eq{dc2eq27}. Then
\e
\io_{S,V,g,j}\vert_{R^\red}=J_\Phi\ci
\io_{R,U,f,i}:K_{X,s}\vert_{R^\red}\longra j^*
\bigl(K_V^{\ot^2}\bigr)\big\vert_{R^\red}.
\label{dc2eq30}
\e
\item[{\bf(iii)}] For each $x\in X^\red,$ there is a canonical
isomorphism
\e
\ka_x:K_{X,s}\vert_x\,{\buildrel\cong\over\longra}\, \bigl(\La^{\rm
top}T_x^*X\bigr)^{\ot^2},
\label{dc2eq31}
\e
where $T_xX$ is the Zariski tangent space of\/ $X$ at\/~$x$.
\item[{\bf(iv)}] Suppose\/ $(R,U,f,i)$ is a critical chart on
$(X,s)$ and\/ $x\in R,$ and let\/ $\io_{R,U,f,i},\ab\ka_x$ be as
in\/ {\bf(i)\rm,\bf(iii)}. Then we have an exact sequence
\e
\begin{gathered}
{}\!\!\!\!\xymatrix@C=22pt@R=15pt{ 0 \ar[r] & T_xX
\ar[r]^(0.4){\d i\vert_x} & T_{i(x)}U
\ar[rr]^(0.53){\Hess_{i(x)}f} && T_{i(x)}^*U \ar[r]^{\d
i\vert_x^*}  & T_x^*X  \ar[r] & 0, }
\end{gathered}
\label{dc2eq32}
\e
and the following diagram commutes:
\e
\begin{gathered}
\xymatrix@C=150pt@R=13pt{ *+[r]{K_{X,s}\vert_x}
\ar[dr]_{\io_{R,U,f,i}\vert_x} \ar[r]_(0.55){\ka_x} &
*+[l]{\bigl(\La^{\rm top}T_x^*X\bigr){}^{\ot^2}}
\ar[d]_(0.45){\al_{x,R,U,f,i}} \\
& *+[l]{K_U\vert_{i(x)}^{\ot^2},\!\!\!} }
\end{gathered}
\label{dc2eq33}
\e
where $\al_{x,R,U,f,i}$ is induced by taking top exterior powers
in\/~{\rm\eq{dc2eq32}}.
\end{itemize}
\label{dc2thm3}
\end{thm}

\begin{rem}{\bf(a)} As in Theorem \ref{dc1thm1} proved in
\cite{BBJ}, if $(X,s)$ is the truncation of a $-1$-shifted
symplectic derived scheme $(\bX,\om)$ in the sense of Pantev et al.\
\cite{PTVV}, then $K_{X,s}\cong \det(\bL_\bX)\vert_{X^\red}$. So
$K_{X,s}$ is isomorphic to the canonical bundle of the derived
scheme $\bX$ in this case, which is why we call it a canonical
bundle.

\smallskip

\noindent{\bf(b)} The line bundle $K_{X,s}$ in Theorem \ref{dc2thm3}
is characterized uniquely up to isomorphism either by parts
(i),(ii), or by parts (i),(iii),(iv).
\label{dc2rem5}
\end{rem}

Here is a formula for pullback of canonical bundles under smooth
morphisms of d-critical loci, which will be useful in \S\ref{dc28}.
By saying that $T^*_{X/Y}$ is a {\it vector bundle of mixed rank},
and by the {\it top exterior power\/} $\La^{\rm top}T^*_{X/Y}$, we
mean the following: as $\phi:X\ra Y$ is smooth, there is a
decomposition $\coprod_{n\ge 0}X_n$ with $X_n\subseteq X$ open and
closed, such that $\phi\vert_{X_n}:X_n\ra Y$ is smooth of relative
dimension $n$. Then $T^*_{X/Y}\vert_{X_n}$ is a vector bundle on
$X_n$ of rank $n$, and the line bundle $\La^{\rm top}T^*_{X/Y}$ on
$X$ is defined by $\La^{\rm top}T^*_{X/Y}\vert_{X_n}=
\La^nT^*_{X/Y}\vert_{X_n}$ for each~$n$.

\begin{prop} Suppose $\phi:(X,s)\ra(Y,t)$ is a morphism of
d-critical loci with\/ $\phi:X\ra Y$ smooth, as in Proposition\/
{\rm\ref{dc2prop3}}. The \begin{bfseries}relative cotangent
bundle\end{bfseries} $T^*_{X/Y}$ is a vector bundle of mixed rank on
$X$ in the exact sequence of coherent sheaves on $X\!:$
\e
\xymatrix@C=35pt{0 \ar[r] & \phi^*(T^*Y) \ar[r]^(0.55){\d\phi^*} &
T^*X \ar[r] & T^*_{X/Y} \ar[r] & 0. }
\label{dc2eq34}
\e
There is a natural isomorphism of line bundles on $X^\red\!:$
\e
\Up_\phi:\phi\vert_{X^\red}^* (K_{Y,t})\ot\bigl(\La^{\rm
top}T^*_{X/Y}\bigr)\big\vert_{X^\red}^{\ot^2}
\,{\buildrel\cong\over\longra}\,K_{X,s},
\label{dc2eq35}
\e
such that for each\/ $x\in X^\red$ the following diagram of
isomorphisms commutes:
\e
\begin{gathered}
\xymatrix@C=160pt@R=17pt{ *+[r]{K_{Y,t}
\vert_{\phi(x)}\ot\bigl(\La^{\rm top}T^*_{X/Y}\vert_x\bigr)^{\ot^2}}
\ar[r]_(0.7){\Up_\phi\vert_x} \ar[d]^{\ka_{\phi(x)}\ot\id} &
*+[l]{K_{X,s}\vert_x}
\ar[d]_{\ka_x} \\
*+[r]{\bigl(\La^{\rm top}T_{\phi(x)}^*Y\bigr)^{\ot^2}\ot
\bigl(\La^{\rm top}T^*_{X/Y}\vert_x\bigr)^{\ot^2}}
\ar[r]^(0.7){\up_x^{\ot^2}} & *+[l]{\bigl(\La^{\rm
top}T_x^*X\bigr)^{\ot^2},\!\!{}} }
\end{gathered}
\label{dc2eq36}
\e
where $\ka_x,\ka_{\phi(x)}$ are as in {\rm\eq{dc2eq31},} and\/
$\up_x:\La^{\rm top}T_{\phi(x)}^*Y\ot \La^{\rm top}T^*_{X/Y}
\vert_x\ra\La^{\rm top}T_x^*X$ is obtained by restricting
\eq{dc2eq34} to $x$ and taking top exterior powers.
\label{dc2prop12}
\end{prop}

\subsection{Orientations on d-critical loci}
\label{dc25}

The next two definitions will be important in the sequels
\cite{Buss,BBBJ,BBDJS,BBJ,BJM}. For examples of results on oriented
d-critical loci, see Theorems \ref{dc1thm3} and \ref{dc1thm4} above.

\begin{dfn} Let $(X,s)$ be a d-critical locus (either complex
analytic or algebraic), and $K_{X,s}$ its canonical bundle from
Theorem \ref{dc2thm3}. An {\it orientation\/} on $(X,s)$ is a choice
of square root line bundle $K_{X,s}^{1/2}$ for $K_{X,s}$ on
$X^\red$. That is, an orientation is a (holomorphic or algebraic)
line bundle $L$ on $X^\red$, together with an isomorphism
$L^{\ot^2}=L\ot L\cong K_{X,s}$. A d-critical locus with an
orientation will be called an {\it oriented d-critical locus}.

An {\it oriented critical chart\/} on an oriented d-critical locus
is a critical chart $(R,U,f,i)$ on $(X,s)$ with an isomorphism
$\jmath_{R,U,f,i}:K_{X,s}^{1/2}\vert_{R^\red}\ra i^*(K_U)
\vert_{R^\red}$ satisfying $\jmath_{R,U,f,i}^2=\io_{R,U,f,i}$, for
$\io_{R,U,f,i}$ as in~\eq{dc2eq29}.
\label{dc2def5}
\end{dfn}

\begin{rem} In view of equation \eq{dc2eq31}, one might hope to
define a canonical orientation $K_{X,s}^{1/2}$ for a d-critical
locus $(X,s)$ by $K_{X,s}^{1/2}\big\vert_x=\La^{\rm top}T_x^*X$ for
$x\in X^\red$. However, {\it this does not work}, as the spaces
$\La^{\rm top}T_x^*X$ do not vary continuously with $x\in X^\red$ if
$X$ is not smooth. Example \ref{dc2ex8} shows that d-critical loci
need not admit orientations.
\label{dc2rem6}
\end{rem}

In the situation of Proposition \ref{dc2prop12}, the factor
$(\La^{\rm top}T^*_{X/Y})\vert_{X^\red}^{\ot^2}$ in \eq{dc2eq35} has
a natural square root $(\La^{\rm top}T^*_{X/Y})\vert_{X^\red}$. Thus
we deduce:

\begin{cor} Let\/ $\phi:(X,s)\ra(Y,t)$ be a morphism of
d-critical loci with\/ $\phi:X\ra Y$ smooth. Then each orientation
$K_{Y,t}^{1/2}$ for\/ $(Y,t)$ lifts to a natural orientation
$K_{X,s}^{1/2}=\phi\vert_{X^\red}^*(K_{Y,t}^{1/2})\ot(\La^{\rm
top}T^*_{X/Y}) \vert_{X^\red}$ for~$(X,s)$.
\label{dc2cor2}
\end{cor}

We can express orientations in terms of {\it principal\/
$\Z_2$-bundles}.

\begin{dfn} Let $(X,s)$ be a d-critical locus. For each embedding
of critical charts $\Phi:(R,U,f,i)\hookra(S,V,g,j)$ on $(X,s)$,
define a principal $\Z_2$-bundle $\pi_\Phi:P_\Phi\ra R$ over $R$ to
be the bundle of square roots of the isomorphism $J_\Phi$ in
\eq{dc2eq27}. That is, local sections $s_\al:R\ra P_\Phi$ correspond
to local isomorphisms $\al:i^*(K_U)\vert_{R^\red}\ra
j^*(K_V)\vert_{R^\red}$ with $\al\ot\al=J_\Phi$. Note that
Proposition \ref{dc2prop11} implies that $P_\Phi$ is independent of
the choice of $\Phi$, that is, if $\Phi,\ti\Phi:(R,U,f,i)
\hookra(S,V,g,j)$ are embeddings then~$P_\Phi=P_{\smash{\ti\Phi}}$.

If $\Psi:(S,V,g,j)\hookra(T,W,h,k)$ is another embedding of critical
charts then \eq{dc2eq28} implies that there is a canonical
isomorphism
\begin{equation*}
\Xi_{\Psi,\Phi}:P_{\Psi\ci\Phi}\,{\buildrel\cong\over\longra}\,
P_\Psi\vert_R\ot_{\Z_2}P_\Phi,
\end{equation*}
such that if local isomorphisms $\al:i^*(K_U)\vert_{R^\red}\ra
j^*(K_V)\vert_{R^\red}$, $\be:j^*(K_V)\vert_{R^\red}\ab\ra
k^*(K_W)\vert_{R^\red}$, $\ga:i^*(K_U)\vert_{R^\red}\ra
k^*(K_W)\vert_{R^\red}$ with $\al\ot\al=J_\Phi$,
$\be\ot\be=J_\Psi\vert_{R^\red}$, $\ga\ot\ga=J_{\Psi\ci\Phi}$
correspond to local sections $s_\al:R\ra P_\Phi$, $s_\be:R\ra
P_\Psi\vert_R$, $s_\ga:R\ra P_{\Phi\ci\Phi}$, then
$\Xi_{\Psi,\Phi}(s_\ga)=s_\be\ot_{\Z_2}s_\al$ if and only if
$\ga=\be\ci\al$, where $\ga=\be\ci\al$ is possible by~\eq{dc2eq28}.

Now let $K_{X,s}^{1/2}$ be a choice of orientation on $(X,s)$, as in
Definition \ref{dc2def5}. For each critical chart $(R,U,f,i)$ on
$(X,s)$, define a principal $\Z_2$-bundle
$\pi_{R,U,f,i}:Q_{R,U,f,i}\ra R$ to be the bundle of square roots of
the isomorphism $\io_{R,U,f,i}$ in \eq{dc2eq29}. That is, local
sections $s_\be:R\ra Q_{R,U,f,i}$ correspond to local isomorphisms
$\be:K_{X,s}^{1/2}\vert_{R^\red}\ra i^*(K_U)\vert_{R^\red}$ with
$\be\ot\be=\io_{R,U,f,i}$.

Given an orientation $K_{X,s}^{1/2}$ and an embedding
$\Phi:(R,U,f,i)\hookra(S,V,g,j)$, we have principal $\Z_2$-bundles
$\pi_\Phi:P_\Phi\ra R$, $\pi_{R,U,f,i}:Q_{R,U,f,i}\ra R$ and
$\pi_{S,V,g,j}:Q_{S,V,g,j}\ra S$. Then there is a natural
isomorphism of principal $\Z_2$-bundles
\e
\La_\Phi:Q_{S,V,g,j}\vert_R\,{\buildrel\cong\over\longra}\,
P_\Phi\ot_{\Z_2} Q_{R,U,f,i}
\label{dc2eq37}
\e
on $R$, defined as follows: local isomorphisms
\begin{gather*}
\al:i^*(K_U)\vert_{R^\red}\longra j^*(K_V)\vert_{R^\red},\quad
\be:K_{X,s}^{1/2}\vert_{R^\red}\longra i^*(K_U)\vert_{R^\red},\\
\text{and\/}\qquad \ga:K_{X,s}^{1/2}\vert_{R^\red}\longra
j^*(K_V)\vert_{R^\red}
\end{gather*}
with $\al\ot\al=i\vert_{R^\red}^*(J_\Phi)$, $\be\ot\be=
\io_{R,U,f,i}$ and $\ga\ot\ga=\io_{S,V,g,j}\vert_{R^\red}$
correspond to local sections $s_\al:R\ra i^*(P_\Phi)$, $s_\be:R\ra
Q_{R,U,f,i}$ and $s_\ga:R\ra Q_{S,V,g,j}\vert_R$. Equation
\eq{dc2eq30} shows that $\ga=\al\ci\be$ is a possible solution for
$\ga$, and we define $\La_\Phi$ in \eq{dc2eq37} such that
$\La_\Phi(s_\ga)=s_\al\ot_{\Z_2}s_\be$ if and only if
$\ga=\al\ci\be$. Note that $\La_\Phi$ is independent of the choice
of $\Phi$, as $J_\Phi,P_\Phi$ are.

If $\Psi:(S,V,g,j)\hookra(T,W,h,k)$ is another embedding of critical
charts then it is easy to check that the following diagram commutes:
\e
\begin{gathered}
\xymatrix@C=170pt@R=15pt{ *+[r]{Q_{T,W,h,k}\vert_R}
\ar[r]_{\La_{\Psi\ci\Phi}} \ar[d]^{\La_\Psi\vert_R} &
*+[l]{P_{\Psi\ci\Phi}\ot_{\Z_2}Q_{R,U,f,i}}
\ar[d]_{\Xi_{\Psi,\Phi}\ot\id_{Q_{R,U,f,i}}} \\
*+[r]{\bigl(P_\Psi\ot_{\Z_2}Q_{S,V,g,j}\bigr)\big\vert_R}
\ar[r]^{\id_{P_\Psi\vert_R}\ot\La_\Phi} &
*+[l]{P_\Psi\vert_R\ot_{\Z_2}P_\Phi\ot_{\Z_2} Q_{R,U,f,i}.\!\!{}} }
\end{gathered}
\label{dc2eq38}
\e
\label{dc2def6}
\end{dfn}

\begin{prop} Let\/ $(X,s)$ be a d-critical locus. Then Definition\/
{\rm\ref{dc2def6}} induces an isomorphism between isomorphism
classes of orientations\/ $K_{X,s}^{1/2}$ on $(X,s),$ and
isomorphism classes of the following collections of data:
\begin{itemize}
\setlength{\itemsep}{0pt}
\setlength{\parsep}{0pt}
\item[{\bf(a)}] For each critical chart\/
$(R,U,f,i)$ on $(X,s),$ a choice of principal\/ $\Z_2$-bundle
$\pi_{R,U,f,i}:Q_{R,U,f,i}\ra R$ on $R,$ and
\item[{\bf(b)}] For each embedding of critical charts\/
$\Phi:(R,U,f,i)\hookra(S,V,g,j),$ a choice of isomorphism
$\La_\Phi:Q_{S,V,g,j}\vert_R\ra P_\Phi\ot_{\Z_2} Q_{R,U,f,i}$ as
in {\rm\eq{dc2eq37},}
\end{itemize}
such that\/ \eq{dc2eq38} commutes for all embeddings\/
$\Phi:(R,U,f,i)\hookra(S,V,g,j),$ $\Psi:(S,V,g,j)\hookra(T,W,h,k),$
where $P_\Phi,P_\Psi,P_{\Psi\ci\Phi},\Xi_{\Psi,\Phi}$ are as in the
first part of Definition\/~{\rm\ref{dc2def6}}.
\label{dc2prop13}
\end{prop}

The proof of Proposition \ref{dc2prop13} is straightforward.
Definition \ref{dc2def6} shows how to go from an orientation
$K_{X,s}^{1/2}$ to a collection of data $Q_{R,U,f,i},\La_\Phi$. For
the converse, given a collection of data $Q_{R,U,f,i},\La_\Phi$,
note that each $Q_{R,U,f,i}$ determines a square root $L_{R,U,f,i}$
for $K_{X,s}\vert_{R^\red}$ uniquely up to isomorphism for each
critical chart $(R,U,f,i)$, and for an embedding
$\Phi:(R,U,f,i)\hookra(S,V,g,j)$ the isomorphism $\La_\Phi$
determines an isomorphism $i_\Phi:L_{R,U,f,i}\ra
L_{S,V,g,j}\vert_{R^\red}$, and for $\Phi:(R,U,f,i)\hookra
(S,V,g,j)$, $\Psi:(S,V,g,j)\hookra(T,W,h,k)$, equation \eq{dc2eq38}
commuting implies that $i_{\Psi\ci\Phi}=i_\Psi\vert_{R^\red}\ci
i_\Phi$. By the sheaf property of line bundles, we can then show
there exists $\smash{K_{X,s}^{1/2}}$ unique up to canonical
isomorphism, with isomorphisms $K_{X,s}^{1/2}\vert_{R^\red}\cong
L_{R,U,f,i}$ for all $(R,U,f,i)$, which are compatible with $i_\Phi$
for all $\Phi$. We leave the details to the reader.

\begin{rem} Let $\Phi:(R,U,f,i)\hookra(S,V,g,j)$ be an embedding of
critical charts on a d-critical locus $(X,s)$. Define $N_{\sst UV},
q_{\sst UV}$ as in Proposition \ref{dc2prop10}, and
$\pi_\Phi:P_\Phi\ra R$ as in Definition \ref{dc2def6}. Then an
alternative interpretation of $P_\Phi$ is as {\it the principal\/
$\Z_2$-bundle of orientations of the nondegenerate quadratic form
$q_{\sst UV}$ on the vector bundle $i^*(N_{\sst UV})$ over\/} $R$.

Thus, Proposition \ref{dc2prop13} shows that an orientation
$\smash{K_{X,s}^{1/2}}$ on $(X,s)$ is equivalent to giving principal
$\Z_2$-bundles $Q_{R,U,f,i}\ra R$ for each chart $(R,U,f,i)$ on
$(X,s)$, such that $Q_{R,U,f,i}$ and $Q_{S,V,g,j}\vert_R$ differ by
the principal $\Z_2$-bundle of orientations of $q_{\sst UV}$ for
each embedding $\Phi:(R,U,f,i)\hookra (S,V,g,j)$. This is why we
chose the term {\it orientation\/} for $\smash{K_{X,s}^{1/2}}$. It
is closely relation to the notion of {\it orientation data\/} in
Kontsevich and Soibelman~\cite[\S 5]{KoSo1}.
\label{dc2rem7}
\end{rem}

Here are some examples of canonical bundles and orientations:

\begin{ex} Let $U$ be a complex manifold, $f:U\ra\C$ be holomorphic,
and $(X,s)$ be the complex analytic d-critical locus from Example
\ref{dc2ex1} with $X=\Crit(f)$. Then Theorem \ref{dc2thm3}(i) with
$(R,U,f,i)=(X,U,f,{\rm inc})$ implies that $K_{X,s}\cong
K_U^{\ot^2}\vert_{X^\red}$. Hence $K_{X,s}$ has a natural square
root $K_{X,s}^{1/2}=K_U\vert_{X^\red}$, and $(X,s)$ a natural
orientation. The analogue holds for algebraic critical loci.
\label{dc2ex6}
\end{ex}

\begin{ex} Let $X$ be a complex manifold, so that $(X,0)$ is a
d-critical locus as in Example \ref{dc2ex3}. Then Theorem
\ref{dc2thm3}(i) with $(R,U,f,i)=(X,X,0,\id_X)$ shows that
$K_{X,0}\cong K_X^{\ot^2}$, where $K_X$ is the usual canonical
bundle of $X$. Again, $(X,0)$ has a natural
orientation~$K_{X,0}^{1/2}=K_X$.

As we call $K_{X,0}$ the canonical bundle of $(X,0)$, one might have
expected $K_{X,0}\cong K_X$. The explanation is that as a derived
scheme, $\Crit(0:X\ra\C)$ is not $X$, but the shifted cotangent
bundle $T^*X[1]$, and the degree $-1$ fibres of the projection
$T^*X[1]\ra X$ include an extra factor of $K_X$ in~$K_{X,0}$.
\label{dc2ex7}
\end{ex}

\begin{ex} Let $X$ be the non-reduced projective $\C$-scheme
\begin{equation*}
\bigl\{[x,y,z]\in\CP^2:z^2=2yz=0\bigr\}.
\end{equation*}
The reduced $\C$-subscheme $X^\red\subset X\subset\CP^2$ is the
$\CP^1$ defined by $z=0$, and $X$ has only one non-reduced point
$[1,0,0]$, with $X\sm\{[1,0,0]\}\cong\C$ smooth. The open
neighbourhood $X\sm\{[0,1,0]\}$ of [1,0,0] in $X$ is isomorphic as a
classical $\C$-scheme to $\Crit(yz^2:\C^2\ra\C)$, where $(y,z)$ are
the coordinates on~$\C^2$.

Extend $X$ to an algebraic d-critical locus as follows: on
$X\sm\{[0,1,0]\}$, define $s$ as in Example \ref{dc2ex1} using
$X\sm\{[0,1,0]\}\cong\Crit(yz^2:\C^2\ra\C)$. But $\cSz_X\equiv 0$ on
$X\sm\{[1,0,0]\}$ by Example \ref{dc2ex3} as $X\sm\{[1,0,0]\}$ is
smooth, so $s$ extends uniquely by zero to all of $X$. Since $(X,s)$
is modelled on $\Crit(yz^2:\C^2\ra\C)$ on $X\sm\{[0,1,0]\}$ and on
$\Crit(0:\C\ra\C)$ on $X\sm\{[1,0,0]\}$, it is an algebraic
d-critical locus.

Theorem \ref{dc2thm3} defines a line bundle $K_{X,s}$ on
$X^\red\cong\CP^1$. Calculation shows that
$K_{X,s}\cong\O_{\CP^1}(-5)$. For the smooth algebraic d-critical
locus $(\CP^1,0)$ we have $K_{\CP^1,0}\cong
K_{\CP^1}^{\ot^2}\cong\O_{\CP^1}(-4)$ as in Example \ref{dc2ex7}, so
the effect of the nonreduced point $[1,0,0]$ in $X$ is to modify
$K_{X,s}$ from $\O_{\CP^1}(-4)$ to $\O_{\CP^1}(-5)$. Since $-5$ is
odd, $K_{X,s}$ admits no square root. Thus, $(X,s)$ {\it is an
example of a non-orientable algebraic d-critical locus}. We can also
consider $(X,s)$ as a complex analytic d-critical locus, where again
it is not orientable.
\label{dc2ex8}
\end{ex}

\subsection{Equivariant d-critical loci}
\label{dc26}

We now discuss group actions on algebraic d-critical loci.

\begin{dfn} Let $(X,s)$ be an algebraic d-critical locus over $\K$, and $\mu:G\t X\ra X$ an action of an algebraic $\K$-group $G$ on the $\K$-scheme $X$. We also write the action as $\mu(\ga):X\ra X$ for $\ga\in G$. We say that $(X,s)$ is $G$-{\it invariant\/} if $\mu(\ga)^\star(s)=s$ for all $\ga\in G$, or equivalently, if $\mu^\star(s)=\pi_X^\star(s)$ in $H^0(\cSz_{G\t X})$, where $\pi_X:G\t X\ra X$ is the projection.

Let $\chi:G\ra\bG_m$ be a morphism of algebraic $\K$-groups, that is, a character of $G$, where $\bG_m=\K\sm\{0\}$ is the multiplicative group. We say that $(X,s)$ is $G$-{\it equivariant, with character\/} $\chi,$ if $\mu(\ga)^\star(s)=\chi(\ga)\cdot s$ for all $\ga\in G$, or equivalently, if $\mu^\star(s)=(\chi\ci\pi_G)\cdot(\pi_X^\star(s))$ in $H^0(\cSz_{G\t X})$, where $H^0(\O_G)\ni\chi$ acts on $H^0(\cSz_{G\t X})$ by multiplication, as $G$ is a smooth $\K$-scheme.

Suppose $(X,s)$ is $G$-invariant or $G$-equivariant, with $\chi=1$ in the $G$-invariant case. We call a critical chart $(R,U,f,i)$ on $(X,s)$ with a $G$-action $\rho:G\t U\ra U$ a $G$-{\it equivariant critical chart\/} if $R\subseteq X$ is a $G$-invariant open subscheme, and $i:R\hookra U$, $f:U\ra\bA^1$ are equivariant with respect to the actions $\mu\vert_{G\t R},\rho,\chi$ of $G$ on $R,U,\bA^1$, respectively.

We call a subchart $(R',U',f',i')\subseteq(R,U,f,i)$ a $G$-{\it equivariant subchart\/} if $R'\subseteq R$ and $U'\subseteq U$ are $G$-invariant open subschemes. Then $(R',U',f',i'),\rho'$ is a $G$-equivariant critical chart, where~$\rho'=\rho\vert_{G\t U'}$.

Suppose $(R,U,f,i),\rho$ and $(S,V,g,j),\si$ are $G$-equivariant critical charts on $(X,s)$, and $\Phi:(R,U,f,i)\hookra(S,V,g,j)$ is an embedding. We call $\Phi$ {\it equivariant\/} if $\Phi:U\hookra V$ is equivariant with respect to the actions $\rho,\si$ of $G$ on~$U,V$.
\label{dc2def7}
\end{dfn}

When we have a $G$-equivariant d-critical locus $(X,s)$, we would like to be able to work only with $G$-equivariant critical charts and subcharts (so in particular, we would like $X$ to be covered by such charts) and $G$-equivariant embeddings. However, as Example \ref{dc2ex9} below shows, $X$ may not be covered by $G$-equivariant critical charts without extra assumptions on~$X,G$.

We will restrict to the case when $G$ is a torus, with a `good' action on~$X$:

\begin{dfn} Let $X$ be a $\K$-scheme, $G$ an algebraic $\K$-torus, and $\mu:G\t X\ra X$ an action of $G$ on $X$. We call $\mu$ a {\it good action\/} if $X$ admits a Zariski open cover by $G$-invariant affine open $\K$-subschemes $U\subseteq X$.
\label{dc2def8}
\end{dfn}

Sumihiro \cite[Cor.~2]{Sumi} proves that every torus action on a
{\it normal\/} $\K$-variety is good. Applying this to the reduced
$\K$-subscheme $X^\red$ of a $\K$-scheme $X$, and noting that open
$U\subseteq X$ is affine if and only if $U^\red\subseteq X^\red$ is
affine, yields:

\begin{lem} Suppose $X$ is a $\K$-scheme whose reduced\/
$\K$-subscheme $X^\red$ is normal. Then any action $\mu$ of an
algebraic $\K$-torus $G$ on $X$ is good.
\label{dc2lem}
\end{lem}

A torus-equivariant d-critical locus $(X,s)$ admits an open cover by
equivariant critical charts if and only if the torus action is good:

\begin{prop} Let\/ $(X,s)$ be an algebraic d-critical locus which
is invariant or equivariant under the action $\mu:G\t X\ra X$ of an
algebraic torus $G$.
\smallskip

\noindent{\bf(a)} If\/ $\mu$ is good then for all\/ $x\in X$ there
exists a $G$-equivariant critical chart\/ $(R,U,f,i),\rho$ on
$(X,s)$ with\/ $x\in R,$ and we may take $\dim U=\dim T_xX$.
\smallskip

\noindent{\bf(b)} Conversely, if for all\/ $x\in X$ there exists a
$G$-equivariant critical chart\/ $(R,U,f,i),\rho$ on $(X,s)$ with\/
$x\in R,$ then $\mu$ is good.
\label{dc2prop14}
\end{prop}

We can also prove a torus-equivariant analogue of
Theorem~\ref{dc2thm2}:

\begin{prop} Let\/ $(X,s)$ be an algebraic d-critical locus
equivariant under an algebraic torus $G,$ and\/ $(R,U,f,i),\rho,$
$(S,V,g,j),\si$ be $G$-equivariant critical charts on $(X,s)$. Then
for each\/ $x\in R\cap S$ there exist\/ $G$-equivariant subcharts
$(R',U',f',i')\subseteq(R,U,f,i),$ $(S',V',g',j')\subseteq
(S,V,g,j)$ with\/ $x\in R'\cap S',$ a $G$-equivariant critical
chart\/ $(T,W,h,k),\tau$ on $(X,s),$ and\/ $G$-equivariant
embeddings $\Phi:(R',U',f',i')\hookra (T,W,h,k)$ and\/
$\Psi:(S',V',g',j')\hookra(T,W,h,k)$.
\label{dc2prop15}
\end{prop}

Suppose now that $(X,s)$ is an algebraic d-critical locus invariant
under a good action $\mu$ of an algebraic torus $G$. Write $X^G$ for
the $G$-fixed subscheme of $X$, so that $X^G$ is a closed
$\K$-subscheme of $X$ with inclusion $\io:X^G\hookra X$. Set
$s^G=\io^\star(s)\in H^0(\cSz_{\smash{X^G}})$, for $\io^\star$ as in
Proposition \ref{dc2prop1}. Let $(R,U,f,i),\rho$ be a
$G$-equivariant critical chart on $X$. Write $R^G,U^G$ for the
$G$-fixed subschemes of $R,U$, and $f^G=f\vert_{U^G}$,
$i^G=i\vert_{R^G}$. It is easy to see that $(R^G,U^G,f^G,i^G)$ is a
critical chart on $(X^G,s^G)$. Since we can cover $X$ by such
$(R,U,f,i),\rho$ by Proposition \ref{dc2prop14}, we can cover $X^G$
by such $(R^G,U^G,f^G,i^G)$. This proves:

\begin{cor} Suppose $(X,s)$ is an algebraic d-critical locus
invariant under a good action $\mu$ of an algebraic torus $G$. Write
$X^G$ for the $G$-fixed subscheme of\/ $X,$ with inclusion
$\io:X^G\hookra X,$ and\/ $s^G=\io^\star(s)\in
H^0(\cSz_{\smash{X^G}})$. Then $(X^G,s^G)$ is an algebraic
d-critical locus.
\label{dc2cor3}
\end{cor}

Maulik \cite{Maul} will use the last three results when $G=\bG_m$ to
prove a torus localization formula for the motives $MF_{X,s}$
associated to oriented algebraic d-critical loci $(X,s)$ by Bussi,
Joyce and Meinhardt \cite{BJM}, as in Theorem \ref{dc1thm4}, writing
$\pi_*(MF_{X,s})$ in terms of $\pi_*(MF_{\smash{X^{\bG_m},
s^{\bG_m}}})$ for~$\pi:X,X^{\bG_m}\ra \Spec\K$.

Here is an example of a non-good torus action on a d-critical locus:

\begin{ex} Let $\bA^2$ have coordinates $(x,y)$, and define
$f:\bA^2\ra\bA^1$ by $f(x,y)=x^2y^2$. Write $(X,s)$ for the
corresponding affine d-critical locus with $X=\Crit(f)$. It is the
union of the $x$- and $y$-axes in $\bA^2$, with a non-reduced point
at $(0,0)$. Let $G=\bG_m$ act on $\bA^2$ by $u:(x,y)\mapsto
(ux,u^{-1}y)$. Then $f$ is $\bG_m$-invariant, so $(X,s)$ is also
$\bG_m$-invariant.

Define an \'etale equivalence relation $\sim$ on $X$ by
$(x,0)\sim(0,x^{-1})$ for $0\ne x\in\bA^1$, and let $\ti
X=X/\mathbin{\sim}$ be the quotient $\K$-scheme. As $\sim$ is
$\bG_m$-equivariant and preserves $s$, the $\bG_m$-action and
d-critical structure $s$ on $X$ both descend to $\ti X$, so $(\ti
X,\ti s)$ is a $\bG_m$-invariant d-critical locus.

Now $\ti X$ is a projective scheme (it can be embedded in
$\mathbb{KP}^2$, with reduced subscheme the nodal cubic
$u^2w=v^2w+v^3$ in homogeneous coordinates $[u,v,w]$ on
$\mathbb{KP}^2$), but it is not affine. The $\bG_m$-action on $\ti
X$ has only two orbits, $(0,0)$ and $\ti X\sm\{(0,0)\}$. Thus, the
only $\bG_m$-invariant open neighbourhood of $(0,0)$ in $\ti X$ is
$\ti X$ itself, which is not affine, so the $\bG_m$-action on $\ti
X$ is not good.

Proposition \ref{dc2prop14}(b) shows that there does not exist a
$\bG_m$-equivariant critical chart $(R,U,f,i),\rho$
on $(\ti X,\ti s)$ with $(0,0)\in R$.
\label{dc2ex9}
\end{ex}

\begin{rem} For actions of {\it reductive\/} groups, we can prove
the following weaker analogues of Propositions \ref{dc2prop14} and
\ref{dc2prop15} by similar methods:
\begin{itemize}
\setlength{\itemsep}{0pt}
\setlength{\parsep}{0pt}
\item[{\bf(i)}] Let $(X,s)$ be an algebraic d-critical locus
which is invariant or equivariant under the action $\mu:T\t X\ra
X$ of a reductive algebraic $\K$-group $G$. Suppose $x\in X$ is
a fixed point of $G$, and there exists a $G$-invariant affine
open neighbourhood of $x$ in $X$ (this is automatic if $X^\red$
is normal). Then there exists a $G$-equivariant critical chart
$(R,U,f,i),\rho$ on $(X,s)$ with $x\in R,$ and we may take $\dim
U=\dim T_xX$.

\item[{\bf(ii)}] Let $(X,s)$ be an algebraic d-critical locus
equivariant under a reductive $\K$-group $G,$ and $x\in X$ be a
fixed point of $G$, and $(R,U,f,i),\rho,$ $(S,V,g,j),\si$ be
$G$-equivariant critical charts on $(X,s)$ with $x\in R\cap S$.
Then there exist $G$-equivariant subcharts
$(R',U',f',i')\subseteq(R,U,f,i),$ $(S',V',g',j')\subseteq
(S,V,g,j)$ with $x\in R'\cap S',$ a $G$-equivariant critical
chart $(T,W,h,k),\tau$ on $(X,s),$ and $G$-equivariant
embeddings $\Phi:(R',U',f',i')\hookra (T,W,h,k)$ and
$\Psi:(S',V',g',j')\hookra(T,W,h,k)$.
\end{itemize}
We make no claims about points $x\in X$ not fixed by~$G$.
\label{dc2rem8}
\end{rem}

\subsection[Background material on Artin stacks, and sheaves upon
them]{Background material on sheaves on Artin stacks}
\label{dc27}

Section \ref{dc28} will extend \S\ref{dc21}--\S\ref{dc25} from
$\K$-schemes to Artin $\K$-stacks. As a preliminary, to establish
notation, we discuss Artin stacks and sheaves upon them.

{\it Artin stacks\/} are a class of geometric spaces, generalizing
schemes and algebraic spaces. For a good introduction to Artin
stacks see G\'omez \cite{Gome}, and for a thorough treatment see
Laumon and Moret-Bailly \cite{LaMo}. Artin stacks over a field $\K$
form a 2-{\it category\/} $\Art_\K$, with objects the Artin stacks
$X,Y,\ldots,$ 1-morphisms $f,g:X\ra Y$, and 2-morphisms $\eta:f\Ra
g$, which are all 2-isomorphisms.

There is a natural full and faithful strict (2-)functor
$F_\Sch^\Art:\Sch_\K\ra\Art_\K$ from the category $\Sch_\K$ of
$\K$-schemes (regarded as a 2-category with only identity
2-morphisms) to the 2-category $\Art_\K$ of Artin $\K$-stacks. By a
common abuse of notation, we will identify $\Sch_\K$ with its image
in $\Art_\K$, and consider schemes as special examples of Artin
stacks. By definition, every Artin $\K$-stack $X$ admits a {\it
smooth atlas}, which is a smooth, surjective 1-morphism $t:T\ra X$
in $\Art_\K$, for some $\K$-scheme~$T$.

Let $X$ be a $\K$-scheme, and $K\supseteq\K$ a field containing
$\K$. A $K$-{\it point\/} $x$ of $X$ is a morphism $x:\Spec K\ra X$
in $\Sch_\K$, and a {\it point\/} $x$ of $X$, written $x\in X$, is a
$K$-point for any $K$. Similarly, if $X$ is an Artin $\K$-stack, a
$K$-{\it point\/} of $X$ is a 1-morphism $x:\Spec K\ra X$ in
$\Art_\K$. Two $K$-points $x,x'$ are {\it equivalent}, written
$x\cong x'$, if there is a 2-isomorphism $\th:x\Ra x'$. A {\it
point\/} $x$ of $X$, written $x\in X$, is a $K$-point for any $K$.

If $x$ is a $K$-point in $X$, the {\it isotropy group\/} or {\it
stabilizer group\/} $\Iso_x(X)$ is the group of 2-isomorphisms
$\th:x\Ra x$. It has the structure of an algebraic $K$-group, and we
write $\fIso_x(X)$ for its Lie algebra, a $K$-vector space. The {\it
Zariski cotangent space\/} $T_x^*X$ of $X$ at $x$ is also a
$K$-vector space; we have $h^0(\bL_X)\vert_x\cong T_x^*X$ and
$h^1(\bL_X)\vert_x\cong\fIso_x(X)^*$, where $\bL_X$ is the {\it
cotangent complex\/} of $X$, as in Remark \ref{dc2rem10} below.

Laumon and Moret-Bailly \cite[\S\S 12, 13, 15, 18]{LaMo} develop a
theory of {\it sheaves on Artin stacks}, including {\it
quasi-coherent}, {\it coherent}, and {\it constructible\/} sheaves,
and their derived categories. Unfortunately, Laumon and Moret-Bailly
wrongly assume that 1-morphisms of algebraic stacks induce morphisms
of lisse-\'etale topoi, so parts of their theory concerning
pullbacks, etc., are unsatisfactory. Olsson \cite{Olss} rewrites the
theory, correcting this mistake. Laszlo and Olsson
\cite{LaOl1,LaOl2} study derived categories of constructible
sheaves, and perverse sheaves, on Artin stacks, in more detail. All
of \cite{LaMo,LaOl1,LaOl2,Olss} work with sheaves on Artin stacks in
the {\it lisse-\'etale topology}, which we now define.

\begin{dfn} Recall that a {\it site\/} is a category with a
Grothendieck topology, as in Artin \cite{Arti}. Let $X$ be an Artin
$\K$-stack. Define  the {\it lisse-\'etale site\/} $\Le(X)$ of $X$
as follows. The category of $\Le(X)$ has objects pairs $(T,t)$,
where $T$ is a $\K$-scheme and $t:T\ra X$ a smooth 1-morphism in
$\Art_\K$, and morphisms $(\phi,\eta):(T,t)\ra(U,u)$, for $\phi:T\ra
U$ a morphism in $\Sch_\K$ and $\eta:t\Ra u\ci\phi$ a 2-morphism in
$\Art_\K$. Composition of morphisms
$(T,t)\,{\buildrel(\phi,\eta)\over\longra}
\,(U,u)\,{\buildrel(\psi,\ze)\over\longra}\,(V,v)$~is
\begin{equation*}
(\psi,\ze)\ci(\phi,\eta):=\bigl(\psi\ci\phi,(\ze*\id_{
\phi})\od\eta\bigr).
\end{equation*}
Identity morphisms are $\id_{(T,t)}=(\id_T,\id_t)$. The coverings of
an object $(T,t)$ in the Grothendieck topology on $\Le(X)$ are those
collections of morphisms $\bigl\{(\phi_i,\eta_i):
(T_i,t_i)\ra(T,t)\bigr\}{}_{i\in I}$ for which $\bigl\{\phi_i:T_i\ra
T\bigr\}{}_{i\in I}$ is an open cover of $T$ in the \'etale topology
on $\Sch_\K$.
\label{dc2def9}
\end{dfn}

Definition \ref{dc2def9} differs from Laumon and Moret-Bailly
\cite[Def.~12.1]{LaMo} in taking objects $(T,t)$ with $T$ a
$\K$-scheme rather than an algebraic $\K$-space. But as in
\cite[Lem.~12.1.2(i)]{LaMo}, the two definitions yield the same
notion of sheaf on~$X$.

We can now define {\it sheaves\/} (of sets, or $\K$-vector spaces,
or $\K$-algebras, or \ldots) {\it on\/} $X$ to be sheaves on the
site $\Le(X)$, using the notion of {\it sheaves on a site\/} from
Artin \cite{Arti}. The {\it structure sheaf\/} $\O_X$ is a sheaf of
$\K$-algebras on $\Le(X)$, and by considering sheaves of
$\O_X$-modules on $\Le(X)$ we can define {\it quasi-coherent\/} and
{\it coherent sheaves\/} on $X$, as in \cite[\S 13]{LaMo}
and~\cite[\S 6]{Olss}.

Laumon and Moret-Bailly \cite[Lem.~12.2.1]{LaMo} give an
alternative, explicit description of the categories of sheaves on an
Artin $\K$-stack $X$. In \S\ref{dc28} we will use the category
$\Sh(X)$ in Proposition \ref{dc2prop16} as our definition of sheaves
on~$X$.

\begin{prop}[Laumon and Moret-Bailly \cite{LaMo}] Let\/ $X$ be an
Artin\/ $\K$-stack. The category of sheaves of sets on $X$ in the
lisse-\'etale topology is equivalent to the category $\Sh(X)$
defined as follows:
\smallskip

\noindent{\bf(A)} Objects $\cA$ of\/ $\Sh(X)$ comprise the following
data:
\begin{itemize}
\setlength{\itemsep}{0pt}
\setlength{\parsep}{0pt}
\item[{\bf(a)}] For each\/ $\K$-scheme $T$ and smooth\/ $1$-morphism
$t:T\ra X$ in $\Art_\K,$ we are given a sheaf of sets $\cA(T,t)$
on $T,$ in the \'etale topology.
\item[{\bf(b)}] For each\/ $2$-commutative diagram in
$\Art_\K\!:$
\e
\begin{gathered}
\xymatrix@C=50pt@R=1pt{ & U \ar[ddr]^u \\
\rrtwocell_{}\omit^{}\omit{^{\eta}} && \\
T  \ar[uur]^{\phi} \ar[rr]_t && X, }
\end{gathered}
\label{dc2eq39}
\e
where $T,U$ are schemes and\/ $t: T\ra X,$ $u:U\ra X$ are
smooth\/ $1$-morphisms in $\Art_\K,$ we are given a morphism
$\cA(\phi,\eta):\phi^{-1} (\cA(U,u)) \ra\cA(T,t)$ of \'etale
sheaves of sets on $T$.
\end{itemize}
This data must satisfy the following conditions:
\begin{itemize}
\setlength{\itemsep}{0pt}
\setlength{\parsep}{0pt}
\item[{\bf(i)}] If\/ $\phi:T\ra U$ in {\bf(b)} is \'etale, then
$\cA(\phi,\eta)$ is an isomorphism.
\item[{\bf(ii)}] For each\/ $2$-commutative diagram in $\Art_\K\!:$
\begin{equation*}
\xymatrix@C=70pt@R=1pt{ & V \ar[ddr]^v \\
\rrtwocell_{}\omit^{}\omit{^{\ze}} && \\
U  \ar[uur]^{\psi} \ar[rr]_(0.3)u && X, \\
\urrtwocell_{}\omit^{}\omit{^{\eta}} && \\
T \ar[uu]_{\phi} \ar@/_/[uurr]_t }
\end{equation*}
with $T,U,V$ schemes and\/ $t,u,v$ smooth, we must have
\begin{align*}
\cA\bigl(\psi\ci\phi,(\ze*\id_{\phi})\od\eta\bigr)
&=\cA(\phi,\eta)\ci\phi^{-1}(\cA(\psi,\ze))\quad\text{as
morphisms}\\
(\psi\ci\phi)^{-1}(\cA(V,v))&=\phi^{-1}\ci
\psi^{-1}(\cA(V,v))\longra\cA(T,t).
\end{align*}
\end{itemize}

\noindent{\bf(B)} Morphisms $\al:\cA\ra\cB$ of\/ $\Sh(X)$ comprise a
morphism $\al(T,t):\cA(T,t)\ra\cB(T,t)$ of \'etale sheaves of sets
on a scheme $T$ for all smooth\/ $1$-morphisms $t:T\ra X,$ such that
for each diagram \eq{dc2eq39} in {\bf(b)} the following commutes:
\begin{equation*}
\xymatrix@C=120pt@R=15pt{*+[r]{\phi^{-1}(\cA(U,u))}
\ar[d]^{\phi^{-1}(\al(U,u))} \ar[r]_(0.55){\cA(\phi,\eta)} &
*+[l]{\cA(T,t)} \ar[d]_{\al(T,t)} \\
*+[r]{\phi^{-1}(\cB(U,u))}
\ar[r]^(0.55){\cB(\phi,\eta)} & *+[l]{\cB(T,t).\!{}} }
\end{equation*}

\noindent{\bf(C)} Composition of morphisms $\cA\,{\buildrel\al
\over\longra}\,\cB\,{\buildrel\be\over\longra}\,\cC$ in $\Sh(X)$ is
$(\be\ci\al)(T,t)=\ab\be(T,t)\ab\ci\ab\al(T,t)$. Identity morphisms
$\id_\cA:\cA\ra\cA$ are $\id_\cA(T,t)=\id_{\cA(T,t)}$.
\smallskip

The analogue of all the above also holds for (\'etale) sheaves of\/
$\K$-vector spaces, sheaves of\/ $\K$-algebras, and so on, in place
of (\'etale) sheaves of sets.

Furthermore, the analogue of all the above holds for quasi-coherent
sheaves, (or coherent sheaves, or vector bundles, or line bundles)
on $X,$ where in {\bf(a)} $\cA(T,t)$ becomes a quasi-coherent sheaf
(or coherent sheaf, or vector bundle, or line bundle) on $T,$ in
{\bf(b)} we replace $\phi^{-1}(\cA(U,u))$ by the pullback\/
$\phi^*(\cA(U,u))$ of quasi-coherent sheaves (etc.), and\/
$\cA(\phi,\eta),\ab\al(T,t)$ become morphisms of quasi-coherent
sheaves (etc.) on\/~$T$.

We can also describe \begin{bfseries}global sections\end{bfseries}
of sheaves on Artin $\K$-stacks in the above framework: a global
section $s\in H^0(\cA)$ of\/ $\cA$ in part\/ {\bf(A)} assigns a
global section $s(T,t)\in H^0(\cA(T,t))$ of\/ $\cA(T,t)$ on\/ $T$
for all smooth\/ $t:T\ra X$ from a scheme $T,$ such that\/
$\cA(\phi,\eta)^*(s(U,u))=s(T,t)$ in $H^0(\cA(T,t))$ for all\/
$2$-commutative diagrams \eq{dc2eq39} with\/ $t,u$ smooth.
\label{dc2prop16}
\end{prop}

\begin{rem} As in Laumon and Moret-Bailly \cite[\S 13.1]{LaMo}, if
$T$ is a $\K$-scheme, there is a difference between the categories
$\Sh(T)_{\rm Zar}$ and $\Sh(T)_{\text{\rm\'et}}$ of sheaves of sets
(say) on $T$ in the Zariski and \'etale topologies. There are
adjoint functors $\ep_*:\Sh(T)_{\text{\rm\'et}}\ra\Sh(T)_{\rm Zar}$
and $\ep^{-1}:\Sh(T)_{\rm Zar}\ra\Sh(T)_{\text{\rm\'et}}$, with
$\ep^{-1}$ fully faithful, but in general $\Sh(T)_{\text{\rm\'et}}$
may be larger than $\Sh(T)_{\rm Zar}$. So one should distinguish
between sheaves in the Zariski and the \'etale topologies.

However, as in \cite[p.~120]{LaMo}, the categories of quasi-coherent
sheaves on $T$ (and hence their full subcategories of coherent
sheaves, vector bundles, and line bundles) in the Zariski and
\'etale topologies are equivalent, essentially by definition. So for
quasi-coherent sheaves we need not distinguish between the Zariski
and \'etale topologies, and in the last part of Proposition
\ref{dc2prop16} we can take the $\cA(T,t)$ to be quasi-coherent
sheaves on $T$ in the usual (Zariski) sense.
\label{dc2rem9}
\end{rem}

For Theorem \ref{dc2thm4} we will need the following properties of
cotangent complexes, as in Illusie \cite{Illu1,Illu2} and Laumon and
Moret-Bailly~\cite[\S 17]{LaMo}.

\begin{rem}{\bf(i)} If $f:X\ra Y$ is a 1-morphism of Artin
$\K$-stacks, we have a distinguished triangle in~$D(\qcoh(X))$:
\e
\xymatrix@C=40pt{ f^*(\bL_Y) \ar[r]^(0.55){\bL_f} & \bL_X \ar[r] &
\bL_{X/Y} \ar[r] & f^*(\bL_X)[1], }
\label{dc2eq40}
\e
where $\bL_X,\bL_Y$ are the {\it cotangent complexes\/} of $X,Y$ and
$\bL_{X/Y}$ (also written $\bL^f_{X/Y}$) is the {\it relative
cotangent complex\/} of~$f$.
\smallskip

\noindent{\bf(ii)} If $f,g:X\ra Y$ are 1-morphisms of Artin
$\K$-stacks and $\eta:f\Ra g$ is a 2-morphism, then we have a
commutative diagram
\begin{equation*}
\xymatrix@C=40pt@R=15pt{ f^*(\bL_Y) \ar[d]^{\eta^*(\bL_Y)}
\ar[r]_(0.55){\bL_f} & \bL_X \ar[d]^\id \ar[r] & \bL_{X/Y}^f
\ar[d]^\cong \ar[r] & f^*(\bL_X)[1] \ar[d]^{\eta^*(\bL_Y)} \\
g^*(\bL_Y) \ar[r]^(0.55){\bL_g} & \bL_X \ar[r] & \bL_{X/Y}^g
\ar[r] & g^*(\bL_X)[1]. }
\end{equation*}

\noindent{\bf(iii)} Let $X\,{\buildrel f\over\longra}\,Y\,
{\buildrel g\over\longra}\,Z$ be 1-morphisms of Artin $\K$-stacks.
Then there is a distinguished triangle in~$D(\qcoh(X))$:
\begin{equation*}
\xymatrix@C=40pt{ f^*(\bL_{Y/Z}) \ar[r] & \bL_{X/Z} \ar[r] &
\bL_{X/Y} \ar[r] & f^*(\bL_{Y/Z})[1]. }
\end{equation*}

\noindent{\bf(iv)} If $f:X\ra Y$ is smooth then $\bL_{X/Y}$ is
equivalent to a {\it vector bundle of mixed rank\/} on $X$ in degree
0, in the sense of Proposition \ref{dc2prop12}. In this case we
write $T^*_{X/Y}$ or $T^{f\,*}_{X/Y}$ for $\bL_{X/Y}$ considered as
a vector bundle of mixed rank. The top exterior power $\La^{\rm
top}\bL_{X/Y}=\La^{\rm top}T^*_{X/Y}$ is a line bundle on~$X$.
\smallskip

\noindent{\bf(v)} If $\phi:T\ra U$ is a smooth morphism in $\Sch_\K$
then the relative cotangent bundle $T^*_{T/U}$ in Proposition
\ref{dc2prop12} is canonically isomorphic to~$\bL_{T/U}$.
\label{dc2rem10}
\end{rem}

\subsection{Extension of \S\ref{dc21}--\S\ref{dc25} to Artin
stacks}
\label{dc28}

We now extend parts of \S\ref{dc21}--\S\ref{dc25} from $\K$-schemes
to Artin $\K$-stacks. In \cite{BBBJ} we will use the ideas of this
section to extend the results of \cite{BBDJS,BBJ,BJM} (summarized in
Theorems \ref{dc1thm1}, \ref{dc1thm3} and \ref{dc1thm4} above) to
Artin stacks. Note that by the same methods we can also extend
\S\ref{dc21}--\S\ref{dc25} to {\it Deligne--Mumford\/ $\K$-stacks\/}
or {\it algebraic\/ $\K$-spaces}, and the proofs simplify as the
\'etale topology is easier to work with than the lisse-\'etale
topology. We leave the details to the interested reader.

Combining Theorem \ref{dc2thm1}, Proposition \ref{dc2prop1}, and the
material of \S\ref{dc27}, we deduce an analogue of Theorem
\ref{dc2thm1} for Artin $\K$-stacks:

\begin{cor} Let\/ $X$ be an Artin $\K$-stack, and write\/
$\Sh(X)_\Kalg,\Sh(X)_\Kvect$ for the categories of sheaves of\/
$\K$-algebras and\/ $\K$-vector spaces on $X$ defined in
Proposition\/ {\rm\ref{dc2prop16}}. Then:
\begin{itemize}
\setlength{\itemsep}{0pt}
\setlength{\parsep}{0pt}
\item[{\bf(a)}] We may define canonical objects\/ $\cS_X$ in
both\/ $\Sh(X)_\Kalg$ and\/ $\Sh(X)_\Kvect$ by
$\cS_X(T,t):=\cS_T$ for all smooth morphisms $t:T\ra X$ for
$T\in\Sch_\K,$ for $\cS_T$ as in Theorem\/ {\rm\ref{dc2thm1}}
taken to be a sheaf of\/ $\K$-algebras (or $\K$-vector spaces)
on $T$ in the \'etale topology, and\/
$\cS_X(\phi,\eta):=\phi^\star:\phi^{-1}
(\cS_X(U,u))=\phi^{-1}(\cS_U)\ra\cS_T=\cS_X(T,t)$ for all\/
$2$-commutative diagrams \eq{dc2eq39} in $\Art_\K$ with\/ $t,u$
smooth, where $\phi^\star$ is as in Proposition\/
{\rm\ref{dc2prop1}}.
\item[{\bf(b)}] There is a natural decomposition
$\cS_X\!=\!\K_X\!\op\!\cSz_X$ in $\Sh(X)_\Kvect$ induced by the
splitting $\cS_X(T,t)\!=\!\cS_T\!=\!\K_T\op\cSz_T$ in Theorem\/
{\rm\ref{dc2thm1}(a),} where $\K_X$ is a sheaf of\/
$\K$-subalgebras of\/ $\cS_X$ in $\Sh(X)_\Kalg,$ and\/ $\cSz_X$
a sheaf of ideals~in\/~$\cS_X$.
\end{itemize}
\label{dc2cor4}
\end{cor}

Here the conditions (i),(ii) on the data $\cS_X(T,t),
\cS_X(\phi,\eta)$ in Proposition \ref{dc2prop16}(A) follow from the
last part of Proposition \ref{dc2prop1} and equation \eq{dc2eq12}.
We can now generalize algebraic d-critical loci to Artin stacks.

\begin{dfn} A {\it d-critical stack\/} $(X,s)$ is an Artin
$\K$-stack $X$ and a global section $s\in H^0(\cSz_X)$, where
$\cSz_X$ is as in Corollary \ref{dc2cor4} and global sections as in
Proposition \ref{dc2prop16}, such that $\bigl(T,s(T,t)\bigr)$ is an
algebraic d-critical locus in the sense of Definition \ref{dc2def1}
for all smooth morphisms $t:T\ra X$ with~$T\in\Sch_\K$.
\label{dc2def10}
\end{dfn}

The next proposition gives a very convenient way to understand
global sections of $\cS_X,\cSz_X$ and d-critical structures on $X$
by working on the scheme $T$ for an atlas $t:T\ra X$ for $X$.

\begin{prop} Let\/ $X$ be an Artin $\K$-stack, and\/ $t:
T\ra X$ a smooth atlas for $X$. Then $T\t_{t,X,t}T$ is equivalent to
a $\K$-scheme $U$ as $t$ is representable and\/ $T$ a scheme, so we
have a $2$-Cartesian diagram
\e
\begin{gathered}
\xymatrix@C=90pt@R=11pt{ *+[r]{U} \ar[d]^{\pi_1} \ar[r]_(0.3){\pi_2}
\drtwocell_{}\omit^{}\omit{^{\eta}} &
*+[l]{
T} \ar[d]_t \\
*+[r]{T} \ar[r]^(0.7)t & *+[l]{X} }
\end{gathered}
\label{dc2eq41}
\e
in $\Art_\K,$ with\/ $\pi_1,\pi_2:U\ra T$ smooth morphisms in
$\Sch_\K$. Also $T,U,\pi_1,\pi_2$ can be naturally completed to a
smooth groupoid in $\Sch_\K,$ and\/ $X$ is equivalent in $\Art_\K$
to the associated groupoid stack\/ $[U\rra T]$.
\begin{itemize}
\setlength{\itemsep}{0pt}
\setlength{\parsep}{0pt}
\item[{\bf(i)}] Let\/ $\cS_X$ be as in Corollary\/
{\rm\ref{dc2cor4},} and\/ $\cS_T,\cS_U$ be as in Theorem\/
{\rm\ref{dc2thm1},} regarded as sheaves on $T,U$ in the \'etale
topology, and define $\pi_i^\star:\pi_i^{-1}(\cS_T)\ra \cS_U$ as
in Proposition\/ {\rm\ref{dc2prop1}} for $i=1,2$. Consider the
map $t^*:H^0(\cS_X)\ra H^0(\cS_T)$ mapping $t^*:s\mapsto
s(T,t)$. This is injective, and induces a bijection
\e
t^*:H^0(\cS_X)\,{\buildrel\cong\over\longra}\,\bigl\{s'\in
H^0(\cS_T):\text{$\pi_1^\star(s')= \pi_2^\star(s')$ in
$H^0(\cS_U)$}\bigr\}.
\label{dc2eq42}
\e
The analogue holds for $\cSz_X,\cSz_T,\cSz_U$.
\item[{\bf(ii)}] Suppose $s\in H^0(\cSz_X),$ so that\/ $t^*(s)\in
H^0(\cSz_T)$ with\/ $\pi_1^\star\ci t^*(s)=\pi_2^\star\ci
t^*(s)$. Then $(X,s)$ is a d-critical stack if and only if\/
$\bigl(T,t^*(s)\bigr)$ is an algebraic d-critical locus, and
then\/ $\bigl(U,\pi_1^\star\ci t^*(s)\bigr)$ is also an
algebraic d-critical locus.
\end{itemize}
\label{dc2prop17}
\end{prop}

\begin{ex} Suppose an algebraic $\K$-group $G$ acts on a $\K$-scheme
$T$ with action $\mu:G\t T\ra T$, and write $X$ for the quotient
Artin $\K$-stack $[T/G]$. Then as in \eq{dc2eq41} there is a natural
2-Cartesian diagram
\begin{equation*}
\xymatrix@C=110pt@R=11pt{ *+[r]{G\t T} \ar[d]^{\pi_T}
\ar[r]_(0.3){\mu} \drtwocell_{}\omit^{}\omit{^{\eta}} & *+[l]{T}
\ar[d]_t \\ *+[r]{T} \ar[r]^(0.6)t & *+[l]{X=[T/G],\!{}} }
\end{equation*}
where $t:T\ra X$ is a smooth atlas for $X$. If $s'\in H^0(\cSz_T)$
then $\pi_1^\star(s')=\pi_2^\star(s')$ in \eq{dc2eq42} becomes
$\pi_T^\star(s')=\mu^\star(s')$ on $G\t T$, that is, $s'$ is
$G$-invariant in the sense of \S\ref{dc26}. Hence, Proposition \ref{dc2prop17} shows that
d-critical structures $s$ on $X=[T/G]$ are in 1-1 correspondence
with $G$-invariant d-critical structures $s'$ on~$T$.
\label{dc2ex10}
\end{ex}

Next we state an analogue of Theorem \ref{dc2thm3}, constructing the
canonical bundle $K_{X,s}$ of a d-critical stack $(X,s)$.

\begin{thm} Let\/ $(X,s)$ be a d-critical stack. Using the
description of quasi-coherent sheaves on $X^\red$ in Proposition
{\rm\ref{dc2prop16}} and the notation of Remark\/
{\rm\ref{dc2rem10},} there is a line bundle $K_{X,s}$ on the
reduced\/ $\K$-substack\/ $X^\red$ of\/ $X$ called the
\begin{bfseries}canonical bundle\end{bfseries} of\/ $(X,s),$
unique up to canonical isomorphism, such that:
\begin{itemize}
\setlength{\itemsep}{0pt}
\setlength{\parsep}{0pt}
\item[{\bf(a)}] For each point\/ $x\in X^\red\subseteq X$ we
have a canonical isomorphism
\e
\ka_x:K_{X,s}\vert_x\,{\buildrel\cong\over\longra}\,
\bigl(\La^{\rm top}T_x^*X\bigr)^{\ot^2} \ot\bigl(\La^{\rm
top}\fIso_x(X)\bigr)^{\ot^2},
\label{dc2eq43}
\e
where $T_x^*X$ is the Zariski cotangent space of\/ $X$ at\/ $x,$
and\/ $\fIso_x(X)$ the Lie algebra of the isotropy group
(stabilizer group) $\Iso_x(X)$ of\/ $X$ at\/~$x$.
\item[{\bf(b)}] If\/ $T$ is a $\K$-scheme and\/ $t:
T\ra X$ a smooth $1$-morphism, so that\/ $t^\red:T^\red\ra
X^\red$ is also smooth, then there is a natural isomorphism of
line bundles on $T^\red\!:$
\e
\Ga_{T,t}:K_{X,s}(T^\red,t^\red)\,{\buildrel\cong\over\longra}\,
K_{T,s(T,t)}\ot \bigl(\La^{\rm top}T^*_{
T/X}\bigr)\big\vert_{T^\red}^{\ot^{-2}}.
\label{dc2eq44}
\e
Here $\bigl(T,s(T,t)\bigr)$ is an algebraic d-critical locus by
Definition\/ {\rm\ref{dc2def10},} and\/ $K_{T,s(T,t)}\ra T^\red$
is its canonical bundle from Theorem\/~{\rm\ref{dc2thm3}}.
\item[{\bf(c)}] In the situation of\/ {\bf(b)\rm,} let\/ $p\in
T^\red\subseteq T,$ so that\/ $t(p)\in X$. Taking the long exact
cohomology sequence of\/ \eq{dc2eq40} for $t:T\ra X$ and
restricting to $p\in T$ gives an exact sequence
\e
0 \longra T^*_{t(p)}X \longra T^*_pT \longra T^*_{ T/X}\vert_p
\longra \fIso_{t(p)}(X)^* \longra 0.
\label{dc2eq45}
\e
Then the following diagram commutes:
\e
\begin{gathered}
{}\!\!\!\!\!\!\!\!\!\!\!\!\!\!\!\!\!\!\!\!
\xymatrix@!0@C=98pt@R=40pt{*+[r]{K_{X,s}\vert_{t(p)}}
\ar[d]^{\ka_{t(p)}} \ar@{=}[r] & K_{X,s}(T^\red,t^\red)\vert_p
\ar[rr]_(0.3){\Ga_{T,t}\vert_p} && *+[l]{K_{T,s(T,t)}\vert_p\ot
\bigl(\La^{\rm top}T^*_{\smash{
T/X}}\bigr)\big\vert_p^{\ot^{-2}}} \ar[d]_{\ka_p\ot\id} \\
*+[r]{\bigl(\La^{\rm top}T_{t(p)}^*X\bigr)^{\ot^2}\!\!\ot\!\bigl(\La^{\rm
top}\fIso_{t(p)}(X)\bigr)^{\ot^2}} \ar[rrr]^(0.54){\al_p^2} &&&
*+[l]{\bigl(\La^{\rm top}T^*_pT\bigr)^{\ot^2}\!\!\ot\!
\bigl(\La^{\rm top}T^*_{T/X}\bigr) \big\vert_p^{\ot^{-2}},}
}\!\!\!\!\!{}
\end{gathered}
\label{dc2eq46}
\e
where $\ka_p,\ka_{t(p)},\Ga_{T,t}$ are as in {\rm\eq{dc2eq31},
\eq{dc2eq43}} and\/ {\rm\eq{dc2eq44},} respectively, and\/
$\al_p:\La^{\rm top}T_{t(p)}^*X\ot\La^{\rm
top}\fIso_{t(p)}(X)\,{\buildrel\cong\over\longra}\,\La^{\rm
top}T^*_pT\ot\La^{\rm top}T^*_{T/X}\vert^{-1}_p$ is induced by
taking top exterior powers in\/~\eq{dc2eq45}.
\end{itemize}
\label{dc2thm4}
\end{thm}

Here is the analogue of Definition~\ref{dc2def5}:

\begin{dfn} Let $(X,s)$ be a d-critical stack, and $K_{X,s}$
its canonical bundle from Theorem \ref{dc2thm4}. An {\it
orientation\/} on $(X,s)$ is a choice of square root line bundle
$K_{X,s}^{1/2}$ for $K_{X,s}$ on $X^\red$. That is, an orientation
is a line bundle $L$ on $X^\red$, together with an isomorphism
$L^{\ot^2}=L\ot L\cong K_{X,s}$. A d-critical stack with an
orientation will be called an {\it oriented d-critical stack}.
\label{dc2def11}
\end{dfn}

Suppose $(X,s)$ is an oriented d-critical stack. Then for each
smooth $t:T\ra X$ we have a square root $K_{X,s}^{1/2}
(T^\red,t^\red)$ for $K_{X,s}(T^\red,t^\red)$. Thus by \eq{dc2eq44},
$K_{X,s}^{1/2}(T^\red,t^\red)\ot (\La^{\rm top}\bL_{\smash{
T/X}})\vert_{T^\red}$ is a square root for $K_{T,s(T,t)}$. This
proves:

\begin{lem} Let\/ $(X,s)$ be a d-critical stack. Then an
orientation $K_{X,s}^{1/2}$ for $(X,s)$ determines a canonical
orientation\/ $\smash{K_{T,s(T,t)}^{1/2}}$ for the algebraic
d-critical locus $\bigl(T,s(T,t)\bigr),$ for all smooth\/ $t: T\ra
X$ with\/ $T$ a $\K$-scheme.
\label{dc2lem3}
\end{lem}

In \cite{BBBJ} we will prove that an oriented d-critical stack
$(X,s)$ has a natural perverse sheaf $P_{X,s}^\bu$, $\cD$-module
$D_{X,s}$, mixed Hodge module $H_{X,s}^\bu$ (over $\K=\C$) and
motive $MF_{X,s}$, as in Theorems \ref{dc1thm3} and \ref{dc1thm4}
for $\K$-schemes.

\section{The sheaves $\cS_X,\cSz_X$}
\label{dc3}

Sections \ref{dc31}--\ref{dc33} prove Theorem \ref{dc2thm1} from
\S\ref{dc21}, and \S\ref{dc34} proves Proposition~\ref{dc2prop1}.

\subsection{Construction of the sheaf $\cS_X$ in Theorem \ref{dc2thm1}}
\label{dc31}

Let $X$ be a fixed complex analytic space. In this section we
construct the sheaf $\cS_X$ in Theorem \ref{dc2thm1}, satisfying
Theorem \ref{dc2thm1}(i),(ii). We will use the following notation.
Define a {\it triple\/} $(R,U,i)$ to be an open subset $R\subseteq
X$, a complex manifold $U$, and an embedding $i:R\hookra U$ of $R$
as a closed complex analytic subspace of $U$, as in Theorem
\ref{dc2thm1}(i). For such a triple $(R,U,i)$, define the sheaf of
ideals $I_{R,U}\subseteq i^{-1}(\O_U)$ as in \eq{dc2eq1}. We will
also write $I_{R,U}'\subseteq\O_U$ for the sheaf of ideals vanishing
on the closed complex analytic subspace~$i(R)\subseteq U$.

If $(R,U,i)$ is a triple and $U'\subseteq U$ is open, set
$R':=i^{-1}(U')\subseteq R$ and $i':=i\vert_{R'}:R'\hookra U'$. Then
$(R',U',i')$ is another triple, which we call a {\it subtriple\/} of
$(R,U,i)$, and write as~$(R',U',i')\subseteq(R,U,i)$.

For each triple $(R,U,i)$, as in \eq{dc2eq2} define
$\cK_{R,U,i},\ka_{R,U}$ by the exact sequence of sheaves of
$\C$-vector spaces on $R$:
\e
\xymatrix@C=20pt{ 0 \ar[r] & \cK_{R,U,i} \ar[rr]^(0.4){\ka_{R,U}} &&
\displaystyle\frac{i^{-1}(\O_U)}{I_{R,U}^2} \ar[rr]^(0.4)\d &&
\displaystyle\frac{i^{-1}(T^*U)}{I_{R,U}\cdot i^{-1}(T^*U)}\,. }
\label{dc3eq1}
\e
That is, $\ka_{R,U}:\cK_{R,U,i}\ra i^{-1}(\O_U)/I_{R,U}^2$ is the
kernel of $\d:i^{-1}(\O_U)/I_{R,U}^2\ra i^{-1}(T^*U)/I_{R,U}\cdot
i^{-1}(T^*U)$. The difference between \eq{dc2eq2} and \eq{dc3eq1} is
that \eq{dc2eq2} includes an isomorphism
$\cS_X\vert_R\cong\cK_{R,U,i}$, but we have not yet defined $\cS_X$.
If $(R',U',i')\subseteq(R,U,i)$ then $\cK_{R',U',i'}=\cK_{R,U,i}
\vert_{\smash{R'}}$ and~$\ka_{R',U'}=\ka_{R,U}\vert_{\smash{R'}}$.

Note that $i^{-1}(\O_U)/I_{R,U}^2$ in \eq{dc3eq1} is a sheaf of
commutative $\C$-algebras on $R$, since $i^{-1}(\O_U)$ is and
$I_{R,U}^2\subset i^{-1}(\O_U)$ is a sheaf of ideals. Now
$\ka_{R,U}(\cK_{R,U,i})$ is the subsheaf of local sections
$f+I_{R,U}^2$ in $i^{-1}(\O_U)/I_{R,U}^2$ such that $\d f\in
I_{R,U}\cdot i^{-1}(T^*U)$. If $f+I_{R,U}^2,g+I_{R,U}^2$ are two
such sections then $\d(fg)=f\,\d g+g\,\d f\in I_{R,U}\cdot
i^{-1}(T^*U)$, so $(f+I_{R,U}^2)\cdot(g+I_{R,U}^2)
\in\ka_{R,U}(\cK_{R,U,i})$. Also $1+I_{R,U}^2\in\ka_{R,U}
(\cK_{R,U,i})$ as $\d(1)=0$. Hence $\ka_{R,U}(\cK_{R,U,i})$ is a
subsheaf of $\C$-vector spaces in $i^{-1}(\O_U)/I_{R,U}^2$ which is
closed under multiplication and contains the identity, so
$\ka_{R,U}(\cK_{R,U,i})$ is a sheaf of $\C$-subalgebras in
$i^{-1}(\O_U)/I_{R,U}^2$. Therefore $\cK_{R,U,i}$ has the structure
of a sheaf of commutative $\C$-algebras on $R$ in a unique way, such
that $\ka_{R,U}$ in \eq{dc3eq1} is a morphism of sheaves of
$\C$-algebras.

Call $\Phi:(R,U,i)\ra(S,V,j)$ a {\it morphism of triples\/} if
$R\subseteq S\subseteq X$, and $\Phi:U\ra V$ is holomorphic with
$\Phi\ci i=j\vert_{R}:R\ra V$. As for \eq{dc2eq3}, form the
commutative diagram of sheaves of $\C$-vector spaces on $R$:
\e
\begin{gathered}
\xymatrix@C=10pt@R=17pt{ 0 \ar[r] & \cK_{S,V,j}\vert_{R}
\ar@{.>}[d]^{\Phi^*} \ar[rrr]^{\ka_{S,V}\vert_{R}} &&&
\displaystyle\frac{j^{-1}(\O_{V})}{I_{S,V}^2}\Big\vert_{R}
\ar@<-2ex>[d]^{i^{-1}(\Phi^\sh)} \ar[rr]^(0.4)\d &&
\displaystyle\frac{j^{-1}(T^*V)}{I_{S,V}\cdot
j^{-1}(T^*V)}\Big\vert_{R} \ar@<-2ex>[d]^{i^{-1}(\d\Phi)} \\
0 \ar[r] & \cK_{R,U,i} \ar[rrr]^{\ka_{R,U}} &&&
\displaystyle\frac{i^{-1}(\O_{U})}{I_{R,U}^2} \ar[rr]^(0.4)\d &&
\displaystyle\frac{i^{-1}(T^*U)}{I_{R,U}\cdot i^{-1}(T^*U)}\,. }
\end{gathered}
\label{dc3eq2}
\e
Since the right hand square of \eq{dc3eq2} commutes as in
\eq{dc2eq3}, exactness of the rows implies there is a unique
$\Phi^*:\cK_{S,V,j}\vert_{R}\ra\cK_{R,U,i}$ making \eq{dc3eq2}
commute. As $\ka_{R,U},\ka_{S,V}\vert_{R},i^{-1}(\Phi^\sh)$ are
morphisms of sheaves of $\C$-algebras, so is $\Phi^*$.

If $(R',U',i')\subseteq(R,U,i)$ and $(S',V',j')\subseteq(S,V,j)$
with $\Phi(U')\subseteq V'\subseteq V$ then setting
$\Phi':=\Phi\vert_{U'}:U'\ra V'$, we have~$\Phi^{\prime
*}=\Phi^*\vert_{R'}$.

If $\Psi:(S,V,j)\ra(T,W,k)$ is another morphism then so is
$\Psi\ci\Phi:(R,U,i)\ra(T,W,k)$, and by considering the diagram
\begin{equation*}
\xymatrix@C=13pt@R=21pt{ 0 \ar[r] & \cK_{T,W,k}\vert_{R}
\ar[d]^{\Psi^*\vert_{R}}
\ar@<-3.3ex>@/_1pc/[dd]_(0.7){(\Psi\ci\Phi)^*}
\ar[rrr]^{\ka_{T,W}\vert_{R}} &&&
\displaystyle\frac{k^{-1}(\O_{W})}{I_{T,W}^2}\Big\vert_{R}
\ar@<-2ex>[d]^{j^{-1}(\Psi^\sh)\vert_{R}}
\ar@<-4ex>@/_1pc/[dd]_(0.7){i^{-1}((\Psi\ci\Phi)^\sh)}
\ar[rr]^(0.4)\d &&
\displaystyle\frac{k^{-1}(T^*W)}{I_{T,W}\cdot
k^{-1}(T^*W)}\Big\vert_{R}
\ar@<-2ex>[d]^{j^{-1}(\d\Psi)\vert_{R}}
\ar@<-4ex>@/_1pc/[dd]_(0.7){i^{-1}(\d(\Psi\ci\Phi))} \\
0 \ar[r] & \cK_{S,V,j}\vert_{R} \ar[d]^(0.6){\Phi^*}
\ar[rrr]^{\ka_{S,V}\vert_{R}} &&&
\displaystyle\frac{j^{-1}(\O_{V})}{I_{S,V}^2}\Big\vert_{R}
\ar@<-2ex>[d]^(0.6){i^{-1}(\Phi^\sh)} \ar[rr]^(0.4)\d &&
\displaystyle\frac{j^{-1}(T^*V)}{I_{S,V}\cdot
j^{-1}(T^*V)}\Big\vert_{R} \ar@<-2ex>[d]^(0.6){i^{-1}(\d\Phi)} \\
0 \ar[r] & \cK_{R,U,i} \ar[rrr]^{\ka_{R,U}} &&&
\displaystyle\frac{i^{-1}(\O_{U})}{I_{R,U}^2}
\ar[rr]^(0.4)\d &&
\displaystyle\frac{i^{-1}(T^*U)}{I_{R,U}\cdot
i^{-1}(T^*U)}\,, }
\end{equation*}
we see that $(\Psi\ci\Phi)^*=\Phi^*\ci\Psi^*\vert_{R}$, that is, the
morphisms $\Phi^*$ in \eq{dc3eq2} are contravariantly functorial. If
$(R,U,i)=(S,V,j)$ and $\Phi=\id_{U}$ then~$\Phi^*=\id$.

We begin with three lemmas. The first is the main point of the
proof:

\begin{lem} The morphism\/ $\Phi^*$ in \eq{dc3eq2} is independent
of the choice of\/ $\Phi$. That is, if\/ $\Phi,\ti\Phi:(R,U,i)
\ra(S,V,j)$ are morphisms of triples then
\e
\Phi^*=\ti\Phi{}^*:\cK_{S,V,j}\vert_{R}\longra\cK_{R,U,i}.
\label{dc3eq3}
\e
\label{dc3lem1}
\end{lem}

\begin{proof} If $x\in R\subseteq S\subseteq X$ and $\al$ is a
local section of $\cK_{S,V,j}\vert_{R}$ defined near $x$ in $R$,
then $\al=f+(I_{S,V}')^2$ for $f$ a local section of $\O_{V}$
defined near $j(x)$ in $V$ such that $\d f$ is a section of
$I_{S,V}'\cdot T^*V\subseteq T^*V$ near $j(x)$ in $V$. Then
$\Phi^*(\al)=f\ci\Phi+(I_{R,U}')^2$ and $\ti\Phi{}^*(\al)=
f\ci\ti\Phi+(I_{R,U}')^2$ near $x$ in~$R$.

Choose holomorphic coordinates $(z_1,\ldots,z_n)$ on $V$ near
$j(x)=\Phi(i(x))=\ti\Phi(i(x))$. Then by a holomorphic version of
Taylor's Theorem we have
\e
\begin{split}
f\ci\ti\Phi-f\ci\Phi&=\ts\sum_{a=1}^n\bigl(\frac{\pd f}{\pd
z_a}\ci\Phi\bigr)\cdot\bigl(z_a\ci\ti\Phi-z_a\ci\Phi\bigr)\\
&+\ts\sum_{a,b=1}^nA_{ab}\bigl(z_a\ci\ti\Phi-z_a\ci\Phi\bigr)
\bigl(z_b\ci\ti\Phi-z_b\ci\Phi\bigr)
\end{split}
\label{dc3eq4}
\e
near $j(x)$, for some holomorphic $A_{ab}:V\ra\C$ defined
near~$j(x)$.

Since $\Phi\ci i=\ti\Phi\ci i$ and $I_{R,U}'$ is the ideal in
$\O_{U}$ vanishing on $i(R)$ near $i(x)$, we see that
$z_a\ci\ti\Phi-z_a\ci\Phi\in I_{R,U}'$ on $U$ near $i(x)$ for each
$a=1,\ldots,n$. Also $\frac{\pd f}{\pd z_a}\in I_{S,V}'$ near $j(x)$
by choice of $\al,f$, and $g\mapsto g\ci\Phi$ maps $I_{S,V}'\ra
I_{R,U}'$ near $j(x),i(x)$ as $\Phi\ci i=j\vert_{R}$, so $\frac{\pd
f}{\pd z_a}\ci\Phi\in I_{R,U}'$ near $i(x)$. Thus each factor
$(\cdots)$ on the right hand side of \eq{dc3eq4} lies in $I_{R,U}'$
near $i(x)$, so $f\ci\ti\Phi-f\ci\Phi\in (I_{R,U}')^2$ near $i(x)$.
Therefore
\begin{equation*}
\Phi^*(\al)=f\ci\Phi+(I_{R,U}')^2=f\ci\ti\Phi+(I_{R,U}')^2=
\ti\Phi{}^*(\al)
\end{equation*}
near $x$ in $R$ for any local section $\al$ of
$\cK_{S,V,j}\vert_{R}$, which proves~\eq{dc3eq3}.
\end{proof}

\begin{lem} Let\/ $(R,U,i),(S,V,j)$ be triples. Then for
each\/ $x\in R\cap S\subseteq X,$ there exists a subtriple
$(R',U',i')\subseteq(R,U,i)$ with\/ $x\in R'\subseteq R\cap S$ and a
morphism $\Phi:(R',U',i')\ra(S,V,j)$.
\label{dc3lem2}
\end{lem}

\begin{proof} Choose holomorphic coordinates $(z_1,\ldots,z_n)$ on
an open neighbourhood $\ti V$ of $j(x)$ in $V$, so that
$(z_1,\ldots,z_n):\ti V\ra\C^n$ is a biholomorphism with an open set
$\ti W\subseteq\C^n$. Let $U'$ be an open neighbourhood of $i(x)$ in
$U$ small enough that $R':=i^{-1}(U')\subseteq R\cap j^{-1}(\ti
V)\subseteq R\cap S\subseteq X$. Then $z_a\ci j\vert_{R'}$ for
$a=1,\ldots,n$ are morphisms $R'\ra\C$. Since $i(R')$ is a closed
complex analytic subspace of $U'$ isomorphic to $R'$, and any
holomorphic function on $i(R')$ extends locally to $U'$ near $i(x)$,
by making $R',U'$ smaller we can suppose there exist holomorphic
$f_a:U'\ra\C$ with $f_a\ci i\vert_{R'}=z_a\ci j\vert_{R'}$
for~$a=1,\ldots,n$.

Making $R',U'$ smaller again, we can suppose that
$(f_1,\ldots,f_n):U'\ra\C^n$ maps into $\ti V\subseteq\C^n$. Then
there is a unique holomorphic map $\Phi:U'\ra \ti V\subseteq V$ with
$z_a\ci\Phi=f_a$ for $a=1,\ldots,n$. Hence $z_a\ci\Phi\ci
i\vert_{R'}=f_a\ci i\vert_{R'}=z_a\ci j\vert_{R'}$ for
$a=1,\ldots,n$, which implies that $\Phi\ci i\vert_{R'}=
j\vert_{R'}$ as $(z_1,\ldots,z_n):\ti V\ra\C^n$ is injective. Thus
$\Phi:(R',U',i')\ra(S,V,j)$ is a morphism of triples.
\end{proof}

\begin{lem} Let\/ $(R,U,i),(S,V,j)$ be triples. Then
there exists a unique isomorphism of sheaves of commutative\/
$\C$-algebras on $R\cap S\!:$
\e
I_{\sst S,V,j}^{\sst R,U,i}:\cK_{S,V,j}\vert_{R\cap S}\longra
\cK_{R,U,i}\vert_{R\cap S}
\label{dc3eq5}
\e
such that if\/ $(R',U',i'),\Phi$ are as in Lemma\/
{\rm\ref{dc3lem2}} then $I_{\sst S,V,j}^{\sst
R,U,i}\big\vert{}_{R'}=\Phi^*$. Also
\e
I_{\sst R,U,i}^{\sst R,U,i}=\id_{\cK_{R,U,i}},\qquad I_{\sst
R,U,i}^{\sst S,V,j}= \bigl(I_{\sst S,V,j}^{\sst R,U,i}\bigr)^{-1},
\label{dc3eq6}
\e
and if\/ $(T,W,k)$ is another triple then
\e
\begin{gathered}
I_{\sst S,V,j}^{\sst R,U,i}\big\vert_{R\cap S\cap T} \ci I_{\sst
T,W,k}^{\sst S,V,j}\big\vert_{R\cap S\cap T}=
I_{\sst T,W,k}^{\sst R,U,i}\big\vert_{R\cap S\cap T}:\\
\cK_{T,W,k}\vert_{R\cap S\cap T}\longra \cK_{R,U,i}\vert_{R\cap
S\cap T}.
\end{gathered}
\label{dc3eq7}
\e
\label{dc3lem3}
\end{lem}

\begin{proof} Suppose $x,(R',U',i'),\Phi$ and $\hat x,(\hat
R',\hat U',\hat i'),\hat\Phi$ are two possible choices in Lemma
\ref{dc3lem2}. Then $\Phi:U'\ra V$, $\hat\Phi:\hat U'\ra V$ induce
morphisms
\begin{equation*}
\Phi^*:\cK_{S,V,j}\vert_{R'}\longra\cK_{R,U, i}
\vert_{R'}\;\>\text{and}\;\> \hat\Phi{}^*:\cK_{S,V,j}
\vert_{\hat R'}\longra\cK_{R,U,i}\vert_{\hat R'}.
\end{equation*}
The restrictions $\Phi\vert_{\smash{U'\cap \hat U'}}$ and
$\hat\Phi\vert_{\smash{U'\cap \hat U'}}$ thus induce morphisms
\begin{equation*}
\Phi^*\vert_{R'\cap\hat R'},\hat\Phi^*\vert_{R'\cap\hat
R'}:\cK_{S,V,j}\vert_{R'\cap\hat R'}\longra\cK_{R,U,
i}\vert_{R'\cap\hat R'}.
\end{equation*}
Lemma \ref{dc3lem1} now shows that $\Phi^*\vert_{R'\cap\hat
R'}=\hat\Phi^*\vert_{R'\cap\hat R'}$.

Thus, Lemma \ref{dc3lem2} shows that for each $x\in R\cap S$ we can
choose an open neighbourhood $R'$ of $x$ in $R\cap S$, and a
morphism $\Phi^*:\cK_{S,V,j}\vert_{R'}\longra\cK_{R,U,
i}\vert_{R'}$. These open neighbourhoods $R'$ form an open cover of
$X$, and on overlaps $R'\cap\hat R'$ the corresponding morphisms
$\Phi^*,\hat\Phi^*$ agree. Therefore by properties of sheaves there
is a unique morphism $I_{\sst S,V,j}^{\sst R,U,i}$ of sheaves of
commutative $\C$-algebras in \eq{dc3eq5} such that $I_{\sst
S,V,j}^{\sst R,U,i}\big\vert {}_{R'}=\Phi^*$ for all
$(R',U',i'),\Phi$ as in Lemma~\ref{dc3lem2}.

To see that $I_{\sst R,U,i}^{\sst R,U,i}=\id_{\cK_{R,U,i}}$ as in
\eq{dc3eq6}, take $(R',U',i')=(R,U,i)= (S,V,j)$ and $\Phi=\id_{U}$,
so that $\Phi^*=\id$.

To prove \eq{dc3eq7}, let $(R,U,i),(S,V,j), (T,W,k)$ be triples and
$x\in R\cap S\cap T$. Apply Lemma \ref{dc3lem2} twice to get open
$i(x)\in U'\subseteq U$ and $j(x)\in V'\subseteq V$ and morphisms
$\Phi:(R',U',i')\ra(S,V,j)$ and $\Psi:(S',V',j')\ra(T,W,k)$. Making
$U',R'$ smaller we can suppose that $j(x)\in \Phi(U')\subseteq
V'\subseteq V$, so that $\Phi$ is also a morphism $(R',U',i')\ra
(S',V',j')$. Functoriality of the $\Phi^*$ now gives
$\Phi^*\ci\Psi^*\vert_{R'}=(\Psi\ci\Phi)^*$. So the defining
property of the $I_{\sst S,V,j}^{\sst R,U,i}$ gives
\begin{equation*}
I_{\sst S,V,j}^{\sst R,U,i}\big\vert_{R'} \ci
I_{\sst T,W,k}^{\sst S,V,j}\big\vert_{R'}=
I_{\sst T,W,k}^{\sst R,U,i}\big\vert_{R'}.
\end{equation*}
As we can cover $R\cap S\cap T$ by such open $R'$, equation
\eq{dc3eq7} follows.

Finally, applying \eq{dc3eq7} with $(T,W,k)=(R,U,i)$ and using the
first equation of \eq{dc3eq6} yields
\begin{equation*}
I_{\sst S,V,j}^{\sst R,U,i}\ci
I_{\sst R,U,i}^{\sst S,V,j}=
I_{\sst R,U,i}^{\sst R,U,i}\big\vert_{R\cap S}=
\id_{\cK_{R,U,i}}\big\vert_{R\cap S}.
\end{equation*}
Exchanging $(R,U,i),(S,V,j)$ proves the second equation of
\eq{dc3eq6}, and also shows that $I_{\sst S,V,j}^{\sst R,U,i}$ is an
isomorphism. The lemma is complete.
\end{proof}

We can now construct the sheaf $\cS_X$ in Theorem \ref{dc2thm1}.
Since $X$ is locally of finite type by our convention in
\S\ref{dc1}, near each $x\in X$ it admits a local embedding
$i:X\hookra U$ into a complex manifold $U$. Therefore we can choose
a family $\bigl\{(R_a,U_a,i_a):a\in A\bigr\}$ of triples such that
$\{R_a:a\in A\}$ is an open cover of $X$. For each $a\in A$ we have
a sheaf of commutative $\C$-algebras $\cK_{R_a,U_a,i_a}$ on $R_a$
from \eq{dc3eq1}, and for all $a,b\in A$ we have an isomorphism
\begin{equation*}
I_{\sst R_b,U_b,i_b}^{\sst R_a,U_a,i_a}:\cK_{R_b,U_b,i_b}\vert_{R_a\cap
R_b}\longra\cK_{R_a,U_a,i_a}\vert_{R_a\cap R_b}
\end{equation*}
on $R_a\cap R_b$ by \eq{dc3eq5}, which satisfy the usual conditions
\eq{dc3eq6}--\eq{dc3eq7} on identities, symmetry, and triple
overlaps.

Hence by properties of sheaves there exists a sheaf $\cS_X$ of
commutative $\C$-algebras on $X$, unique up to canonical
isomorphism, with isomorphisms $\io_{R_a,U_a}:\cS_X\vert_{R_a}
\ra\cK_{R_a,U_a,i_a}$ for $a\in A$, such that $I_{\sst
R_b,U_b,i_b}^{\sst R_a,U_a,i_a}\ci \io_{R_b,U_b}\vert_{R_a\cap
R_b}=\io_{R_a,U_a}\vert_{R_a\cap R_b}$ for all $a,b\in A$. One way
to do this is to define an explicit presheaf ${\cal P}\cS_X$ on $X$
using $R_a,\cK_{R_a,U_a,i_a},I_{\sst R_b,U_b,i_b}^{\sst
R_a,U_a,i_a}$, and take $\cS_X$ to be the sheafification of~${\cal
P}\cS_X$.

Now suppose $(R,U,i)$ is any triple, as in Theorem \ref{dc2thm1}(i).
We can construct $\cS_X$ using the family $\bigl\{(R_a,U_a,i_a):a\in
A\bigr\}\cup \bigl\{(R,U,i)\bigr\}$ instead of
$\bigl\{(R_a,U_a,i_a):a\in A\bigr\}$, and get the same (not just
isomorphic) sheaf $\cS_X$ and the same isomorphisms
$\io_{R_a,U_a}:\cS_X\vert_{R_a}\ra \cK_{R_a,U_a,i_a}$, but now we
also have an isomorphism $\io_{R,U}:\cS_X\vert_R\ra\cK_{R,U,i}$.
Combining this with \eq{dc3eq1} gives the exact sequence
\eq{dc2eq2}, proving Theorem~\ref{dc2thm1}(i).

Suppose $\Phi:(R,U,i)\ra(S,V,j)$ is any morphism of triples, as in
Theorem \ref{dc2thm1}(ii), so that $R\subseteq S$. We can construct
$\cS_X$ using the family $\bigl\{(R_a,U_a,i_a):a\in A\bigr\}\cup
\bigl\{(R,U,i),(S,V,j)\bigr\}$ instead of $\bigl\{(R_a,U_a,i_a):a\in
A\bigr\}$, and get the same sheaf $\cS_X$ and isomorphisms
$\io_{R_a,U_a}$, but now we also have isomorphisms
$\io_{R,U}:\cS_X\vert_{R}\ra\cK_{R,U,i}$ and
$\io_{S,V}:\cS_X\vert_{S}\ra\cK_{S,V,j}$ satisfying $I_{\sst
S,V,j}^{\sst R,U,i}\ci\io_{S,V} \vert_{R}=\io_{R,U}$. Consider the
diagram:
\e
\begin{gathered}
\xymatrix@C=6pt@R=14pt{ \cS_X\vert_{R} \ar[d]^\id
\ar[rrr]^(0.4){\io_{S,V}\vert_{R}} &&& \cK_{S,V,j}\vert_{R}
\ar@<-3ex>[d]^{I_{\sst S,V,j}^{\sst R,U,i}=\Phi^*}
\ar[rrr]^(0.4){\ka_{S,V}\vert_{R}} &&&
\displaystyle\frac{j^{-1}(\O_{V})}{I_{S,V}^2}\Big\vert_{R}
\ar@<-2ex>[d]^{i^{-1}(\Phi^\sh)} \ar[rr]^(0.4)\d &&
\displaystyle\frac{j^{-1}(T^*V)}{I_{S,V}\cdot
j^{-1}(T^*V)}\Big\vert_{R} \ar@<-5ex>[d]^{i^{-1}(\d\Phi)} \\
\cS_X\vert_{R} \ar[rrr]^(0.4){\io_{R,U}} &&& \cK_{R,U,i}
\ar[rrr]^{\ka_{R,U}} &&&
\displaystyle\frac{i^{-1}(\O_{U})}{I_{R,U}^2} \ar[rr]^(0.4)\d &&
\displaystyle\frac{i^{-1} (T^*U)}{I_{R,U}\cdot i^{-1}(T^*U)}\,.
}\!\!\!\!\!\!\!\!\!\!\!\!\!\!\!{}
\end{gathered}
\label{dc3eq8}
\e
Here $I_{\sst S,V,j}^{\sst R,U,i}=\Phi^*$ by Lemma \ref{dc3lem3}.
The left hand square of \eq{dc3eq8} commutes as $I_{\sst
S,V,j}^{\sst R,U,i}\ci\io_{S,V} \vert_{R}=\io_{R,U}$, and the right
hand two squares commute by \eq{dc3eq2}. Composing the first two
horizontal morphisms in the rows of \eq{dc3eq8} gives \eq{dc2eq3}.
This proves Theorem~\ref{dc2thm1}(ii).

\subsection{Theorem \ref{dc2thm1}(a)--(c): properties of $\cS_X$}
\label{dc32}

Next we prove that the sheaf $\cS_X$ constructed in \S\ref{dc31}
satisfies Theorem \ref{dc2thm1}(a)--(c). We continue to use the
notation of \S\ref{dc31}. For part (a), define $\cSz_X\subset\cS_X$
to be the kernel of the morphism $\cS_X\ra\O_{X^\red}$ in
\eq{dc2eq5}, where $\be_X:\cS_X\ra\O_X$ is defined using
\eq{dc3eq1}--\eq{dc3eq2} and $\O_X\vert_R\cong
i^{-1}(\O_U)/I_{R,U}$. As \eq{dc2eq5} are morphisms of sheaves of
$\C$-algebras, $\cSz_X$ is a sheaf of ideals in~$\cS_X$.

There is also a natural inclusion $\C_X\hookra\cS_X$ of sheaves of
$\C$-algebras, where $\C_X$ is the sheaf of locally constant
functions $X\ra\C$: for any triple $(R,U,i)$, the sheaf $\C_U$ is a
subsheaf of $\O_U$, so $\bigl(i^{-1}(\C_U)+I_{R,U}^2\bigr)
/I_{R,U}^2$ is a subsheaf of $i^{-1}(\O_U)/I_{R,U}^2$ which lies in
the kernel of $\d$ in \eq{dc2eq2}, and so lifts to a subsheaf of
$\cS_X\vert_R$, which is isomorphic to $\C_X\vert_R$, as locally
constant functions on $X$ lift locally uniquely to locally constant
functions on $U$ near~$i(X)$.

It remains to show that $\cS_X=\C_X\op\cSz_X$. To see this, suppose
$(R,U,i)$ is a triple and $s$ is a local section of $\cS_X$ near
$x\in R$, so that $\io_{R,U}(s\vert_R)=f+(I_{R,U}')^2$ for $f$ a
local section of $\O_U$ defined near $i(x)$ in $U$ with $\d f$ a
local section of $I_{R,U}'\cdot T^*U$. Then $f\vert_{i(X^\red)}$ is
locally constant on $i(X^\red)$ near $i(x)$, so $f\vert_{i(X^\red)}$
extends locally uniquely to a locally constant function $c:U\ra\C$
defined near $i(x)$. Writing $f_0=f-c$, we have
$f_0\vert_{i(X^\red)}=0$ near $i(x)$. The local section $s$ now
splits uniquely as $s=s_0+t$ with $\io_{R,U}(s_0)=f_0+(I_{R,U}')^2$
and $\io_{R,U}(t)=c+(I_{R,U}')^2$. But $s_0$ is a local section of
$\cSz_X$ as $f_0\vert_{i(X^\red)}=0$, and $t$ a local section of
$\C_X$ as $c$ is locally constant. Hence $\cS_X=\C_X\op\cSz_X$,
proving~(a).

Part (b) involves the {\it cotangent complex\/} $\bL_X\in
D(\qcoh(X))$. For background on (co)tangent complexes, see Illusie
\cite{Illu1}, \cite[\S 1]{Illu2} for $\K$-schemes, and Palamodov
\cite{Pala1,Pala2,Pala3,Pala4} for complex analytic spaces. We need
only two facts: that $h^0(\bL_X)\cong T^*X$, and if $R\subseteq X$
is open and $i:R\hookra U$ is a closed embedding of $R$ into a
complex manifold $U$, then the truncation $\tau_{\ge -1}(\bL_X)$
satisfies
\e
\tau_{\ge -1}(\bL_X)\vert_R\cong \bigl[\!\xymatrix@C=35pt{
I_{R,U}/I_{R,U}^2 \ar[r]^(0.49)\ga & i^*(T^*U)}\!\bigr]
\label{dc3eq9}
\e
in $D(\qcoh(R))$, where $I_{R,U}/I_{R,U}^2$ is in degree $-1$ and
$i^*(T^*U)$ in degree 0, and the morphism $\ga$ maps
$\ga:f+I_{R,U}^2\mapsto i^*(\d f)$.

Consider the diagram of sheaves on $R$:
\e
\begin{gathered}
\xymatrix@C=9pt@R=11pt{ & 0 \ar[d] && 0 \ar[d] && 0 \ar[d]
\\
0 \ar[r] & h^{-1}(\bL_X)\vert_R \ar@{.>}[d]^{\al_X\vert_R} \ar[rr]
&& I_{R,U}/I_{R,U}^2 \ar[rr]^\ga \ar[d] && i^*(T^*U)
\ar[d]^(0.4)\cong \ar[r] & {\begin{subarray}{l}\ts h^0(\bL_X)\vert_R
\\ \ts \cong T^*X\vert_R \end{subarray}} \ar[r] & 0
\\
0 \ar[r] & \cS_X\vert_R \ar[rr]^(0.4){\io_{R,U}}
\ar@{.>}[d]^{\be_X\vert_R} &&
\displaystyle\frac{i^{-1}(\O_U)}{I_{R,U}^2} \ar[rr]^(0.4)\d
\ar[d]^\pi && \displaystyle\frac{i^{-1}(T^*U)}{I_{R,U}\cdot
i^{-1}(T^*U)} \ar[d]
\\
0 \ar[r] & \O_X\vert_R \ar[rr]_(0.45)\cong^(0.45)\de \ar[d]^\d &&
\displaystyle\frac{i^{-1}(\O_U)}{I_{R,U}} \ar[d] \ar[rr] && 0 \ar[d]
\\
& T^*X\vert_R && 0 && 0.\!\!{} }\!\!\!\!\!\!\!\!\!\!\!\!\!\!\!\!{}
\end{gathered}
\label{dc3eq10}
\e
Here the first row is exact by \eq{dc3eq9}, the second row by
\eq{dc2eq2}, and the third row and second and third columns are
obviously exact. Also the middle two squares commute by definition
of $\ga$ in \eq{dc3eq9}. Properties of exact sequences now imply
there are unique morphisms $\al_X\vert_R,\be_X\vert_R$ as shown
making \eq{dc3eq10} commute. Using functoriality of the isomorphism
\eq{dc3eq9}, one can show that these $\al_X\vert_R,\be_X\vert_R$ are
independent of $R,U,i$ locally, and so glue on an open cover to give
unique global morphisms $\al_X,\be_X$ in \eq{dc2eq6}, noting that
the first column of \eq{dc3eq10} is the restriction of \eq{dc2eq6}
to~$R$.

Since taking kernels of morphisms of sheaves of $\C$-vector spaces
is a left exact functor from the category of morphisms of such
sheaves to the category of such sheaves, equation \eq{dc3eq10} also
implies that the first column of \eq{dc3eq10} is exact at
$h^{-1}(\bL_X)\vert_R$ and $\cS_X\vert_R$. To prove it is exact at
$\O_X\vert_R$, we work at the level of stalks. Let $x\in R\subseteq
X$, and write $\cS_{X,x},\O_{X,x},T^*X_x,\ldots$ for the stalks of
$\cS_X,\O_X,T^*X,\ldots$ at $x\in X$, so that $\cS_{X,x},\ldots$ are
$\C$-vector spaces whose elements are germs at $x$ of sections of
$\cS_X,\ldots.$ A sequence of sheaves of $\C$-vector spaces on $R$
is exact if and only if it is exact on stalks at every~$x\in R$.

Let $\eta\in\O_{X,x}$ with $\d\eta=0\in T^*X_x$. Then $\de_x(\eta)
\in i^{-1}(\O_U)_x/I_{R,U,x}$, so we may write $\de_x(\eta)=
\ze+I_{R,U,x}$ for some $\ze\in i^{-1}(\O_U)_x$. The exact sequences
in \eq{dc3eq10} induce an isomorphism
\begin{equation*}
T^*X_x\cong i^{-1}(T^*U)_x/\bigl(I_{R,U,x}\cdot i^{-1}(T^*U)_x+
\d(I_{R,U,x})\bigr),
\end{equation*}
which identifies $\d\eta\in T^*X_x$ with $\d\ze+\bigl(I_{R,U,x}\cdot
i^{-1}(T^*U)_x+\d(I_{R,U,x})\bigr)$, so $\d\ze\in I_{R,U,x}\cdot
i^{-1}(T^*U)_x+\d(I_{R,U,x})\subseteq i^{-1}(T^*U)_x$ as $\d\eta=0$.
As $\ze$ is unique up to addition of an element of $I_{R,U,x}$, by
changing our choice of $\ze$ we can eliminate the $\d(I_{R,U,x})$
component in $\d\ze$, so that $\d\ze\in I_{R,U,x}\cdot
i^{-1}(T^*U)_x$. Then $\ze+I_{R,U,x}^2$ lies in
$i^{-1}(\O_U)_x/I_{R,U,x}^2$ with $\d(\ze+I_{R,U,x}^2)=0$ in
$i^{-1}(T^*U)_x/I_{R,U,x}\cdot i^{-1}(T^*U)_x$. Hence by the
exactness of the middle row of \eq{dc3eq10} in stalks at $x$, there
exists a unique $\th\in\cS_{X,x}$ with $\io_{R,U,x}(\th)=
\ze+I_{R,U,x}^2$. Therefore
\begin{equation*}
\de_x\ci\be_{X,x}(\th)=\pi_x\ci\io_{R,U,x}(\th)=
\pi_x(\ze+I_{R,U,x}^2)=\ze+I_{R,U,x}=\de_x(\eta),
\end{equation*}
as the bottom left square in \eq{dc3eq10} commutes.

Since $\de$ is an isomorphism, this forces $\be_{X,x}(\th)=\eta$.
Similarly, if $\th\in\cS_{X,x}$ then $\d\ci\be_{X,x}(\th)=0$. Thus
the first column of \eq{dc3eq10} is exact on stalks at
$\O_X\vert_R$, so it is exact. As such open $R\subseteq X$ cover
$X$, we see that \eq{dc2eq6} is exact. Exactness of \eq{dc2eq7}
follows from \eq{dc2eq5}--\eq{dc2eq6}. This proves
Theorem~\ref{dc2thm1}(b).

For part (c), consider the morphism of sheaves on $R$
\begin{equation*}
\d:i^{-1}(\O_U)/I_{R,U}^2\longra i^{-1}(T^*U)/\d(I_{R,U}^2).
\end{equation*}
Composing this with $\io_{R,U}$ maps $\cS_X\vert_R\ra
i^{-1}(T^*U)/\d(I_{R,U}^2)$. From \eq{dc2eq2}, the image of
$\d\ci\io_{R,U}$ lies in $\bigl(I_{R,U}\cdot
i^{-1}(T^*U)/\d(I_{R,U}^2)\bigr)\subseteq
i^{-1}(T^*U)/\d(I_{R,U}^2)$. Also, as $\d^2=0$, the image lies in
$\Ker\d+\d(I_{R,U}^2)$. This defines a morphism
\e
\d\ci\io_{R,U}:\cS_X\vert_R\longra\frac{\Ker\bigl(\d:I_{R,U}\cdot
i^{-1}(T^*U)\longra i^{-1}(\La^2T^*U)\bigr)}{
\Im\bigl(\d:I_{R,U}^2\longra I_{R,U}\cdot i^{-1}(T^*U)\bigr)}\,,
\label{dc3eq11}
\e
where the right hand side is the cohomology of~\eq{dc2eq8}.

In the splitting $\cS_X\vert_R=\C_X\vert_R\op\cSz_X\vert_R$, clearly
$\C_X\vert_R$ lies in the kernel of $\d\ci\io_{R,U}$ in
\eq{dc3eq11}, since $\d$ of a locally constant function is zero. We
claim that the restriction of \eq{dc3eq11} to
$\cSz_X\vert_R\subset\cS_R$ is an isomorphism. As for part (b), it
is enough to prove this on the stalks at each $x\in R$. By the
Poincar\'e Lemma on the complex manifold $U$, written using
morphisms $\d:\La^kT^*U_{i(x)}\ra\La^{k+1}T^*U_{i(x)}$ of stalks at
$i(x)$, and pulled back to $X$ using $i$, we have an exact sequence
\e
\xymatrix@C=27pt{ 0 \ar[r] & \C \ar[r]^(0.4)1 & i^{-1}(\O_X)_x
\ar[r]^\d & i^{-1}(T^*U)_x \ar[r]^(0.45)\d & i^{-1}(\La^2T^*U)_x. }
\label{dc3eq12}
\e

Let $\phi$ lie in the stalk at $x$ of the r.h.s.\ of \eq{dc3eq11}.
Then $\phi=\psi+\d(I_{R,U,x}^2)$, where $\psi\in I_{R,U,x}\cdot
i^{-1}(T^*U)_x\subseteq i^{-1}(T^*U)_x$ with $\d\psi=0$ in
$i^{-1}(\La^2T^*U)_x$. By exactness of \eq{dc3eq12}, we may write
$\psi=\d\ze$ for $\ze\in i^{-1}(\O_U)_x$, where $\ze$ is unique up
to addition of a constant $1(c)$ for $c\in\C$. We fix $\ze$ uniquely
by requiring that $\ze(i(x))=0$. Then $\ze+I_{R,U,x}^2$ lies in
$i^{-1}(\O_U)_x/I_{R,U,x}^2$ with $\d(\ze+I_{R,U,x}^2)=0$ in
$i^{-1}(T^*U)_x/(I_{R,U,x}\cdot i^{-1}(T^*U)_x)$, since
$\d\ze=\psi\in I_{R,U,x}\cdot i^{-1}(T^*U)_x$. Hence
$\ze+I_{R,U,x}^2=\io_{R,U,x}(\th)$ for unique $\th\in\cS_{X,x}$ as
\eq{dc2eq2} is exact, and $\ze(i(x))=0$ implies that
$\th\in\cSz_{X,x}$. Therefore $\d\ci\io_{R,U}\vert_{\cS_X\vert_R}$
is an isomorphism on stalks, and so is an isomorphism, proving
Theorem~\ref{dc2thm1}(c).

\subsection{Modifications to the proof for the algebraic case}
\label{dc33}

Next we explain how to modify \S\ref{dc31}--\S\ref{dc32} to work
with $\K$-schemes in algebraic geometry, rather than complex
analytic spaces and complex manifolds.

In \S\ref{dc31} we replace complex analytic spaces $X,R$ by
$\K$-schemes, and complex manifolds $U$ by smooth $\K$-schemes, and
open subsets $R\subseteq X$, $U'\subseteq U$, etc., are taken to be
open in the Zariski topology. Then the proofs work in the algebraic
case without modification, working throughout with sheaves in either
the Zariski or the \'etale topology, with the exception of Lemma
\ref{dc3lem2}, which is false: one can write down examples of
$\K$-schemes $X$ with embeddings $i:X\hookra U$, $j:X\hookra V$ into
smooth $\K$-schemes $U,V$, such that for Zariski or \'etale open
$\es\ne U'\subseteq U$, there exist no morphisms $\Phi:U'\ra V$
with~$\Phi\ci i'=j\vert_{U'}$.

Here are two different ways to fix this:
\begin{itemize}
\setlength{\itemsep}{0pt}
\setlength{\parsep}{0pt}
\item[(A)] The analogue of Lemma \ref{dc3lem2} is true if we require
that the smooth $\K$-scheme $V$ is isomorphic to a Zariski open
subset of an affine space $\bA^n$, as then we can take the
coordinates $(z_1,\ldots,z_n)$ in the proof of Lemma
\ref{dc3lem2} to be the embedding~$V\hookra\bA^n$.
\item[(B)] Suppose $R,S\subseteq X$ are Zariski open and
$i:R\hookra U$, $j:S\hookra V$ are closed embeddings into smooth
$\K$-schemes $U,V$. Write $\hat U,\hat V$ for the formal
completions of $U,V$ along $i(R),j(S)$, with inclusions
$\hat\imath:R\hookra \hat U$, $\hat\jmath:S\hookra\hat V$. Then
one can prove the following formal analogue of Lemma
\ref{dc3lem2}: for each $x\in R\cap S\subseteq X,$ there exists
a Zariski open $\hat U'\subseteq\hat U$ with $x\in R'\subseteq
R\cap S$, and a morphism of formal $\K$-schemes $\hat\Phi:\hat
U'\ra\hat V$ with $\hat\Phi\ci\hat\imath'=\hat\jmath\vert_{R'}$,
where $R'=\hat\imath^{-1}(\hat U')$ and~$\hat\imath'=
\hat\imath\vert_{R'}$.
\end{itemize}

Using approach (A), the whole of \S\ref{dc31} works provided we
restrict to triples $(R,U,i)$ with $U$ isomorphic to a Zariski open
in $\bA^n$. In particular, we can construct the sheaf $\cS_X$ using
a family $\bigl\{(R_a,U_a,i_a):a\in A\bigr\}$ with
$U_a\subseteq\bA^{n_a}$ Zariski open. Note that immediately after
the proof of Lemma \ref{dc3lem3}, Zariski locally near each $x\in X$
there exists an embedding $i:X\hookra U$ with~$U\subseteq\bA^n$.

The disadvantage of this is that it proves Theorem
\ref{dc2thm1}(i),(ii) only for $U,V$ Zariski open in affine spaces
$\bA^n$. To prove them for $U,V$ general smooth $\K$-schemes, we
have to do some more work.

To prove the algebraic version of \S\ref{dc31} using (B), observe
that \eq{dc2eq2}--\eq{dc2eq3} defining $\cS_X$ depend only on the
formal completion $\hat U$ of $U$ along $i(R)$, since
\begin{equation*}
\frac{i^{-1}(\O_U)}{I_{R,U}^2}\cong \frac{\hat\imath^{-1}(
\O_{\smash{\hat U}})}{I_{R,\smash{\hat U}}^2}\quad\text{and}\quad
\frac{i^{-1}(T^*U)}{I_{R,U}\cdot i^{-1}(T^*U)}\cong
\frac{\hat\imath^{-1}(T^*\hat U)}{I_{\smash{R,\hat U}}\cdot
\hat\imath^{-1}(T^*\hat U)}\,.
\end{equation*}
Thus we may replace $U,V$ by $\hat U,\hat V$ throughout the proof.
For triples $(R,U,i),\ab(S,V,j)$ with $U,V$ smooth $\K$-schemes and
$R\subseteq S$, a morphism $\Phi:U\ra V$ with $\Phi\ci i=j\vert_{R}$
induces a morphism $\hat\Phi:\hat U\ra\hat V$ with $\hat\Phi\ci
\hat\imath=\hat\jmath\vert_{R}$. However, the converse is false:
there may be $\hat\Phi:\hat U\ra\hat V$ with $\hat\Phi\ci
\hat\imath=\hat\jmath\vert_{R}$ which are not induced by any
$\Phi:U\ra V$ Zariski or \'etale locally on $U$. Because there are
more formal morphisms $\hat\Phi$ than morphisms $\Phi$, the formal
analogue of Lemma \ref{dc3lem2} holds, which makes the proof work.

For the extension of \S\ref{dc32} to $\K$-schemes, the proofs of
Theorem \ref{dc2thm1}(a),(b) need no modification. But for (c),
equation \eq{dc3eq12} is not exact at $i^{-1}(T^*U)_x$ in the
algebraic case: algebraic closed 1-forms need not be locally exact,
in either the Zariski or the \'etale topology. For example, the
closed 1-form $z^{-1}\d z$ on $\bA^1\sm\{0\}$ is not algebraically
locally exact, since $\log z$ is not an algebraic function.

Because of this, in the algebraic case the morphism from $\cSz_X$ to
the cohomology of \eq{dc2eq8} constructed in \S\ref{dc32} is
injective, but generally not surjective. The solution is to modify
\eq{dc2eq8}, replacing $U$ by the formal completion $\hat U$ of $U$
along $i(R)$, as in \eq{dc2eq9}. In fact the Poincar\'e Lemma may
not hold on $\hat U$ either (consider the case $X=U=\hat U$ and
$i=\id_X$), but what matters is that one can show that if
$\hat\al\in H^0(I_{\smash{R,\hat U}}\cdot T^*\hat U)\subset
H^0(T^*\hat U)$ with $\d\hat\al=0$ then $\hat\al=\d f$ for a unique
$f\in H^0(\O_{\smash{\hat U}})$ with $f\ci i\vert_{X^\red}=0$, that
is, the particular class of formal closed 1-forms we are interested
in can be integrated to formal functions. This completes the proof
of Theorem~\ref{dc2thm1}.

\subsection{Proof of Proposition \ref{dc2prop1}}
\label{dc34}

Let $\phi:X\ra Y$ be a morphism of complex analytic spaces. To
construct the morphism $\phi^\star:\phi^{-1}(\cS_Y)\ra \cS_X$ in
Proposition \ref{dc2prop1}, we will generalize the construction of
$\cS_X$ in \S\ref{dc31}. Modifying the notation of \S\ref{dc31},
define an $X$-{\it triple\/} $(R,U,i)_X$ to be an open subset
$R\subseteq X$, a complex manifold $U$, and a closed embedding
$i:R\hookra U$, and a $Y$-{\it triple\/} $(S,V,j)_Y$ to be an open
subset $S\subseteq Y$, a complex manifold $V$, and a closed
embedding $j:S\hookra V$. Define $X$-{\it subtriples\/}
$(R',U',i')_X\subseteq(R,U,i)_X$ and $Y$-{\it subtriples\/}
$(S',V',j')_Y\subseteq(S,V,j)_Y$ as in~\S\ref{dc31}.

Call $\Up:(R_1,U_1,i_1)_X\ra(R_2,U_2,i_2)_X$ a {\it morphism of\/
$X$-triples\/} if $R_1\subseteq R_2\subseteq X$, and $\Up:U_1\ra
U_2$ is holomorphic with $\Up\ci i_1=i_2\vert_{R_1}:R_1\ra U_2$.
Call $\Phi:(R,U,i)_X\ra(S,V,j)_Y$ a {\it morphism of\/
$X,Y$-triples\/} if $\phi(R)\subseteq S\subseteq Y$, and $\Phi:U\ra
V$ is holomorphic with $\Up\ci i=j\ci\phi\vert_R:R\ra V$. Call
$\Psi:(S_1,V_1,j_1)_Y\ra(S_2,V_2,j_2)_Y$ a {\it morphism of\/
$Y$-triples\/} if $S_1\subseteq S_2\subseteq Y$, and $\Psi:V_1\ra
V_2$ is holomorphic with $\Psi\ci j_1=j_2\vert_{S_1}:S_1\ra V_2$.

Let $\Phi:(R,U,i)_X\ra(S,V,j)_Y$ be a morphism of $X,Y$-triples.
Consider the diagram of sheaves of $\C$-vector spaces on $R:$
\e
\begin{gathered}
\xymatrix@R=20pt@C=9pt{ 0 \ar[r] & \phi^{-1}(\cS_Y)\vert_R
\ar[rrr]_(0.45){\phi^{-1}(\io_{S,V})\vert_R}
\ar@{.>}[d]^{\phi^\star_\Phi} &&& {\frac{\ts\phi^{-1}\ci
j^{-1}(\O_V)\vert_R}{\ts\phi^{-1}(I_{S,V}^2)\vert_R}}
\ar[d]^{i^{-1}(\Phi^\sharp)} \ar[rr]_(0.42){\phi^{-1}(\d)} &&
\frac{\ts \phi^{-1}(j^{-1}(T^*V))\vert_R}{\ts \phi^{-1}(I_{S,V}\cdot
j^{-1}(T^*V))\vert_R}
\ar[d]_{i^{-1}(\d\Phi)} \\
0 \ar[r] & \cS_X\vert_R \ar[rrr]^(0.45){\io_{R,U}} &&& {\frac{\ts
i^{-1}(\O_U)}{\ts I_{R,U}^2}} \ar[rr]^(0.42)\d && \frac{\ts
i^{-1}(T^*U)}{\ts I_{R,U}\cdot i^{-1}(T^*U)}\,,
}\!\!\!\!\!\!\!\!\!\!\!\!\!\!\!{}
\end{gathered}
\label{dc3eq13}
\e
which is \eq{dc2eq11} with $\phi^\star_\Phi$ in place of
$\phi^\star\vert_R$. The rows of \eq{dc3eq13} are exact, as
\eq{dc2eq2} is exact, and the right hand square of \eq{dc3eq13}
commutes by definition of $\d\Phi$. So by exactness, there is a
unique morphism $\phi^\star_\Phi$ of sheaves of $\C$-vector spaces
making \eq{dc3eq13} commute. Since $\io_{R,U},\phi^{-1}(\io_{S,V})
\vert_R,i^{-1}(\Phi^\sharp)$ are morphisms of sheaves of commutative
$\C$-algebras, so is~$\phi^\star_\Phi$.

Suppose $\Up:(R_1,U_1,i_1)_X\ra(R_2,U_2,i_2)_X$ is a morphism of
$X$-triples, and $\Phi_2:(R_2,U_2,i_2)_X\ra(S_2,V_2,j_2)_Y$ a
morphism of $X,Y$-triples. Then $\Phi_2\ci\Up:(R_1,U_1,i_1)_X\ra
(S_2,V_2,j_2)_Y$ is a morphism of $X,Y$-triples, and composing
\eq{dc2eq3} for $\Up$ with the restriction of \eq{dc3eq13} for
$\Phi_2$ to $R_1$, we see that
\e
\phi^\star_{\Phi_2\ci\Up}=\phi^\star_{\Phi_2}\vert_{R_1}:
\phi^{-1}(\cS_Y)\vert_{R_1}\longra\cS_X\vert_{R_1}.
\label{dc3eq14}
\e
Similarly, suppose $\Phi_1:(R_1,U_1,i_1)_X\ra(S_1,V_1,j_1)_Y$ is a
morphism of $X,Y$-triples, and $\Psi:(S_1,V_1,j_1)_Y
\ra(S_2,V_2,j_2)_Y$ a morphism of $Y$-triples. Then $\Psi\ci\Phi_1:
(R_1,U_1,i_1)_X\ra (S_2,V_2,j_2)_Y$ is a morphism of $X,Y$-triples,
and composing \eq{dc3eq13} for $\Phi_1$ with $\phi\vert_{R_1}^{-1}$
applied to \eq{dc2eq3} for $\Psi$, we see that
\e
\phi^\star_{\Psi\ci\Phi_1}=\phi^\star_{\Phi_1}:
\phi^{-1}(\cS_Y)\vert_{R_1}\longra\cS_X\vert_{R_1}.
\label{dc3eq15}
\e

Easy generalizations of the proof of Lemmas \ref{dc3lem1} and
\ref{dc3lem2} show:

\begin{lem} The morphism\/ $\phi^\star_\Phi$ in \eq{dc3eq13} is
independent of the choice of\/ $\Phi$. That is, if\/
$\Phi,\ti\Phi:(R,U,i)_X\ra(S,V,j)_Y$ are morphisms of\/
$X,Y$-triples then
\e
\phi^\star_\Phi=\phi^\star_{\ti\Phi}:\phi^{-1}(\cS_Y)\vert_R
\longra\cS_X\vert_R.
\label{dc3eq16}
\e

\label{dc3lem4}
\end{lem}

\begin{lem} Let\/ $(R,U,i)_X$ be an $X$-triple and\/ $(S,V,j)_Y$ a
$Y$-triple. Then for each\/ $x\in R\cap \phi^{-1}(S)\subseteq X,$
there exist an $X$-subtriple $(R',U',i')_X\subseteq(R,U,i)_X$ with\/
$x\in R'\subseteq R\cap \phi^{-1}(S)$ and a morphism of\/
$X,Y$-triples\/~$\Phi:(R',U',i')_X\ra(S,V,j)_Y$.
\label{dc3lem5}
\end{lem}

We now claim that there exists a unique morphism $\phi^\star:
\phi^{-1}(\cS_Y)\ra \cS_X$ of sheaves of commutative $\C$-algebras
on $X$, such that $\phi^\star\vert_R=\phi^\star_\Phi$ for all
morphisms of $X,Y$-triples $\Phi:(R,U,i)_X\ra(S,V,j)_Y$. To show
this, it is enough to prove:
\begin{itemize}
\setlength{\itemsep}{0pt}
\setlength{\parsep}{0pt}
\item[(a)] For all $x\in X$, there exists a morphism of
$X,Y$-triples $\Phi:(R,U,i)_X\ra(S,V,j)_Y$ with $x\in R$; and
\item[(b)] If $\Phi_1:(R_1,U_1,i_1)_X\ra(S_1,V_1,j_1)_Y$ and
$\Phi_2:(R_2,U_2,i_2)_X\ra(S_2,V_2,j_2)_Y$ are morphisms of
$X,Y$-triples, then $\phi^\star_{\Phi_1}\vert_{R_1\cap R_2}=
\phi^\star_{\Phi_2}\vert_{R_1\cap R_2}$.
\end{itemize}
To see part (a) holds, choose an $X$-triple $(R,U,i)_X$ with $x\in
R$ and a $Y$-triple $(S,V,j)_Y$ with $\phi(x)\in S$, and apply Lemma
\ref{dc3lem5}. For (b), let $x\in R_1\cap R_2$. Applying Lemma
\ref{dc3lem2} twice gives an $X$-subtriple
$(R_1',U_1',i_1')_X\subseteq (R_1,U_1,i_1)_X$ with $x\in
R_1'\subseteq R_1\cap R_2$ and a morphism of $X$-triples
$\Up:(R_1',U_1',i_1')_X\ra (R_2,U_2,i_2)_X$, and a $Y$-subtriple
$(S_1',V_1',j_1')_Y\subseteq (S_1,V_1,j_1)_Y$ with $\phi(x)\in S_1'$
and a morphism of $Y$-triples $\Psi:(S_1',V_1',j_1')_Y\ra
(S_2,V_2,j_2)_Y$. Making $R_1',U_1'$ smaller if necessary we may
suppose that $\Phi_1(U_1')\subseteq V_1'\subseteq V_1$. Then we have
\begin{equation*}
\phi^\star_{\Phi_1}\vert_{R_1'}=\phi^\star_{\Phi_1\vert_{U_1'}}=
\phi^\star_{\Psi\ci\Phi_1\vert_{U_1'}}=\phi^\star_{\Phi_2\ci\Up}
=\phi^\star_{\Phi_2}\vert_{R_1'},
\end{equation*}
using \eq{dc3eq15} in the second step, Lemma \ref{dc3lem4} for
$\Psi\ci\Phi_1\vert_{U_1'},\Phi_2\ci\Up:(R_1',U_1',i_1')_X\ra
(S_2,V_2,j_2)_Y$ in the third, and \eq{dc3eq14} in the fourth. Thus
for each $x\in R_1\cap R_2$ we can find an open $x\in R_1'\subseteq
R_1\cap R_2$ with $\phi^\star_{\Phi_1}\vert_{R_1'}=
\phi^\star_{\Phi_2}\vert_{R_1'}$. This implies~(b).

Thus, comparing \eq{dc2eq11} and \eq{dc3eq13} shows that there
exists a unique morphism $\phi^\star$ such that \eq{dc2eq11}
commutes for all $R,S,U,V,i,j,\Phi$ as in Proposition
\ref{dc2prop1}. Consider the diagram
\e
\begin{gathered}
\xymatrix@C=50pt@R=15pt{ 0 \ar[r] & \phi^{-1}(\cSz_Y)
\ar@{.>}[d]_{\phi^\star\vert_{\phi^{-1}(\cSz_Y)}} \ar[r] &
\phi^{-1}(\cS_Y) \ar[d]_{\phi^\star}
\ar[r]_{\phi^{-1}(i_Y^\sh\ci\be_Y)} & \phi^{-1}(\O_{Y^\red})
\ar[d]_{\phi\vert_{X^\red}^\sharp} \\ 0 \ar[r] & \cSz_X \ar[r] &
\cS_X \ar[r]^{i_X^\sh\ci\be_X} & \O_{X^\red}. }
\end{gathered}
\label{dc3eq17}
\e
The rows are exact by Theorem \ref{dc2thm1}(a), and one can see the
right hand square commutes by composing the left hand square of
\eq{dc2eq11} with the projections to $\phi^{-1}(\O_{Y^\red})\vert_R,
\O_{X^\red}\vert_R$. Thus by exactness, $\phi^\star$ maps
$\phi^{-1}(\cSz_Y)\ra \cSz_X$ as in~\eq{dc3eq17}.

Now let $\psi:Y\ra Z$ be another morphism of complex analytic
spaces. Using the obvious notation, suppose $\Phi:(R,U,i)_X\ra
(S,V,j)_Y$, $\Psi:(S,V,j)_Y\ra(T,W,k)_Z$ are morphisms of
$X,Y$-triples and $Y,Z$-triples, so that $\Psi\ci\Phi:(R,U,i)_X\ra
(T,W,k)_Z$ is a morphism of $X,Z$-triples. Then comparing
\eq{dc2eq11} for $\Psi\ci\Phi$ with the composition of
$\phi\vert_R^{-1}$ applied to \eq{dc2eq11} for $\Psi$ with
\eq{dc2eq11} for $\Phi$ shows that $(\psi\ci\phi)^\star\vert_R=
\phi^\star\ci\phi^{-1}(\psi^\star)\vert_R$. As we can cover $X$ by
such open $R\subseteq X$, equation \eq{dc2eq12} follows.

To show that $\id_X^\star=\id_{\cS_X}$, compare \eq{dc2eq3} with
\eq{dc2eq11} with $R,U,i,\id_U$ in place of $S,V,j,\Phi$. Finally,
if $\phi:X\ra Y$ is an \'etale morphism of complex analytic spaces,
then $\phi$ is a local isomorphism in the complex analytic topology,
so $\phi^\star$ is an isomorphism. This proves Proposition
\ref{dc2prop1} for complex analytic spaces.

The extension to $\K$-schemes works as in \S\ref{dc33}. For the last
part, for $\cS_X,\cS_Y$ sheaves in the Zariski topology, if
$\phi:X\ra Y$ is a Zariski inclusion then it is an isomorphism
locally in the Zariski topology, so $\phi^\star$ is an isomorphism.
Similarly, for $\cS_X,\cS_Y$ sheaves in the \'etale topology, if
$\phi:X\ra Y$ is \'etale then it is an isomorphism locally in the
\'etale topology, so again $\phi^\star$ is an isomorphism.

\section{D-critical loci}
\label{dc4}

We now prove Propositions \ref{dc2prop2} and \ref{dc2prop3}
from~\S\ref{dc22}.

\subsection{Proof of Proposition \ref{dc2prop2}}
\label{dc41}

We will prove the $\K$-scheme case, as it is more complicated, and
we tackle the second part of the proposition first.

Suppose $(X,s)$ is an algebraic d-critical locus, and $x\in X$. Let
$(T,W,h,k)$ be a critical chart on $(X,s)$ with $x\in T$, and set
$\dim T_xX=m$ and $\dim W=n$. Then $\Hess_{k(x)}h$ has rank $n-m$.
Choose \'etale coordinates $(z_1,\ldots,z_n)$ on a Zariski open
neighbourhood $W'$ of $k(x)$ in $W$, such that
\e
\frac{\pd h}{\pd z_i\pd z_j}\Big\vert_{k(x)}=\begin{cases} 1, & i=j\in
\{m+1,m+2,\ldots,n\}, \\
0, & \text{otherwise.}\end{cases}
\label{dc4eq1}
\e
Define $V=\bigl\{w\in W':\frac{\pd h}{\pd z_{m+1}}(w)=\cdots=
\frac{\pd h}{\pd z_n}(w)=0\bigr\}$. Equation \eq{dc4eq1} implies
that the equations $\frac{\pd h}{\pd z_j}(w)=0$ for $j=m+1,\ldots,n$
are transverse at $k(x)$, so $V$ is smooth of dimension $n-m$ near
$k(x)$. Making $W'$ smaller, we can suppose $V$ is smooth of
dimension~$n-m$.

Define $S=k^{-1}(W')$, so that $x\in S\subseteq T\subseteq X$ is
Zariski open and $k\vert_S:S\hookra W'$ is a closed embedding with
$k(S)$ the $\K$-subscheme $\d h\vert_{W'}^{-1}(0)$ in $W'$. But $\d
h\vert_{W'}=0$ implies that $\frac{\pd h}{\pd z_{m+1}}(w)=\cdots=
\frac{\pd h}{\pd z_n}(w)=0$, so $k(S)\subseteq V\subseteq W'$. Thus
$j:=k\vert_S:S\hookra V$ is a closed embedding. Write
$g:=h\vert_V:V\ra\bA^1$. It is now easy to check that $(S,V,g,j)$ is
a critical chart on $(X,s)$ with $x\in S$ and~$\dim V=\dim T_xX$.

Suppose now that $x\in R\subseteq X$ is Zariski open, and
$i:R\hookra U$ is a closed embedding into a smooth $\K$-scheme $U$
with $\dim U=\dim T_xX=m$, and $f:U\ra\bA^1$ is regular with
$\io_{R,U}(s\vert_R)=i^{-1}(f)+I_{R,U}^2$. Let $(S,V,g,j)$ be as
above. Then $i\vert_{R\cap S}\t j\vert_{R\cap S}:R\cap S\ra U\t V$
is a closed embedding. Choose a Zariski open neighbourhood $\ti S$
of $x$ in $R\cap S$ and a smooth locally closed $\K$-subscheme $\ti
V$ of $U\t V$ such that $\dim\ti V=m$ and $(i\t j)(\ti S)\subseteq
\ti V\subseteq U\t V$ as $\K$-subschemes of $U\t V$, with $(i\t
j)(\ti S)$ closed in~$\ti V$.

Write $\ti\jmath=(i\t j)\vert_{\ti S}:\ti S\ra \ti V$ and $\pi_U:\ti
V\ra U$, $\pi_V:\ti V\ra V$ for the projections. As
$T_{(i(x),j(x))}\ti V= \d \ti\jmath\vert_x(T_xX)=(\d i\vert_x\t\d
j\vert_x)(T_xX)$ and $\d i\vert_x:T_xX\ra T_{i(x)}U$, $\d
j\vert_x:T_xX\ra T_{j(x)}V$ are isomorphisms, we see that
$\d\pi_U\vert_{(i(x),j(x))}:T_{(i(x),j(x))}\ti V\ra T_{i(x)}U$ and
$\d\pi_V\vert_{(i(x),j(x))}:T_{(i(x),j(x))}\ti V\ra T_{j(x)}V$ are
isomorphisms, so $\pi_U,\pi_V$ are \'etale near $(i(x),j(x))$.
Making $\ti S,\ti V$ smaller, we can suppose $\pi_U,\pi_V$ are
\'etale. Since $\pi_V\ci\ti\jmath=j\vert_{\ti S}$, and $j:S\hookra
V$ is a closed embedding, we see that $\ti\jmath(\ti S)$ is open and
closed in $\pi_V^{-1}(j(S))$. Thus, making $\ti V$ smaller, we can
suppose that~$\ti\jmath(\ti S)=\pi_V^{-1}(j(S))$.

Define $\ti g=g\ci\pi_V:\ti V\ra\bA^1$. Since $\pi_V$ is \'etale,
$\Crit(\ti g)=\pi_V^{-1}(\Crit(g))=\pi_V^{-1}(j(S))=\ti\jmath(\ti
S)$. As $(S,V,g,j)$ is a critical chart, we have
$\io_{S,V}(s\vert_S)=j^{-1}(g)+I_{S,V}^2$. Combining this with $\ti
g=g\ci\pi_V$, $\pi_V\ci\ti\jmath=j\vert_{\ti S}$, and Theorem
\ref{dc2thm1}(ii), proves that $\io_{\smash{\ti S,\ti
V}}(s\vert_{\smash{\ti S}})= \ti\jmath^{-1}(\ti g)+I_{\smash{\ti
S,\ti V}}^2$. Therefore $(\ti S,\ti V,\ti g,\ti\jmath)$ is a
critical chart on~$(X,s)$.

Now $\pi_U:\ti S\ra U$ is \'etale, and using
$\io_{R,U}(s\vert_R)=i^{-1}(f)+I_{R,U}^2$ and Theorem
\ref{dc2thm1}(ii) again shows that $\ti\jmath^{-1}(\ti
g)+I_{\smash{\ti S,\ti V}}^2=\ti\jmath^{-1}(\pi_U^*(f))+
I_{\smash{\ti S,\ti V}}^2$, so that
\e
\pi_U^*(f)-\ti g\in I_{\d\ti g}^2\subseteq\O_{\ti V},
\label{dc4eq2}
\e
where $I_{\d\ti g}\subseteq\O_{\ti V}$ is the ideal generated by $\d\ti
g$. Differentiating \eq{dc4eq2} gives
\begin{equation*}
\d\bigl(\pi_U^*(f)-\ti g\bigr)\in I_{\d\ti g}\cdot I_{\d\ti
g,\pd^2\ti g}\cdot T^*V,
\end{equation*}
where $I_{\d\ti g,\pd^2\ti g}\subseteq\O_{\ti V}$ is the ideal
generated by the first and second derivatives of $\ti g$. Therefore
locally on $\ti V$ we may write
\e
\d\bigl(\pi_U^*(f)\bigr)=(\id+\al)\cdot\d\ti g,\quad \text{where}
\quad \al\in I_{\d\ti g,\pd^2\ti g}\cdot TV\ot T^*V.
\label{dc4eq3}
\e

As $\ti\jmath(x)\in\Crit(\ti g)$ and $T_{\ti\jmath(x)}\Crit(\ti g)
=T_{\ti\jmath(x)}\ti V$ we have $\d\ti g\vert_{\ti\jmath(x)}=
\pd^2\ti g\vert_{\ti\jmath(x)}=0$, so $\al\vert_{\ti\jmath(x)}=0$,
and $\id+\al$ is invertible near $\ti\jmath(x)$. Making $\ti S,\ti
V$ smaller, we can suppose $\id+\al$ is invertible on $\ti V$. So
\eq{dc4eq3} implies that $\Crit(\pi_U^*(f))=\Crit(\ti g)$, as
$\K$-subschemes of $\ti V$. Hence $(\ti S,\ti
V,\pi_U^*(f),\ti\jmath)$ is a critical chart on $(X,s)$.

As $\pi_U:\ti V\ra U$ is \'etale, it maps
$\Crit(\pi_U^*(f))\ra\Crit(f)$ on the image of $\pi_U$. But
$\Crit(\pi_U^*(f))=\ti\jmath(\ti S)$, so
$\pi_U\bigl(\Crit(\pi_U^*(f))\bigr)=\pi_U\ci\jmath(\ti S)=i(\ti S)$.
Thus $\Crit(f)\cap \pi_U(\ti V)=i(\ti S)$, and $\Crit(f)$ coincides
with $i(R)$ near $i(x)$. Choose open $i(x)\in U'\subseteq\pi_U(\ti
V)\subseteq U$, and set $R'=i^{-1}(U')$, $f'=f\vert_{U'}$ and
$i'=i\vert_{R'}$. Since $\pi_U$ is \'etale, $\pi_U^{-1}(\Crit(f'))=
\Crit(\pi_U^*(f))\cap \pi_U^{-1}(U')=\ti\jmath(\ti S)\cap
\pi_U^{-1}(U')=\pi_U^{-1}(i'(R'))$, and this forces
$\Crit(f')=i'(R')$, since the \'etale morphism $\pi_U$ is a Zariski
isomorphism $\jmath(\ti S)\ra i(\ti S)$. Therefore $(R',U',f',i')$
is a critical chart on $(X,s)$. This proves the second part of
Proposition \ref{dc2prop2} in the $\K$-scheme case.

For the first part, suppose $(X,s)$ is an algebraic d-critical
locus, $R\subseteq X$ is Zariski open, $i:R\hookra U$ is a closed
embedding into a smooth $\K$-scheme $U$, and $x\in R$. Let $n=\dim
U$ and $m=\dim T_xX$. Choose a Zariski open neighbourhood $U'$ of
$x$ in $X$ and \'etale coordinates $(z_1,\ldots,z_n):U'\ra\bA^n$
such that $z_{m+1}\vert_{i'(R')}=\cdots=z_n\vert_{i'(R')}=0$, where
$R'=i^{-1}(U')$ and $i'=i\vert_{R'}:R'\hookra U'$, a closed
embedding. Making $R',U'$ smaller if necessary, we may choose
regular $f:U'\ra\bA^1$ with $\io_{R',U'}(s\vert_{R'})=i^{\prime
-1}(f)+I_{R',U'}^2$.

Set $V'=\{v\in U':z_{m+1}(v)=\cdots=z_n(v)=0\}$, so that $V'$ is a
smooth $\K$-subscheme of $U'$ with $\dim V'=m=\dim T_xX$ and
$i'(R')\subseteq V'\subseteq U'$. The proof above shows that
$\Crit(f\vert_{V'})=i'(R')$ near $i'(x)$, so making $R',U',V'$
smaller we may suppose that $\Crit(f\vert_{V'})=i'(R')$. Define
$f':U'\ra\bA^1$ by
\e
f'=f-\sum_{j=m+1}^n\frac{\pd f}{\pd z_j}\,z_j
+\frac{1}{2}\sum_{j,k=m+1}^n\biggl(\de_{jk}+\frac{\pd^2f}{\pd z_j\pd
z_k}\biggr)z_jz_k.
\label{dc4eq4}
\e

Differentiating \eq{dc4eq4} shows that for $j=m+1,\ldots,n$ we have
\begin{equation*}
\frac{\pd f'}{\pd z_j}=z_j+\frac{1}{2}\sum_{k,l=m+1}^n
\frac{\pd^3f}{\pd z_j\pd z_k\pd z_l}\,z_kz_l.
\end{equation*}
Rewriting this in matrix form gives
\e
\begin{pmatrix} \frac{\pd f'}{\pd z_{m+1}} \\ \vdots \\
\frac{\pd f'}{\pd z_n} \end{pmatrix} = \begin{pmatrix} \de_{jk} +
\frac{1}{2}\sum_{l=m+1}^n \frac{\pd^3f}{\pd z_j\pd z_k\pd z_l}\,z_l
\end{pmatrix}_{j,k=m+1}^n \begin{pmatrix} z_{m+1} \\ \vdots \\ z_n
\end{pmatrix}
\label{dc4eq5}
\e
As $z_l=0$ on $V'\subseteq U'$ the first matrix on the r.h.s.\ of
\eq{dc4eq5} is invertible near $V'$ in $U'$, so making $U'$ smaller
while fixing $R',V'$, we can suppose it is invertible. Then the
$\K$-subscheme $\frac{\pd f'}{\pd z_{m+1}}=\cdots=\frac{\pd f'}{\pd
z_n}=0$ in $U'$ is $z_{m+1}=\cdots=z_n=0$, that is, $V'$. Therefore
as $f'\vert_{V'}=f\vert_{V'}$ by \eq{dc4eq4} we have
\e
\Crit(f')=\Crit(f'\vert_{V'})=\Crit(f\vert_{V'})=i'(R').
\label{dc4eq6}
\e

Since $z_j\vert_{i'(R')}=\frac{\pd f}{\pd z_j}\vert_{i'(R')}=0$ for
$j=m+1,\ldots,n$ we have $i^{\prime -1}(z_j)\in I_{R',U'}$ and
$i^{\prime -1}\bigl(\frac{\pd f}{\pd z_j}\bigr)\in I_{R',U'}$ for
$j=m+1,\ldots,n$. Thus \eq{dc4eq4} implies that
\e
i^{\prime -1}(f')+I_{R',U'}^2=i^{\prime
-1}(f)+I_{R',U'}^2=\io_{R',U'}(s\vert_{R'}).
\label{dc4eq7}
\e
Equations \eq{dc4eq6}--\eq{dc4eq7} imply that $(R',U',f',i')$ is a
critical chart on $(X,s)$. This completes the proof of Proposition
\ref{dc2prop2} for $\K$-schemes.

For complex analytic spaces, the proof above also works more-or-less
without change, but it can be simplified, as \'etale morphisms of
complex analytic spaces are local isomorphisms, and so are
invertible on suitable open sets.

\subsection{Proof of Proposition \ref{dc2prop3}}
\label{dc42}

Let $\phi:X\ra Y$ be a smooth morphism of complex analytic spaces
(or $\K$-schemes) and $t\in H^0(\cSz_Y)$, and set
$s=\phi^\star(t)\in H^0(\cSz_X)$. Fix $x\in X$ with $y=\phi(x)$.
Write $\dim T_xX=m$ and $\dim T_yY=n$, so that $m\ge n$ and $\phi$
is smooth of relative dimension $m-n$ near $x$. We may choose
(Zariski) open $y\in S\subseteq Y$ and $x\in
R\subseteq\phi^{-1}(S)\subseteq X$, closed embeddings $i:R\hookra
U$, $j:S\hookra V$ for $U,V$ complex manifolds (or smooth
$\K$-schemes) with $\dim U=m$, $\dim V=n$, and a morphism $\Phi:U\ra
V$ smooth of relative dimension $m-n$ with $\Phi\ci
i=j\ci\phi\vert_R:R\ra V$. Making $S,V$ (and hence $R,U$) smaller,
we may choose holomorphic $g:V\ra\C$ (or regular $g:V\ra\bA^1$) with
$\io_{S,V}(t\vert_S)=j^{-1}(g)+I_{S,V}^2$. Define $f=g\ci\Phi$. Then
$s=\phi^\star(t)$ and Proposition \ref{dc2prop1} imply
that~$\io_{R,U}(s\vert_R)=i^{-1}(f)+I_{R,U}^2$.

Since $\dim T_xX=\dim U$, Proposition \ref{dc2prop2} implies that
$(X,s)$ is a d-critical locus near $x$ if and only if
$\Crit(f)=i(R)$ near $i(x)$ as complex analytic subspaces (or
$\K$-subschemes) of $X$. Similarly, $(Y,t)$ is a d-critical locus
near $y$ if and only if $\Crit(g)=j(R)$ near $j(y)$. But $\Phi:U\ra
V$ is smooth and $f=g\ci\Phi$ implies that
$\Crit(f)=\Phi^{-1}(\Crit(g))$. Also $\Phi\ci i=j\ci\phi\vert_R$ and
$\phi,\Phi$ smooth of relative dimension $m-n$ imply that
$i(R)=\Phi^{-1}(j(S))$ near $i(R)$.

Therefore $\Crit(f)=i(R)$ near $i(x)$ if and only if $\Crit(g)=j(R)$
near $j(y)$, and so $(X,s)$ is a d-critical locus near $x$ if and
only if $(Y,t)$ is a d-critical locus near $y$. Hence $(Y,t)$ a
d-critical locus implies $(X,s)$ is a d-critical locus, and if
$\phi:X\ra Y$ is surjective then $(X,s)$ a d-critical locus implies
$(Y,t)$ is a d-critical locus. If $(X,s),(Y,t)$ are d-critical loci
then $\phi:(X,s)\ra(Y,t)$ is a morphism of d-critical loci by
Definition \ref{dc2def1}. This proves Proposition~\ref{dc2prop3}.

\section{Comparing critical charts $(R,U,f,i)$}
\label{dc5}

Next we prove Theorem \ref{dc2thm2} and Propositions \ref{dc2prop6}
and \ref{dc2prop7}--\ref{dc2prop9} from~\S\ref{dc23}.

\subsection{Proof of Proposition \ref{dc2prop6}}
\label{dc51}

Let $(R,U,f,i)$ be a critical chart on an algebraic d-critical locus
$(X,s)$ over $\K$, and $x\in R$. Then $U$ is a smooth $\K$-scheme,
so there exist an affine open neighbourhood $\ti U$ of $i(x)$ in $U$
and a closed embedding of $\K$-schemes $\ti\Phi:\ti U\hookra\bA^n$
for some $n\ge 0$. Choose a Zariski open neighbourhood $V$ of
$\ti\Phi(i(x))$ in $\bA^n$ and \'etale coordinates
$(z_1,\ldots,z_n)$ on $V$ such that $\ti\Phi(\ti U)\cap V$ is the
$\K$-subscheme defined by $z_{m+1}=\cdots=z_n=0$ in $V$, where
$m=\dim U$. Set $U'=\ti\Phi{}^{-1}(V)$, $R'=S=i^{-1}(U')$,
$i'=i\vert_{R'}:R'\hookra U'$, $f'=f\vert_{U'}:U'\ra\bA^1$,
$\Phi=\ti\Phi\vert_{U'}:U'\hookra V$, and~$j=\Phi\ci i':S=R'\hookra
V$.

Now $\Phi(U')$ is a closed $\K$-subscheme of $V$, and
$f'\ci\Phi^{-1}:\Phi(U')\ra\bA^1$ a regular function. Zariski
locally on $V$, we may extend $f'\ci\Phi^{-1}$ from $\Phi(U')$ to
$V$. Thus, making $V,U',R',S$ smaller, there exists a regular
$h:V\ra\bA^1$ with $h\ci\Phi=f':U'\ra\bA^1$. Then we may form the
partial derivatives $\frac{\pd h}{\pd z_a},\frac{\pd^2h}{\pd z_a\pd
z_b}$ for $a,b=1,\ldots,n$. Define $g:V\ra\bA^1$ by
\e
g=\ts h-\sum_{a=m+1}^nz_a\cdot\frac{\pd h}{\pd z_a}+
\ha\sum_{a,b=m+1}^nz_az_b\cdot\frac{\pd^2 h}{\pd z_a\pd
z_b}+\sum_{a=m+1}^nz_a^2.
\label{dc5eq1}
\e
Here we use $\mathop{\rm char}\K\ne 2$ from \S\ref{dc1}. Since
$z_a\vert_{\Phi(U')}=0$ for $a=m+1,\ldots,n$, we have
$g\ci\Phi=h\ci\Phi=f':U'\ra\bA^1$. Also equation \eq{dc5eq1} implies
that
\e
\ts\frac{\pd g}{\pd
z_a}\vert_{\Phi(U')}\!=\!0,\;\>a=m\!+\!1,\ldots,n,\;\>
\ts\frac{\pd^2g}{\pd z_a\pd
z_b}\vert_{\Phi(U')}\!=\!2\de_{ab},\;\>a,b=m\!+\!1,\ldots,n.
\label{dc5eq2}
\e

Consider the $\K$-subscheme of $V$ defined by the equations
$\frac{\pd g}{\pd z_a}=0$ for $a=m+1,\ldots,n$. The first equation
of \eq{dc5eq2} shows this $\K$-subscheme contains $\Phi(U')$, and
the second that this $\K$-subscheme coincides with $\Phi(U')$ near
$\Phi(U')$ in $V$. Hence, by making $V$ smaller while keeping
$U',R',S$ fixed, we can take this $\K$-subscheme to be $\Phi(U')$.
We now claim that
\e
\Crit(g)=\Phi(\Crit(f'))=\Phi(i(R'))=j(S)
\label{dc5eq3}
\e
as $\K$-subschemes of $V$. To see this, note that the equations
$\frac{\pd g}{\pd z_a}=0$ for $a=1,\ldots,m$ defining $\Crit(g)$
divide into $\frac{\pd g}{\pd z_a}=0$ for $a=m+1,\ldots,n$, which
define the smooth $\K$-subscheme $\Phi(U')$ in $V$, and $\frac{\pd
g}{\pd z_a}=0$ for $a=1,\ldots,m$, which under $\Phi(U')\cong U'$
correspond to $\d f'=0$ as~$g\ci\Phi=f'$.

As $(R',U',f',i')$ is a critical chart on $(X,s)$ we have
$\io_{R',U'}(s\vert_{R'})=i^{\prime -1}(f')+I_{R',U'}^2$. Applying
Theorem \ref{dc2thm1}(ii) with $(R',U',i)$ in place of $(R,U,i)$ and
using $g\ci\Phi=f'$, $j=\Phi\ci i'$ and $R'=S$ shows that
$\io_{S,V}(s\vert_S)=j^{-1}(g) +I_{S,V}^2$. Together with
\eq{dc5eq3}, this implies that $(S,V,g,j)$ is a critical chart, and
$\Phi:U'\hookra V$ an embedding with $j=\Phi\ci i'$, $g\ci\Phi=f'$
show $\Phi:(R,U,f,i)\hookra(S,V,g,j)$ is an embedding of critical
charts. This proves Proposition~\ref{dc2prop6}.

\subsection{Proof of Theorem \ref{dc2thm2}}
\label{dc52}

We begin with the complex analytic case. Let $(X,s)$ be a complex
analytic d-critical locus, $(R,U,f,i),(S,V,g,j)$ be critical charts
on $(X,s)$, and $x\in R\cap S$. Then $i(R\cap S)$ is a locally
closed complex analytic subspace of $U$, and $j\ci
i^{-1}\vert_{i(R\cap S)}:i(R\cap S)\ra V$ a morphism to a complex
manifold $V$. So we can extend $j\ci i^{-1}$ locally to a
holomorphic map $U\ra V$. That is, we can choose an open
neighbourhood $U'$ of $i(x)$ in $U$ with $R':=i^{-1}(U')\subseteq
R\cap S$, and a holomorphic map $\Th:U'\ra V$ such that $\Th\ci
i'=j\vert_{R'}:R'\ra V$, for~$i':=i\vert_{R'}$.

Theorem \ref{dc2thm1}(ii) with $(R',U',i'),\Th$ in place of
$(R,U,i),\Phi$ gives a commutative diagram \eq{dc2eq3}. Applying
this to $s\vert_{R'}$ shows that
\e
\io_{R',U'}(s\vert_{R'})=i^{\prime -1}(\Th^\sh)\bigl[
\io_{S,V}(s\vert_{S})\bigr].
\label{dc5eq4}
\e
Write $I_{R',U'}'\subseteq \O_{U'}$ for the ideal vanishing on
$i(R')\subseteq U'$, and $f'=f\vert_{U'}$. Then
\ea
i^{\prime -1}\bigl[f'&+(I_{R',U'}')^2\bigr]= i^{\prime
-1}(f')+I_{R',U'}^2=\io_{R',U'}(s\vert_{R'})=i^{\prime
-1}(\Th^\sh)\bigl[\io_{S,V}(s\vert_S)\bigr]
\nonumber\\
&=i^{\prime -1}(\Th^\sh)\bigl[j^{-1}(g)+I_{S,V}^2\bigr]=i^{\prime
-1}(\Th^\sh)\bigl[i^{\prime -1}\ci\Th^{-1}(g)+I_{S,V}^2\bigr]
\nonumber\\
&=i^{\prime -1}\bigl[\Th^\sh\bigl(\Th^{-1}(g)\bigr)\bigr]
+I_{R',U'}^2=i^{\prime -1}\bigl[g\ci\Th+(I_{R',U'}')^2\bigr],
\label{dc5eq5}
\ea
using $i^{\prime -1}(I_{R',U'}')=I_{R',U'}$ in the first and seventh
steps, $\io_{R',U'}(s\vert_{R'})=i^{\prime -1}(f')+I_{R',U'}^2$ in
the second, \eq{dc5eq4} in the third, $\io_{S,V}(s\vert_S)=
j^{-1}(g)+I_{S,V}^2$ in the fourth, $\Th\ci i'=j\vert_{R'}$ in the
fifth, and $i^{\prime -1}(\Th^\sh)(I_{S,V})=I_{R',U'}$ in the sixth.

Equation \eq{dc5eq5} implies that $f'-g\ci\Th\in (I_{R',U'}')^2$.
Therefore, making $U',R'$ smaller if necessary, we can choose
holomorphic functions $r_a,s_a:U'\ra\C$ for $a=1,\ldots,n$, some
$n\ge 0$, such that $r_a,s_a\in H^0(I_{R',U'}')$ for $a=1,\ldots,n$
and
\e
f'=g\ci\Th+r_1s_1+\cdots+r_ns_n:U'\longra\C.
\label{dc5eq6}
\e
Define $W=V\t\C^{2n}$, $T=S$, $k=j\t(0,\ldots,0):T=S\hookra
V\t\C^{2n}=W$, $(S',V',g',j')=(S,V,g,j)$, and $\Phi:U'\ra W$,
$\Psi:V'\ra W$, $h:W\ra\C$ by
\begin{gather}
\Phi(u)=\bigl(\Th(u),(r_1(u),\ldots,r_n(u),s_1(u),\ldots,s_n(u))\bigr),\;\>
\Psi(v)=\bigl(v,(0,\ldots,0)\bigr),
\nonumber\\
\text{and}\quad h\bigl(v,(y_1,\ldots,y_n,z_1,\ldots,z_n)\bigr)=g(v)+
y_1z_1+\cdots+y_nz_n.
\label{dc5eq7}
\end{gather}
By increasing $n$ and adding extra functions $r_a,s_a$ with $s_a=0$,
we can suppose $\Phi$ is an embedding near $i(x)\in U'$. So making
$U',R'$ smaller we can take $\Phi:U'\hookra W$ to be an embedding.

Since $(S,V,g,j)$ is a critical chart, $(T,W,h,k)$ is one too. Also
$\Phi:U'\hookra W$, $\Psi:V'\hookra W$ are embeddings, and
$f'=h\ci\Phi$, $g'=h\ci\Psi$, $\Phi\ci i'=k\vert_{R'}$, $\Psi\ci
j'=k\vert_{S'}$ follow from \eq{dc5eq6}--\eq{dc5eq7} and $\Th\ci
i'=j\vert_{R'}$. Therefore $\Phi:(R',U',f',i')\hookra (T,W,h,k)$,
$\Psi:(S',V',g',j')\hookra (T,W,h,k)$ are embeddings of critical
charts with $x\in R'\cap S'$, proving Theorem \ref{dc2thm2} in the
complex analytic case.

For the algebraic case, with $(X,s)$ an algebraic d-critical locus
over $\K$, and $(R,U,f,i),(S,V,g,j)$ critical charts, we would like
to follow the above method, but there is a problem with the first
step: if $V$ is a general smooth $\K$-scheme, we may not be able to
choose a Zariski open $i(x)\in U'\subseteq U$ and a morphism
$\Th:U'\ra V$ such that $\Th\ci i'=j\vert_{R'}:R'\ra V$. However,
this is valid if $V$ is Zariski open in an affine space~$\bA^m$.

So we modify the method above as follows: first we apply Proposition
\ref{dc2prop6}, proved in \S\ref{dc51}, to get a subchart
$(S',V',g',j')\subseteq(S,V,g,j)$ with $x\in S'$ and an embedding
$\Xi:(S',V',g',j')\hookra(\ti S,\ti V,\ti g,\ti\jmath)$ for $(\ti
S,\ti V,\ti g,\ti\jmath)$ a critical chart on $(X,s)$ with $\ti
V\subseteq\bA^m$ Zariski open. Thus we may choose Zariski open
$i(x)\in U'\subseteq U$ and a morphism $\Th:U'\ra\ti V$ with $\Th\ci
i'=\ti\jmath\vert_{R'}:R'\ra W$. Then we follow the rest of the
proof above with $(\ti S,\ti V,\ti g,\ti\jmath),\bA^n$ in place of
$(S,V,g,j),\C^n$, so that $W=\ti V\t\bA^{2n}$, except that we define
$\Psi:V'\hookra W$ by $\Psi(v)=\bigl(\Xi(v), (0,\ldots,0)\bigr)$. We
leave the details to the reader.

\subsection{Proof of Proposition \ref{dc2prop7}}
\label{dc53}

Choose holomorphic coordinates $(\dot y_1,\ldots,\dot y_m, \dot
z_1,\ldots,\dot z_n)$ on an open neighbourhood $\dot V$ of $j(x)$ in
$V$, where $\dim U=m$ and $\dim V=m+n$, such that
$j(x)=(0,\ldots,0)$ and $\Phi(U)\cap\dot V$ is the submanifold $\dot
z_1=\cdots=\dot z_n=0$ in $\dot V$. Set $\dot U=\Phi^{-1}(\dot V)$
and $\dot x_a=\dot y_a\ci\Phi\vert_{\dot U}$ for $a=1,\ldots,m$.
Then $\dot U$ is an open neighbourhood of $i(x)$ in $U$, and $(\dot
x_1,\ldots,\dot x_m)$ are holomorphic coordinates on $\dot U$ with
$i(x)=(0,\ldots,0)$. Write $f\vert_{\dot U}=\dot f(\dot
x_1,\ldots,\dot x_m)$ and $g\vert_{\dot V}=\dot g(\dot
y_1,\ldots,\dot y_m, \dot z_1,\ldots,\dot z_n)$ as functions of
these coordinates, so that $f=g\ci\Phi$ implies that $\dot f(\dot
y_1,\ldots,\dot y_m)=\dot g(\dot y_1,\ldots,\dot y_m,0,\ldots,0)$
for~$(\dot y_1,\ldots,\dot y_m)\in\dot U$.

Then the ideal $I_{R,U}'=I_{(\d f)}$ is on $\dot U$ the ideal of
holomorphic functions in $(\dot x_1,\ldots,\dot x_m)$ generated by
$\frac{\pd\dot f}{\pd \dot x_a}$ for $a=1,\ldots,m$, and the ideal
$I_{S,V}'=I_{(\d g)}$ is on $\dot V$ the ideal of holomorphic
functions in $(\dot y_1,\ldots,\dot y_m,\dot z_1,\ldots,\dot z_n)$
generated by $\frac{\pd\dot g}{\pd \dot y_a}$ for $a=1,\ldots,m$ and
$\frac{\pd\dot g}{\pd \dot z_b}$ for $b=1,\ldots,n$. Since $\Phi$
maps $U$ to $\dot z_1=\cdots=\dot z_n=0$ and $i(R)$ to
$j(R)\subseteq j(S)$, we have $I_{(\d f)}\cong I_{(\d g)}\vert_{\dot
z_1=\cdots=\dot z_n=0}$, that is,
\begin{align*}
&\bigl(\ts\frac{\pd\dot f}{\pd \dot x_a}(\dot y_1,\ldots,\dot y_m):
a=1,\ldots,m\bigr)=\\
&\bigl(\ts\frac{\pd\dot g}{\pd \dot y_a}(\dot y_1,\ldots,\dot y_m,
0,\ldots,0):a\!=\!1,\ldots,m,\;\> \frac{\pd\dot g}{\pd\dot z_b}
(\dot y_1,\ldots,\dot y_m,0,\ldots,0):b\!=\!1,\ldots,n\bigr).
\end{align*}
As $\frac{\pd\dot g}{\pd\dot y_a}(\dot y_1,\ldots,\dot
y_m,0,\ldots,0)=\frac{\pd\dot f}{\pd\dot x_a}(\dot y_1,\ldots,\dot
y_m)$, this holds provided each $\frac{\pd\dot g}{\pd \dot z_b}(\dot
y_1,\ab\ldots,\ab \dot y_m,\ab 0,\ab\ldots,0)$ lies in
$\bigl(\frac{\pd\dot g}{\pd \dot y_a}(\dot y_1,\ldots,\dot
y_m,0,\ldots,0): a=1,\ldots,m\bigr)$. Thus, making $\dot U,\dot V$
smaller if necessary, we can suppose there exist holomorphic
functions $A_{ab}(\dot y_1,\ldots,\dot y_m)$ on $\dot U$ for
$a=1,\ldots,m$, $b=1,\ldots,n$ such that for each~$b$
\begin{equation*}
\ts\frac{\pd\dot g}{\pd \dot z_b}(\dot y_1,\ldots,\dot y_m,0,
\ldots,0)=\ts\sum_{a=1}^mA_{ab}(\dot y_1,\ldots,\dot y_m)\cdot
\frac{\pd\dot g}{\pd \dot y_a}(\dot y_1,\ldots,\dot y_m,0,\ldots,0).
\end{equation*}

Define holomorphic coordinates $(\ti y_1,\ldots,\ti y_m,\ti
z_1,\ldots,\ti z_n)$ on an open neighbourhood $\ti V$ of $j(x)$ in
$\dot V$ by $\ti y_a=\dot y_a+\sum_{b=1}^nA_{ab}(\dot
y_1,\ldots,\dot y_m)\,\dot z_b$ and $\ti z_b=\dot z_b$. Here $\ti
y_a,\ti z_b$ are defined on all of $\dot V$, but they need only be a
coordinate system near $j(x)$ in $\dot V$ where $\dot
z_1,\ldots,\dot z_n$ are small, so we shrink $\dot V$ to $\ti
V\subseteq\dot V$. We also write $\ti U=\Phi^{-1}(\ti
V)\subseteq\dot U$ and $\ti x_a=\dot x_a\vert_{\ti U}$. Then $\ti
y_a\ci\Phi\vert_{\ti U}=\ti x_a$ as $\dot z_b \ci\Phi\vert_{\ti
U}=0$. Making $\ti V$ smaller if necessary we can suppose that if
$(\ti y_1,\ldots,\ti y_m,\ti z_1,\ldots,\ti z_n)\in\ti V$ then $(\ti
y_1,\ldots,\ti y_m)\in\dot U$, using the coordinates $(\dot
x_1,\ldots,\dot x_m)$ on~$\dot U$.

Write $g\vert_{\ti V}=\ti g(\ti y_1,\ldots,\ab\ti y_m,\ab\ti
z_1,\ab\ldots,\ti z_n)$ using these coordinates. Then
\begin{align*}
&\ts\frac{\pd\ti g}{\pd\ti z_b}(\ti y_1,\ldots,\ti
y_m,0,\ldots,0)\\
&=\ts\sum_{a=1}^m\frac{\pd\dot g}{\pd \dot y_a}(\dot y_1,\ldots,
\dot y_m,0,\ldots,0)\cdot\frac{\pd \dot y_a}{\pd\ti z_b}
+\sum_{c=1}^n\frac{\pd\dot g}{\pd \dot z_c}(\dot y_1,\ldots,
\dot y_m,0,\ldots,0)\cdot\frac{\pd\dot z_c}{\pd\ti z_b}\\
&=\ts\sum_{a=1}^m\frac{\pd\dot g}{\pd \dot y_a}(\dot y_1,
\ldots,\dot y_m,0)\cdot(-A_{ab}(\dot y_1, \ldots,\dot y_m)) \\
&\quad\ts+\sum_{c=1}^n\bigl(\sum_{a=1}^mA_{ac}(\dot y_1,\ldots,
\dot y_m)\cdot\frac{\pd\dot g}{\pd \dot y_a}(\dot y_1,\ldots,
\dot y_m,0,\ldots,0)\bigr)\cdot\de_{bc}=0.
\end{align*}
So $\ti g(\ti y_1,\ldots,\ti y_m,0,\ldots,0)=\dot f(\ti
y_1,\ldots,\ti y_m)$ and $\frac{\pd\ti g}{\pd\ti z_b}(\ti
y_1,\ldots,\ti y_m,0,\ldots,0)=0$ for $b=1,\ldots,n$ in the new
coordinates $(\ti y_1,\ldots,\ti y_m,\ti z_1,\ldots,\ti z_n)$.

Consider the holomorphic function $h:\ti V\ra\C$ given by
\e
h(\ti y_1,\ldots,\ti y_m,\ti z_1,\ldots,\ti z_n)=\ti g(\ti
y_1,\ldots,\ti y_m,\ti z_1,\ldots,\ti z_n)-\dot f(\ti y_1,\ldots,\ti
y_m).
\label{dc5eq8}
\e
It satisfies $h(\ti y_1,\ldots,\ti y_m,0,\ldots,0)=0$ and $\frac{\pd
h}{\pd\ti z_b}(\ti y_1,\ldots,\ti y_m,0,\ldots,0)=0$ for all
$b=1,\ldots,n$ and $(\ti y_1,\ldots,\ti y_m)\in\ti U$, so $h$ lies
in the ideal $(\ti z_1,\ldots,\ti z_n)^2$. Thus, making $\ti V,\ti
U$ smaller if necessary, we may write
\e
\begin{split}
h(\ti y_1,\ldots,\ti y_m,\ti z_1,\ldots,\ti
z_n)=\ts\sum_{b,c=1}^n\ti z_b\ti z_c\,Q_{bc}(\ti y_1,\ldots,\ti
y_m,\ti z_1,\ldots,\ti z_n),
\end{split}
\label{dc5eq9}
\e
for some (nonunique) holomorphic functions $Q_{bc}:\ti V\ra\C$
with~$Q_{bc}=Q_{cb}$.

Now $\Crit(\ti g)=\Crit(g)\cap\ti V=j(S)\cap\ti V$ lies in $\Phi(\ti
U)=\{\ti z_1=\cdots=\ti z_n=0\}$ as complex analytic subspaces of
$\ti V$. Therefore $\ti z_1,\ldots,\ti z_n$ lie in the ideal $I_{(\d
\ti g)}$ generated by $\frac{\pd\ti g}{\pd\ti y_a},\frac{\pd\ti
g}{\pd\ti z_b}$, so making $\ti U,\ti V$ smaller, for each
$d=1,\ldots,n$ there exist holomorphic functions $B_{ad}(\ti
y_1,\ldots,\ti y_m,\ti z_1,\ldots,\ti z_n),C_{bd}(\ti y_1,\ldots,\ti
y_m,\ab\ti z_1,\ab\ldots,\ab\ti z_n)$ on $\ti V$ such that using
\eq{dc5eq8}--\eq{dc5eq9} we have
\ea
\ti z_d&=\ts\sum\limits_{a=1}^mB_{ad}\cdot\frac{\pd\ti g}{\pd\ti
y_a}+\sum\limits_{b=1}^nC_{bd}\cdot\frac{\pd\ti g}{\pd\ti z_b}
\nonumber\\
&=\ts\sum\limits_{a=1}^mB_{ad}\cdot\Bigl[\frac{\pd\dot f}{\pd \dot
x_a}(\ti y_1,\ldots,\ti
y_m)+\sum\limits_{a=1}^m\sum\limits_{b,c=1}^n\ti z_b\ti z_c\frac{\pd
Q_{bc}}{\pd\ti y_a}(\ti y_1,\ldots,\ti y_m,\ti z_1,\ldots,\ti
z_n)\Bigr]
\nonumber\\
&\quad +2\ts\sum\limits_{b,c=1}^nC_{bd}\cdot\ti z_cQ_{bc}(\ti
y_1,\ldots,\ti y_m,\ti z_1,\ldots,\ti z_n)
\nonumber\\
&\quad+\ts\sum\limits_{b,c,e=1}^nC_{ed}\cdot\ti z_b\ti z_c\frac{\pd
Q_{bc}}{\pd\ti z_e}(\ti y_1,\ldots,\ti y_m,\ti z_1,\ldots,\ti z_n).
\label{dc5eq10}
\ea

Apply $\frac{\pd}{\pd\ti z_c}$ to \eq{dc5eq10} and restrict it to
the point $(0,\ldots,0)=j(x)$, noting that $\frac{\pd\dot f}{\pd
\dot x_a}(0,\ldots,0)=0$. This yields
\begin{equation*}
\de_{cd}=2\ts\sum_{b=1}^nC_{bd}(0,\ldots,0)\cdot
Q_{bc}(0,\ldots,0), \quad\text{for all $c,d=1,\ldots,n$.}
\end{equation*}
Hence the symmetric matrix $\bigl(Q_{bc}(0,\ldots,0)
\bigr){}_{b,c=1}^n$ is invertible. Thus, by applying an element of
$\GL(n,\C)$ to the coordinates $(\ti z_1,\ldots,\ti z_n)$ we can
suppose that $Q_{bc}(0,\ldots,0)=\de_{bc}$ for~$b,c=1,\ldots,n$.

We now define new holomorphic coordinates
$(y_1,\ldots,y_m,z_1,\ldots,z_n)$ on an open neighbourhood $V'$ of
$j(x)$ in $\ti V$, and write
$g'(y_1,\ldots,y_m,z_1,\ldots,z_n)=g\vert_{V'}$ for $g$ as a
function of these new coordinates and $U'=\Phi^{-1}(V')$, such that:
\begin{itemize}
\setlength{\itemsep}{0pt}
\setlength{\parsep}{0pt}
\item[(a)] $y_a=\ti y_a$ for $a=1,\ldots,m$.
\item[(b)] $\frac{\pd z_b}{\pd\ti z_c}=\de_{bc}$ at $j(x)$
for $b,c=1,\ldots,n$.
\item[(c)] $\Phi(U')$ is the submanifold $z_1=\cdots=z_n=0$ in
$V'$.
\item[(d)] $g'(y_1,\ldots,y_m,z_1,\ldots,z_n)=\dot f(y_1,\ldots,y_m)+
z_1^2+\cdots+z_n^2$.
\end{itemize}
We define $z_b$ by reverse induction on $b=n,n-1,\ldots,1$.

For the first step, as $Q_{nn}=1$ at $j(x)=(0,\ldots,0)$, we may
restrict to a small open neighbourhood $V'$ of $j(x)$ in $\ti V$ on
which $Q_{nn}(\ti y_1,\ldots,\ti y_m,\ti z_1,\ldots,\ti z_n)$ is
invertible and has a square root $Q_{nn}^{1/2}$. Rewrite \eq{dc5eq9}
on $V'$ as
\begin{gather}
\ti g(\ti y_1,\ldots,\ti y_m,\ti z_1,\ldots,\ti z_n)=\dot f(\ti
y_1,\ldots,\ti y_m)\ts+\sum_{b,c=1}^{n-1}\ti z_b\ti
z_c\Bigl[Q_{bc}-Q_{nn}^{-1}Q_{bn}Q_{cn}\Bigr]
\nonumber\\
+\Bigl[Q_{nn}^{1/2}\ti z_n+\ts\sum_{b=1}^{n-1}Q_{nn}^{-1/2}Q_{bn}\ti
z_b\Bigr]^2=\dot f(\ti y_1,\ldots,\ti y_m)\ts+\sum_{b,c=1}^{n-1}\ti
z_b\ti z_c\hat Q_{bc}+z_n^2,
\nonumber\\
\text{where}\;\> \hat Q_{bc}=Q_{bc}-Q_{nn}^{-1}Q_{bn}Q_{cn},\;\>
z_n=Q_{nn}^{1/2}\ti z_n+\ts\sum_{b=1}^{n-1}Q_{nn}^{-1/2}Q_{bn}\ti
z_b.
\label{dc5eq11}
\end{gather}
Note that $\hat Q_{bc}(0,\ldots,0)=\de_{bc}$ for $b,c=1,\ldots,n-1$,
as $Q_{bc}(0,\ldots,0)=\de_{bc}$. For the second step, making $V'$
smaller so that $\hat Q_{n-1\,n-1}$ is invertible and has a square
root $\hat Q_{n-1\,n-1}^{1/2}$ on $V'$, we have
\begin{align*}
\ti g(\ti y_1,\ldots,\ti y_m,\ti z_1,\ldots,\ti z_n)&=f(\ti
y_1,\ldots,\ti y_m)\ts+\sum_{b,c=1}^{n-2}\ti z_b\ti z_c\check
Q_{bc}+z_{n-1}^2+z_n^2,\\
\text{where}\qquad \check Q_{bc}&=\hat Q_{bc}-\hat
Q_{n-1\,n-1}^{-1}\hat Q_{b\,n-1}\hat Q_{c\,n-1}\\
\text{and}\qquad z_{n-1}&=\hat Q_{n-1\,n-1}^{1/2}\ti
z_{n-1}+\ts\sum_{b=1}^{n-2}\hat Q_{n-1\,n-1}^{-1/2}\hat
Q_{b\,n-1}\ti z_b.
\end{align*}
Continuing in this way, we define $j(x)\in V'\subseteq\ti V$ and
holomorphic functions $y_1,\ldots,y_m,z_1,\ab\ldots,\ab z_n:V'\ra\C$
satisfying (a)--(d) above.

Parts (a)--(c) imply that $(y_1,\ldots,y_m,\ab z_1,\ab\ldots,\ab
z_n)$ are a coordinate system near $j(x)$ in $V'$, so making $V'$
smaller, we can suppose they are coordinates on $V'$. Define
$\al:V'\ra U$ and $\be:V'\ra\C^n$ by
$\al:(y_1,\ldots,y_m,z_1,\ldots,z_n) \mapsto(y_1,\ldots,y_m)$ and
$\be:(y_1,\ldots,y_m,z_1,\ldots,z_n) \mapsto(z_1,\ldots,z_n)$, using
coordinates $(y_1,\ab\ldots,\ab y_m,z_1,\ldots,z_n)$ on $V'$ and
$(\dot x_1,\ldots,\dot x_m)$ on $\dot U$. Proposition \ref{dc2prop7}
then follows from (a),(c),(d) above.

\subsection{Proof of Proposition \ref{dc2prop8}}
\label{dc54}

We will adapt \S\ref{dc53} to the algebraic context. The first part,
until just before \eq{dc5eq8}, works with $i(x)\in\ti U\subseteq
\dot U\subseteq U$ and $j(x)\in\ti V\subseteq\dot V\subseteq V$
Zariski open and $(\dot x_1,\ldots,\dot x_m):\dot U\ra\bA^n$, $(\dot
y_1,\ldots,\dot y_m, \dot z_1,\ldots,\dot z_n):\dot V\ra\bA^{m+n}$,
$(\ti y_1,\ldots,\ab\ti y_m,\ab \ti z_1,\ab\ldots,\ab\ti z_n):\ti
V\ra\bA^{m+n}$ \'etale coordinates on $\dot U,\dot V,\ti V$. Note
that $f,g$ are not functions of $\dot x_a$ or $\dot y_a,\dot z_b$ or
$\ti y_a,\ti z_b$, except in an \'etale sense, so we cannot
rigorously write $f\vert_{\dot U}=\dot f(\dot x_1,\ldots,\dot x_m)$,
and so on. Nonetheless, the partial derivatives $\frac{\pd f}{\pd
\dot x_a},\ab \frac{\pd g}{\pd\dot y_a},\ab\frac{\pd g}{\pd \dot
z_b},\ab\frac{\pd g}{\pd\ti y_a},\frac{\pd g}{\pd\ti z_b}$ are all
well defined on~$\dot U,\dot V,\ti V$.

In this way, we construct \'etale coordinates $(\ti y_1,\ldots,\ti
y_m, \ti z_1,\ldots,\ti z_n)$ on $\ti V\subseteq V$ such that
$\Phi(\ti U)$ is the smooth $\K$-subscheme $\ti z_1=\cdots=\ti
z_n=0$ in $\ti V$, and $\frac{\pd g}{\pd\ti z_b}\vert_{\Phi(\ti
U)}=0$ for $b=1,\ldots,n$. Now \eq{dc5eq8} does not make sense on
$\ti V$, since we cannot extend $f$ from $\ti U$ to $\ti V$ as $f$
is not a function of $(\ti y_1,\ldots,\ti y_m)$. Instead, we form
the Cartesian square of smooth $\K$-schemes and \'etale morphisms
\e
\begin{gathered}
\xymatrix@C=130pt@R=17pt{ *+[r]{\check V=(\ti U\t\bA^n)
\t_{\bA^{m+n}}\ti V} \ar[r]_(0.6){\check\jmath}
\ar[d]^{\check\al\t\check\be} & *+[l]{\ti V}
\ar[d]_{(\ti y_1,\ldots,\ti y_m,\ti z_1,\ldots,\ti z_n)} \\
*+[r]{\ti U\t\bA^n} \ar[r]^(0.4){(\ti x_1,\ldots,\ti x_m)\t
\id_{\bA^n}} & *+[l]{\bA^{m+n}=\bA^m\t\bA^n.\!\!{}}}
\end{gathered}
\label{dc5eq12}
\e
There is a unique $\check v\in\check V$ with $\check\al(\check
v)=i(x)$, $\check\be(\check v)=(0,\ldots,0)$ and
$\check\jmath(\check v)=j(x)$. Then $\check\al\t\check\be$ and
$\check\jmath$ are \'etale, so we can regard $\check\jmath:\check
V\ra\ti V\subseteq V$ as an \'etale open set in $V$. Define \'etale
coordinates $(\check y_1,\ldots,\check y_m,\check z_1,\ldots,\check
z_n):\check V\ra\bA^{m+n}$ by $\check y_a=\ti y_a\ci\check\jmath$,
$\check z_b=\ti z_b\ci\check\jmath$, and define $\check f,\check
g:\check V\ra\bA^1$ by $\check f=f\ci\check\al$ and $\check
g=g\ci\check\jmath$. The analogue of \eq{dc5eq8} is now $\check
h=\check g-\check f:\check V\ra\bA^1$. The previous argument now
shows that on the smooth $\K$-subscheme $\check U\in\check V$
defined by $\check z_1=\cdots=\check z_n=0$ we have $\check
h\vert_{\check U}=0$ and $\frac{\pd\check h}{\pd\check
z_b}\vert_{\check U}=0$ for $b=1,\ldots,n$. Therefore $\check h$
lies in the ideal $(\check z_1,\ldots,\check z_n)^2$ on $\check V$.
So making $\ti U,\check U,\ti V,\check V$ smaller, we may write
$\check h=\ts\sum_{b,c=1}^n\check z_b\check z_cQ_{bc}$, the analogue
of \eq{dc5eq9}, for some $Q_{bc}:\check V\ra\bA^1$ with
$Q_{bc}=Q_{cb}$, where the last part requires $\mathop{\rm
char}\K\ne 2$, and $Q_{bc}(0,\ldots,0)=\de_{bc}$
becomes~$Q_{bc}(\check v)=\de_{bc}$.

The last part of \S\ref{dc53}, defining the coordinates
$(y_1,\ldots,y_m,z_1,\ldots,z_n)$, involves taking square roots
$Q_{nn}^{1/2},\hat Q_{n-1\,n-1}^{-1/2},\ldots.$ These generally will
not exist on $\check V$ or Zariski open subsets of $\check V$, but
they will exist on \'etale open subsets of $\check V$, noting that
$\K$ is algebraically closed, so that square roots exist in $\K$. So
we can construct an \'etale open neighbourhood $\jmath':V'\ra\check
V$ of $\check v$ in $\check V$ on which we define \'etale
coordinates $(y_1,\ldots,y_m,z_1,\ldots,z_n):V'\ra\bA^{m+n}$
satisfying the analogues of (a)--(d) in \S\ref{dc53}, where (d)
becomes $\check g\ci\jmath'=\check f\ci\jmath'+z_1^2+\cdots+z_n^2$.

Set $U'=\{v'\in V':z_1(v')=\cdots=z_n(v')=0\}$, which is a smooth
$\K$-subscheme of $V'$ as $(y_1,\ldots,y_m,z_1,\ldots,z_n)$ are
\'etale coordinates on $V'$. Define $\io:U'\ra U$, $\jmath:V'\ra V$,
$\Phi':U'\ra V'$, $\al:V'\ra U$, and $\be:V'\ra\bA^n$ by
$\io=\check\al \ci\jmath'\vert_{U'}$,
$\jmath=\check\jmath\ci\jmath'$, $\Phi'=\id_{U'}$, $\al=\check\al
\ci\jmath'$ and $\be=(z_1,\ldots,z_n)$. As $\jmath':V'\ra\check V$
is an \'etale open neighbourhood of $\check v$ in $\check V$, there
exists $u'\in V'$ with $\jmath'(u')=\check v$, and $z_b(u')=0$ for
$b=1,\ldots,n$ as $\check z_b(\check v)=\ti
z_b\ci\check\jmath(\check v)=\ti z_b(j(x))=0$, so $u'\in U$ with
$\io(u')=\check\al \ci\jmath'(u')=\check\al(\check v)=i(x)$. Also
$\io,\jmath,\al\t\be$ are \'etale as
$\check\al\t\check\be,\check\jmath,\jmath'$ are.

To see that $\Phi\ci\io=\jmath\ci\Phi'$, note that
\begin{align*}
&(\ti y_1,\ldots,\ti y_m,\ti z_1,\ldots,\ti z_n)\ci\Phi\ci\io=
(\ti x_1,\ldots,\ti x_m,0,\ldots,0)\ci\check\al \ci\jmath'\vert_{U'}\\
&=\bigl((\ti x_1,\ldots,\ti x_m)\t\id_{\bA^n}\bigr)\ci(\check\al\t
\check\be)\ci\jmath'\vert_{U'}\\
&=(\ti y_1,\ldots,\ti y_m,\ti z_1,\ldots,\ti z_n)\ci\check\jmath
\ci\jmath'\vert_{U'}\ci\id_{U'}=(\ti y_1,\ldots,\ti y_m,
\ti z_1,\ldots,\ti z_n)\ci\jmath\ci\Phi',
\end{align*}
using the definitions and \eq{dc5eq12}. Since $\Phi\ci\io(u')=j(x)=
\jmath\ci\Phi'(u')$ and $(\ti y_1,\ldots,\ab\ti y_m,\ab\ti
z_1,\ab\ldots,\ti z_n)$ are \'etale coordinates on $\ti V\subseteq
V$, this implies that $\Phi\ci\io=\jmath\ci\Phi'$ near $u'$ in $U'$,
so making $U'$ smaller if necessary we have
$\Phi\ci\io=\jmath\ci\Phi'$. The equations $\al\ci\Phi'=\io,$
$\be\ci\Phi'=0$ are immediate, and
$g\ci\jmath=f\ci\al+(z_1^2+\cdots+z_n^2)\ci\be$ follows from $\check
g\ci\jmath'=\check f\ci\jmath'+z_1^2+\cdots+z_n^2$. This proves
Proposition~\ref{dc2prop8}.

\subsection{Proof of Proposition \ref{dc2prop9}}
\label{dc55}

For Proposition \ref{dc2prop9}, we first follow the proof of
Proposition \ref{dc2prop8} in \S\ref{dc54} until immediately before
the choice of $\jmath':V'\ra\check V$, so we have \'etale
coordinates $(\ti y_1,\ldots,\ti y_m, \ti z_1,\ldots,\ti z_n)$ on
$j(x)\in \ti V\subseteq V$ and $(\ti x_1,\ldots,\ti x_m)$ on
$i(x)\in \ti U=\Phi^{-1}(\ti V)\subseteq U$ with $\ti x_a=\ti
y_a\ci\Phi\vert_{\smash{\ti U}}$, $\frac{\pd g}{\pd\ti
z_b}\vert_{\Phi(\ti U)}=0$, and a Cartesian square \eq{dc5eq12} with
\'etale coordinates $(\check y_1,\ldots,\check y_m,\check
z_1,\ldots,\check z_n)$ on $\check V$ and $\check v\in\check V$ with
$\check\al(\check v)=i(x)$, $\check\be(\check v)=(0,\ldots,0)$ and
$\check\jmath(\check v)=j(x)$, and functions $\check f,\check
g,\check h,Q_{bc}:\check V\ra\bA^1$ with $\check f=f\ci\check\al$,
$\check g=g\ci\check\jmath$, and $\check h=\check g-\check f
=\ts\sum_{b,c=1}^n\check z_b\check z_cQ_{bc}$, with~$Q_{bc}(\check
v)=\de_{bc}$.

We have morphisms $\id\t 0:\ti U\ra \ti U\t\bA^n$,
$\Phi\vert_{\smash{\ti U}}:\ti U\ra\ti V$ with $\bigl((\ti
x_1,\ldots,\ti x_m)\t\id_{\bA^n}\bigr)\ci(\id\t 0)=(\ti
y_1,\ldots,\ti y_m, \ti z_1,\ldots,\ti z_n)\ci\Phi\vert_{\smash{\ti
U}}:\ti U\ra\bA^{m+n}$. Thus by the Cartesian property of
\eq{dc5eq12}, there is a unique morphism $\check\Phi:\ti U\ra\check
V$ with $(\check\al\t\check\be)\ci\check\Phi=\id\t 0$ and
$\check\jmath\ci\check\Phi=\Phi\vert_{\smash{\ti U}}$. Also
$\check\Phi(i(x))=\check v$, since $(\id\t
0)(i(x))=(i(x),0)=(\check\al\t\check\be)(\check v)$
and~$\Phi\vert_{\smash{\ti U}}(i(x))=j(x)=\check\jmath(\check v)$.

We now modify the inductive procedure in \S\ref{dc53}--\S\ref{dc54},
to construct a Zariski open neighbourhood $U'$ of $i(x)$ in $\ti
U\subseteq U$, an \'etale open neighbourhood $\jmath':V'\ra\check V$
of $\check v$ in $\check V$ with $\check\al\ci\jmath'(V')\subseteq
U'\subseteq\ti U$, a morphism $\Phi':U'\ra V'$ with
$\jmath'\ci\Phi'=\check\Phi\vert_{U'}$, \'etale coordinates
$(y_1,\ldots,\ab y_m,\ab z_1,\ab\ldots,z_n):V' \ra\bA^{m+n}$, and
regular functions $q_1,\ldots,q_n:U'\ra\bA^1 \sm\{0\}$, such that:
\begin{itemize}
\setlength{\itemsep}{0pt}
\setlength{\parsep}{0pt}
\item[(a)] $y_a=\check y_a\ci\jmath'=\ti y_a\ci\check\jmath\ci\jmath'$
for $a=1,\ldots,m$.
\item[(b)] $\frac{\pd z_b}{\pd z_c'}=\de_{bc}$ at $\Phi'(i(x))$
for $b,c=1,\ldots,n$, where $(y_1',\ldots,y_m',
z_1',\ldots,z_n')=(\check y_1,\ldots,\check y_m,\check
z_1,\ldots,\check z_n)\ci\jmath'$ are \'etale coordinates on
$V'$.
\item[(c)] $\Phi'(U')$ is the submanifold $z_1=\cdots=z_n=0$ in
$V'$.
\item[(d)] $\check h\ci\jmath'=(q_1\ci\check\al\ci\jmath')\cdot z_1^2
+\cdots+(q_n\ci\check\al\ci\jmath')\cdot z_n^2$.
\end{itemize}

In the first step of the induction, as $Q_{nn}(\check v)=1$ and
$\check\Phi(i(x))=\check v$, we can choose a Zariski open
neighbourhood $U'$ of $i(x)$ in $\ti U$ such that $q_n:=Q_{nn}\ci
\check\Phi$ is nonzero on $U'$, so that $q_n:U'\ra\bA^1\sm\{0\}$.
Then $\check\al^{-1}(U')$ is a Zariski open neighbourhood of $\check
v$ in $\check V$, with $\check\Phi(U')\subseteq\check\al^{-1}(U')$
as $\check\al\ci\check\Phi=\id_{\smash{\ti U}}$, and
\begin{equation*}
\check P_n:=\bigl(Q_{nn}\vert_{\check\al^{-1}(U')}\bigr)\cdot
\bigl(q_n^{-1}\ci\check\al
\vert_{\check\al^{-1}(U')}\bigr):\check\al^{-1}(U')\longra\bA^1
\end{equation*}
is a regular function, with $\check P_n\ci\check\Phi\vert_{U'}=1$.

Define $\hat\jmath:\hat V\ra\check P_n^{-1}(\bA^1\sm\{0\})
\subseteq\check V$ to be the \'etale double cover parametrizing
square roots of $\check P_n$ wherever $\check P_n$ is nonzero. Then
$\check P_n\ci\hat\jmath$ has a natural square root $(\check
P_n\ci\hat\jmath)^{1/2}:\hat V\ra\bA^1\sm\{0\}$. Since $\check
P_n\ci\check\Phi\vert_{U'}=1$, there is a unique lift
$\hat\Phi:U'\ra \hat V$ such that
$\hat\jmath\ci\hat\Phi=\check\Phi\vert_{U'}$ and $(\check
P_n\ci\hat\jmath)^{1/2}\ci\hat\Phi=1$.

In an analogue of \eq{dc5eq11}, we may now write
\ea
\check h\ci\hat\jmath&=\ts\sum_{b,c=1}^n(\check z_b\check
z_cQ_{bc})\ci\hat\jmath=\ts\sum_{b,c=1}^{n-1}\bigl(\check z_b\check
z_c\bigl[Q_{bc}- Q_{nn}^{-1}Q_{bn}Q_{cn}\bigr]\bigr)\ci\hat\jmath
\nonumber\\
&+(q_n\ci\check\al\ci\hat\jmath)\cdot \Bigl[(\check
P_n\ci\hat\jmath)^{1/2}\bigl(\check z_n+\ts\sum_{b=1}^{n-1}
Q_{nn}^{-1}Q_{bn}\check z_b\bigr)\ci\hat\jmath\Bigr]^2
\nonumber\\
&=\ts\sum_{b,c=1}^{n-1}\hat z_b\hat z_c\hat Q_{bc}
+(q_n\ci\hat\al)\cdot\hat z_n^2,
\label{dc5eq13}
\ea
where $\hat z_b=\check z_b\ci\hat\jmath$ for $b=1,\ldots,n-1$, $\hat
z_n=(\check P_n\ci\hat\jmath)^{1/2}\bigl(\check z_n
+\ts\sum_{b=1}^{n-1}Q_{nn}^{-1}Q_{bn}\check z_b\bigr)\ci\hat\jmath$,
$\hat Q_{bc}=[Q_{bc}- Q_{nn}^{-1}Q_{bn}Q_{cn}]\ci\hat\jmath$ for
$b,c=1,\ldots,n-1$, and $\hat\al=\check\al\ci\hat\jmath$.

In the second step we define $q_{n-1}:=\hat Q_{n-1\,n-1}\ci
\hat\Phi$, and making $U'$ smaller we can suppose that $q_{n-1}$
maps $U'\ra\bA^1\sm\{0\}$. Then we define
\begin{equation*}
\hat P_{n-1}:=\bigl(\hat Q_{n-1\,n-1}\vert_{\hat\al^{-1}(U')}\bigr)\cdot
\bigl(q_{n-1}^{-1}\ci\hat\al
\vert_{\hat\al^{-1}(U')}\bigr):\hat\al^{-1}(U')\longra\bA^1,
\end{equation*}
so that $\hat P_{n-1}\ci\hat\Phi=1$. Let $\ddot\jmath:\ddot V\ra\hat
P_{n-1}^{-1}(\bA^1\sm\{0\}) \subseteq\hat V$ be the \'etale double
cover parametrizing square roots of $\hat P_{n-1}$ where $\hat
P_{n-1}$ is nonzero. Then $\hat P_{n-1}\ci\ddot\jmath$ has a square
root $(\hat P_{n-1}\ci\ddot\jmath)^{1/2}:\ddot V\ra\bA^1\sm \{0\}$,
and there is a unique lift $\ddot\Phi:U'\ra \ddot V$ such that
$\ddot\jmath\ci\ddot\Phi=\hat\Phi$ and $(\ddot
P_{n-1}\ci\ddot\jmath)^{1/2}\ci\ddot\Phi=1$. As for \eq{dc5eq13} we
have
\begin{equation*}
\check h\ci\hat\jmath\ci\ddot\jmath=\ts\sum_{b,c=1}^{n-2}\ddot
z_b\ddot z_c\ddot Q_{bc} +(q_{n-1}\ci\ddot\al)\cdot\ddot z_{n-1}^2
+(q_n\ci\ddot\al)\cdot\ddot z_n^2,
\end{equation*}
where $\ddot z_{n-1}=(\hat P_{n-1}\ci\ddot\jmath)^{1/2}\bigl(\hat
z_{n-1} +\ts\sum_{b=1}^{n-2}\hat Q_{n-1n-1}^{-1}\hat Q_{bn-1}\hat
z_b\bigr)\ci\ddot\jmath$, $\ddot z_b=\hat z_b\ci\ddot\jmath$ for
$b=1,\ldots,n-2,n$, $\ddot Q_{bc}=[\hat Q_{bc}- \hat
Q_{n-1n-1}^{-1}\hat Q_{bn-1}\hat Q_{cn-1}]\ci\ddot\jmath$ for
$b,c=1,\ldots,n-2$, and $\ddot\al=\hat\al\ci\ddot\jmath$. Continuing
in this way, after $n$ inductive steps we define data
$\jmath',V',\Phi',y_a,z_b,q_b$ satisfying (a)--(d) above. The
important difference with \S\ref{dc54} is that each time we pass to
a further \'etale cover $\hat V,\ddot V,\ldots$ of $\check V$ to
take a square root, we also lift $\check\Phi\vert_{U'}:U'\ra\check
V$ to morphisms $\hat\Phi:U'\ra\hat V,$  $\ddot\Phi:U'\ra\ddot V$,
\ldots, to these \'etale covers, for $U'\subseteq U$ Zariski open.
Proposition \ref{dc2prop9} now follows as in
\S\ref{dc53}--\S\ref{dc54}, with \eq{dc2eq14} coming from (d) above.

\section{Canonical bundles of d-critical loci}
\label{dc6}

We prove the results of \S\ref{dc24}, Propositions \ref{dc2prop10},
\ref{dc2prop11}, \ref{dc2prop12} and Theorem~\ref{dc2thm3}.

\subsection{Proof of Proposition \ref{dc2prop10}}
\label{dc61}

We will prove the complex analytic case of Proposition
\ref{dc2prop10}, involving (a) and Proposition \ref{dc2prop7}. The
algebraic cases with (b),(c) and Propositions \ref{dc2prop8} and
\ref{dc2prop9} are similar.

Let $(X,s)$ be a complex analytic d-critical locus and
$x,U',V',n,\al,\be$ be as in Proposition \ref{dc2prop7}. To see
there exists a unique isomorphism $\hat\be$ satisfying \eq{dc2eq16},
consider the diagram of vector bundles on~$U'$:
\e
\begin{gathered}
\xymatrix@C=12pt@R=15pt{ 0 \ar[r] & TU\vert_{U'} \ar[d]^\id_\cong
\ar[rr]_(0.4){\d\Phi\vert_{U'}} && \Phi^*(TV)\vert_{U'}
\ar[d]^{\Phi\vert_{U'}^*(\d(\al\t\be))}_\cong
\ar[rr]_(0.55){\Pi\vert_{U'}} && N_{\sst
UV}\vert_{U'} \ar@{.>}[d]^{\hat\be^*}_\cong \ar[r] & 0 \\
0 \ar[r] & TU\vert_{U'} \ar[rr]^(0.4){\d(\id\t 0)} &&
T(U\!\t\!\C^n)\vert_{U'\t 0} \ar[rr] && \ban{\frac{\pd}{\pd
z_1},\ldots,\frac{\pd}{\pd z_n}}_{U'} \ar[r] & 0. }\!\!\!\!\!\!\!{}
\end{gathered}
\label{dc6eq1}
\e
Here \eq{dc6eq1} has exact rows, the first two columns are
isomorphisms as $\al\t\be$ is \'etale, and the left hand square
commutes since $\al\ci\Phi\vert_{U'}=\id$ and
$\be\ci\Phi\vert_{U'}=0$. Therefore by exactness there is a unique
isomorphism $\hat\be^*$ making \eq{dc6eq1} commute, and taking duals
shows that \eq{dc2eq16} holds. So \eq{dc2eq17} prescribes $q_{\sst
UV}\vert_{R'}$, for $R'=i^{-1}(U')\subseteq R\subseteq X$. This
$q_{\sst UV}\vert_{R'}$ is a nondegenerate holomorphic quadratic
form on $i^*(N_{\sst UV})\vert_{R'}$, since $\d z_1\ot\d
z_1+\cdots+\d z_n\ot\d z_n$ is a nondegenerate holomorphic quadratic
form on $\ban{\frac{\pd}{\pd z_1},\ldots,\frac{\pd}{\pd
z_n}}{}_{U'}$.

Thus, for each $x\in R$ we can find an open $x\in R'\subseteq R$ and
a given value for the restriction $q_{\sst UV}\vert_{R'}$. Since
such $R'$ form an open cover for $R$, and $S^2i^*(N_{\sst UV}^*)$ is
a sheaf on $R$, these values for $q_{\sst UV}\vert_{R'}$ come from a
unique $q_{\sst UV}\in H^0(S^2i^*(N_{\sst UV}^*))$ if and only if
they agree on overlaps $R'\cap R''$ between different subsets
$R',R''$ in the open cover.

Suppose $x,U',V',\al,\be,R'$ and $x',U'',V'',\al',\be',R''$ are
alternative choices in Proposition \ref{dc2prop7}. Then \eq{dc2eq17}
gives values for $q_{\sst UV}\vert_{R'}$ and $q_{\sst
UV}\vert_{R''}$, which agree on the overlap $R'\cap R''$ if
\e
\begin{split}
i&\vert^*_{R'\cap R''}\bigl[(S^2\hat\be)(\d
z_1\ot\d z_1+\cdots+\d z_n\ot\d z_n)\bigr]\\
&=i\vert^*_{R'\cap R''}\bigl[(S^2\hat\be')(\d z_1\ot\d z_1+\cdots+\d
z_n\ot\d z_n)\bigr].
\end{split}
\label{dc6eq2}
\e
Combining \eq{dc6eq1} for $U',V',\al,\be$ and $U'',V'',\al',\be'$
gives a commutative diagram of vector bundles on $U'\cap U''$, with
exact rows:
\e
\begin{gathered}
\xymatrix@C=12pt@R=21pt{ 0 \ar[r] & TU\vert_{U'} \ar[d]^\id_\cong
\ar[rr]_(0.4){\d(\id\t 0)} && T(U\!\t\!\C^n)\vert_{U'\t 0}
\ar[d]^(0.6){\begin{subarray}{l}\Phi\vert_{U'}^*(\d((\al'\t\be')\\
{}\quad\;\ci(\al\t\be)^{-1}))\end{subarray}}_\cong
\ar[rr]_(0.55){\Pi\vert_{U'}} && \ban{\frac{\pd}{\pd
z_1},\ldots,\frac{\pd}{\pd z_n}}_{U'}
\ar[d]^{\hat\be^{\prime *}\ci\hat\be^{*-1}}_\cong \ar[r] & 0 \\
0 \ar[r] & TU\vert_{U'} \ar[rr]^(0.4){\d(\id\t 0)} &&
T(U\!\t\!\C^n)\vert_{U'\t 0} \ar[rr] && \ban{\frac{\pd}{\pd
z_1},\ldots,\frac{\pd}{\pd z_n}}_{U'} \ar[r] & 0. }
\end{gathered}
\label{dc6eq3}
\e

Here $(\al'\t\be')\ci(\al\t\be)^{-1}$ is a local biholomorphism
$U\t\C^n\ra U\t\C^n$ defined near $(U'\cap U'')\t\{0\}$, which is
the identity on $U\t\{0\}$, and preserves the function $f\boxplus
z_1^2+\cdots+z_n^2:U\t\C^n\ra\C$. So restricting to $\Crit(f\boxplus
z_1^2+\cdots+z_n^2)=\Crit(f)\t\{0\}$ where $\Hess\bigl(f\boxplus
z_1^2+\cdots+z_n^2\bigr)$ is defined, we see that
$\d((\al'\t\be')\ci(\al\t\be)^{-1})\vert_{\Crit(f\vert_{U'\cap
U''})\t\{0\}}$ preserves $\Hess\bigl(f\boxplus
z_1^2+\cdots+z_n^2\bigr)$ in
$H^0\bigl(S^2T^*(U\t\C^n)\vert_{\Crit(f\vert_{U'\cap
U''})\t\{0\}}\bigr)$. As
\begin{equation*}
\Hess\bigl(f\boxplus z_1^2+\cdots+z_n^2\bigr)=\Hess f+\d
z_1\ot\d z_1+\cdots+\d z_n\ot\d z_n,
\end{equation*}
from \eq{dc6eq3} we see that $\hat\be^{\prime
*}\ci\hat\be^{*-1}$ preserves $\d z_1\ot\d
z_1+\cdots+\d z_n\ot\d z_n$ on $\Crit(f\vert_{U'\cap U''})$. That
is,
\begin{align*}
(S^2(\hat\be^{-1}&\ci\hat\be'))(\d z_1\ot\d z_1+
\cdots+\d z_n\ot\d z_n)\vert_{\Crit(f\vert_{U'\cap U''})}\\
&=(\d z_1\ot\d z_1+\cdots+\d z_n\ot\d z_n)
\vert_{\Crit(f\vert_{U'\cap U''})}.
\end{align*}
Composing with $S^2\hat\be$ and applying $i\vert^*_{R'\cap R''}$
gives \eq{dc6eq2}. Hence, there exists a unique, nondegenerate
$q_{\sst UV}$ satisfying Proposition~\ref{dc2prop10}(a).

For the final part of the proposition, suppose
$\Psi:(S,V,g,j)\hookra(T,W,h,k)$ is another embedding, so that
$\Psi\ci\Phi:(R,U,f,i)\hookra(T,W,h,k)$ is also an embedding, and
define $N_{\sst UV},q_{\sst UV}$, $N_{\sst VW},q_{\sst VW}$ and
$N_{\sst UW},q_{\sst UW}$ from $\Phi,\Psi$ and $\Psi\ci\Phi$ as
above. Then existence of unique $\ga_{\sst UVW},\de_{\sst UVW}$
making \eq{dc2eq22} commute is immediate by exactness, and exactness
of the line in \eq{dc2eq22} including $\ga_{\sst UVW},\de_{\sst
UVW}$ also follows easily. So \eq{dc2eq23} is exact on~$R$.

We will first prove the required isomorphism \eq{dc2eq24} exists
locally. Let $x\in R$. Applying Proposition \ref{dc2prop7} to
$\Phi:(R,U,f,i)\hookra(S,V,g,j)$ at $x$ gives $i(x)\in U'\subseteq
U$, $j(x)\in V'\subseteq V$, $\al:V'\ra U$ and $\be:V'\ra\C^m$ for
$m=\dim V-\dim U$, satisfying conditions including
$g\vert_{V'}=f\ci\al+(y_1^2+\cdots+y_m^2)\ci\be$, writing
$(y_1,\ldots,y_m)$ for the coordinates on $\C^m$. Then Proposition
\ref{dc2prop10}(a) gives
\e
\begin{gathered}
\hat\be:\langle\d y_1,\ldots,\d
y_m\rangle_{U'}\,{\buildrel\cong\over\longra}\, N_{\sst
UV}^*\vert_{U'},
\qquad\text{and}\\
q_{\sst UV}\vert_{R'}=i\vert_{R'}^*\bigl[(S^2\hat\be)(\d y_1\ot\d
y_1+\cdots+\d y_m\ot\d y_m)\bigr],
\end{gathered}
\label{dc6eq4}
\e
where $R'=i^{-1}(U')\subseteq R\subseteq X$. Similarly, applying
Proposition \ref{dc2prop7} to $\Psi:(S,V,g,j)\hookra(T,W,h,k)$ at
$x$ gives $j(x)\in V''\subseteq V$, $k(x)\in W''\subseteq V$,
$\al':W''\ra V$ and $\be':W''\ra\C^n$ for $n=\dim W-\dim V$,
satisfying conditions including
$h\vert_{W''}=g\ci\al'+(z_1^2+\cdots+z_n^2)\ci\be'$, writing
$(z_1,\ldots,z_n)$ for the coordinates on $\C^n$, and with
$S''=j^{-1}(V'')$, Proposition \ref{dc2prop10}(a) gives
\e
\begin{gathered}
\hat\be':\langle\d z_1,\ldots,\d
z_n\rangle_{V''}\,{\buildrel\cong\over\longra}\, N_{\sst
VW}^*\vert_{V''},
\qquad\text{and}\\
q_{\sst VW}\vert_{S''}=j\vert_{S''}^*\bigl[(S^2\hat\be')(\d z_1\ot\d
z_1+\cdots+\d z_n\ot\d z_n)\bigr].
\end{gathered}
\label{dc6eq5}
\e

Set $U'''=U'\cap\Phi^{-1}(V'')$ and $W'''=\al^{\prime -1}(V')$,
write $(y_1,\ldots,y_m,z_1,\ldots, z_n)$ for the coordinates on
$\C^{m+n}$, and define $\al'':W'''\ra U$ and $\be'':W'''\ra\C^{m+n}$
by $\al''=\al\ci\al'\vert_{W'''}$ and
$\be''=(\be\ci\al'\vert_{W'''})\t\be'\vert_{W'''}$. Then
\begin{align*}
h\vert_{W'''}&=g\ci\al'\vert_{W'''}+(z_1^2+\cdots+z_n^2)\ci\be'\vert_{W'''}\\
&=f\ci\al\ci\al'\vert_{W'''}+(y_1^2+\cdots+y_m^2)\ci\be\ci\al'\vert_{W'''}+
(z_1^2+\cdots+z_n^2)\ci\be'\vert_{W'''}\\
&=f\ci\al''+(y_1^2+\cdots+y_m^2+z_1^2+\cdots+z_n^2)\ci\be''.
\end{align*}
The other conditions are easy to verify, so
$U''',W''',\al'',\be'',m+n$ are a possible outcome for Proposition
\ref{dc2prop7} applied to $\Psi\ci\Phi:(R,U,f,i) \hookra(T,W,h,k)$
at $x$. Hence Proposition \ref{dc2prop10}(a) with
$R'''=i^{-1}(U''')=R'\cap S''$ gives
\e
\begin{split}
\hat\be'':\langle\d y_1,\ldots,\d y_m,\d z_1,\ldots,\d
z_n\rangle_{W'''}\,&{\buildrel\cong\over\longra}\, N_{\sst
UW}^*\vert_{W'''},
\qquad\text{and}\\
q_{\sst UW}\vert_{R'''}=i\vert_{W'''}^*\bigl[(S^2\hat\be'')( \d
y_1\ot\d y_1&+\cdots+\d y_m\ot\d y_m\\
{}+\d z_1\ot\d z_1&+\cdots+\d z_n\ot\d z_n)\bigr].
\end{split}
\label{dc6eq6}
\e

The isomorphism $i^*(N_{\sst UW})\cong i^*(N_{\sst UV})\op
j^*(N_{\sst VW})\vert_R$ in \eq{dc2eq24} is now clear on the open
subset $x\in R'''=R'\cap S''\subseteq R$, from the isomorphisms
\begin{align*}
i\vert_{R'''}^*(\hat\be)&:\langle\d y_1,\ldots,\d
y_m\rangle_{R'''}\,{\buildrel\cong\over\longra}\, i^*(N_{\sst
UV})\vert_{R'''}^*,
\\
j\vert_{R'''}^*(\hat\be')&:\langle\d z_1,\ldots,\d
z_n\rangle_{R'''}\,{\buildrel\cong\over\longra}\, j^*(N_{\sst
VW})\vert_{R'''}^*,
\\
i\vert_{R'''}^*(\hat\be'')&:\langle\d y_1,\ldots,\d
y_m\rangle_{R'''}\op\langle\d z_1,\ldots,\d
z_n\rangle_{R'''}\,{\buildrel\cong\over\longra}\,
i^*(N_{\sst UW})\vert_{R'''}^*.
\end{align*}
It is easy to see this isomorphism is compatible with \eq{dc2eq23},
and \eq{dc6eq4}--\eq{dc6eq6} imply that the isomorphism induces
equation \eq{dc2eq25} on~$R'''$.

Now isomorphisms \eq{dc2eq24} compatible with \eq{dc2eq23} are in
1-1 correspondence with complementary vector subbundles to
$i^*(\ga_{\sst UVW})\bigl(i^*(N_{\sst UV})\bigr)$ in $i^*(N_{\sst
UW})$. In this case, the complementary vector subbundle is just the
orthogonal subbundle to $i^*(\ga_{\sst UVW})\bigl(i^*(N_{\sst
UV})\bigr)$ using the complex inner product $q_{\sst UW}$ on
$i^*(N_{\sst UW})$. This orthogonal subbundle is complementary
provided $q_{\sst UW}\vert_{i^*(\ga_{\sst UVW})(i^*(N_{\sst UV}))}$
is nondegenerate, which holds as this restriction is isomorphic
to~$q_{\sst UV}$.

Thus, the isomorphism \eq{dc2eq24} exists locally on $R$, and is
unique (even locally) if it exists. So we can glue local choices on
an open cover of $R$ by subsets $R'''$ to get a unique global
isomorphism \eq{dc2eq24} compatible with \eq{dc2eq23} and
\eq{dc2eq25}. This completes the proof of
Proposition~\ref{dc2prop10}.

\subsection{Proof of Proposition \ref{dc2prop11}}
\label{dc62}

Let $\Phi:(R,U,f,i)\hookra(S,V,g,j)$ be an embedding of critical
charts on a complex analytic d-critical locus $(X,s)$. Fix $x\in
R\subseteq X$. As for \eq{dc2eq15}, define the normal $\C$-vector
spaces $N_{\sst XU}\vert_x,N_{\sst XV}\vert_x$ to $X$ in $U,V$ at
$x$ by the exact sequences
\e
\begin{gathered}
\xymatrix@C=20pt@R=7pt{ 0 \ar[r] & T_xX \ar[rr]_(0.4){\d i\vert_x}
&& T_{i(x)}U \ar[rr] && N_{\sst XU}\vert_x \ar[r] & 0, \\
0 \ar[r] & T_xX \ar[rr]^(0.4){\d j\vert_x} && T_{j(x)V} \ar[rr] &&
N_{\sst XV}\vert_x \ar[r] & 0, }
\end{gathered}
\label{dc6eq7}
\e
where $T_xX$ is the Zariski tangent space of $X$ at $x$. Write $\dim
T_xX=l$, $\dim U=l+m$ and $\dim V=l+m+n$, so that $\dim N_{\sst
XU}\vert_x=m$, $\dim N_{\sst XV}\vert_x=m+n$, and $\dim N_{\sst
UV}\vert_{i(x)}=n$. As for \eq{dc2eq26}, equation \eq{dc6eq7}
induces isomorphisms
\e
\begin{split}
\rho_{\sst XU}\vert_x&:\La^lT_x^*X\ot\La^mN_{\sst XU}^*\vert_x
\longra\La^{l+m}T_{i(x)}^*U=K_U\vert_{i(x)},\\
\rho_{\sst XV}\vert_x&:\La^lT_x^*X\ot\La^{m+n}N_{\sst XV}^*\vert_x
\longra\La^{l+m+n}T_{j(x)}^*V=K_V\vert_{j(x)}.
\end{split}
\label{dc6eq8}
\e

We also have a commutative diagram with exact rows:
\e
\begin{gathered}
\xymatrix@C=23pt@R=15pt{ 0\ar[r] & T_xX \ar[d]_{\id} \ar[r]_(0.4){\d
i\vert_x} & T_{i(x)}U \ar[d]^{\d\Phi\vert_{i(x)}}
\ar[rr]_(0.55){\Hess_{i(x)}f} && T_{i(x)}^*U \ar[r]_{\d
i\vert_x^*} \ar@{.>}[d]^{} & T_x^*X \ar[d]^{\id} \ar[r] & 0 \\
0\ar[r] & T_xX \ar[r]^(0.4){\d j\vert_x} & T_{j(x)}V
\ar[rr]^(0.55){\Hess_{j(x)}g} && T_{j(x)}^*V \ar[r]^{\d j\vert_x^*}
& T_x^*X \ar[r] & 0. }\!\!\!\!\!{}
\end{gathered}
\label{dc6eq9}
\e
Since $T_xX=\Ker(\Hess_{i(x)}f)$, by \eq{dc6eq7} $\Hess_{i(x)}f$ is
the pullback to $T_{i(x)}U$ of a nondegenerate quadratic form
$\Hess_{\smash{i(x)}}'f$ on $N_{\sst XU}\vert_x$. Then
$\det(\Hess_{\smash{i(x)}}'f)$ is a nonzero element of $\La^mN_{\sst
XU}^*\vert_x^{\ot^2}$. Similarly, $\Hess_{j(x)}g$ is the pullback to
$T_{j(x)}V$ of $\Hess_{\smash{j(x)}}'g$ on $N_{\sst XV}\vert_x$, and
$0\ne\det(\Hess_{\smash{j(x)}}'g)\in \La^{m+n}N_{\sst
XV}^*\vert_x^{\ot^2}$.

In a similar way to \eq{dc2eq23}, there is a natural exact sequence
\e
\begin{gathered}
\xymatrix@C=20pt{ 0 \ar[r] & N_{\sst XU}\vert_x \ar[rr] && N_{\sst
XV}\vert_x \ar[rr] && N_{\sst UV}\vert_{i(x)} \ar[r] & 0, }
\end{gathered}
\label{dc6eq10}
\e
and as for \eq{dc2eq24}, there is a unique isomorphism $N_{\sst
XV}\vert_x\cong N_{\sst XU}\vert_x\op N_{\sst UV}\vert_{i(x)}$
compatible with \eq{dc6eq10} and identifying $\Hess_{j(x)}'g$ with
$\Hess_{i(x)}'f\op q_{\sst UV}\vert_x\op 0$ as in \eq{dc2eq25}. From
this, we deduce that the following diagram commutes:
\e
\begin{gathered}
\xymatrix@C=145pt@R=23pt{
*+[r]{(\La^{l+m}T_{i(x)}^*U)^{\ot^2}\!=\!K_U^{\ot^2}\vert_{i(x)}}
\ar[dr]_(0.45){J_\Phi\vert_x} \ar[r]_(0.55){\id\ot\det(q_{\sst
UV})\vert_x} &
*+[l]{K_U^{\ot^2}\vert_{i(x)}\ot\La^nN_{\sst
UV}^*\vert_{i(x)}^{\ot^2}}
\ar[d]_(0.45){\rho_{\sst UV}^{\ot^2}\vert_x} \\
*+[r]{(\La^lT_x^*X)^{\ot^2}\ot\La^mN_{\sst
XU}^*\vert_x^{\ot^2} } \ar[u]_{\rho_{\sst XU}^{\ot^2}\vert_x} &
*+[l]{(\La^{l+m+n}T_{j(x)}^*V)^{\ot^2}=K_V^{\ot^2}\vert_{j(x)}} \\
*+[r]{(\La^lT_x^*X)^{\ot^2}} \ar[u]_{\id\ot\det(\Hess_{\smash{i(x)}}'f)}
\ar[r]^(0.4){\id\ot\det(\Hess_{\smash{j(x)}}'g)} &
*+[l]{(\La^lT_x^*X)^{\ot^2}\ot\La^{m+n}N_{\sst XV}^*\vert_x^{\ot^2}.\!\!{}}
\ar[u]^{\rho_{\sst XV}^{\ot^2}\vert_x} }
\end{gathered}
\label{dc6eq11}
\e

Here the upper right triangle is the restriction of \eq{dc2eq27} to
$x$, and depends on $\Phi$. But the rest of the diagram depends on
$x,(R,U,f,i),(S,V,g,j)$ but not on $\Phi$. So \eq{dc6eq11} a diagram
of commuting isomorphisms implies that $J_\Phi\vert_x$ is
independent of the choice of $\Phi$. Thus, if $\Phi,\ti\Phi:
(R,U,f,i)\hookra(S,V,g,j)$ are embeddings of critical charts, then
$J_\Phi\vert_x=J_{\smash{\ti\Phi}}\vert_x$ for all $x\in R^\red$. As
$R^\red$ is a reduced complex analytic space, this implies that
$J_\Phi= J_{\smash{\ti\Phi}}$, as we want.

To prove \eq{dc2eq28}, consider the commutative diagram
\begin{equation*}
\text{\begin{footnotesize}$\displaystyle\xymatrix@!0@C=130pt@R=50pt{
&& *+[r]{\substack{\ts i^*\bigl(K_U^{\ot^2}\bigr)\ot\\
\ts i^*\bigl(\La^{\rm top}N_{\sst
UW}^*\bigr){}^{\ot^2}\vert_{R^\red}}} \ar[d]_(0.4){\text{from
\eq{dc2eq23}}} \ar@<4ex>@/^2pc/[ddd]^{\rho_{\sst UW}^{\ot^2}} \\
*+[r]{i^*\bigl(K_U^{\ot^2}\bigr)\vert_{R^\red}}
\ar[dr]_{J_\Phi} \ar@/_2.5pc/[ddrr]_{J_{\Psi\ci\Phi}}
\ar[r]^(0.6){\id\ot\det(q_{\sst UV})}
\ar@/^1pc/[urr]^(0.5){\id\ot\det(q_{\sst UW})\,\,\,\,\,\,\,{}} &
{\substack{\ts i^*\bigl(K_U^{\ot^2}\bigr)\ot\\
\ts i^*\bigl(\La^{\rm top}N_{\sst
UV}^*\bigr){}^{\ot^2}\vert_{R^\red}}} \ar[d]^(0.55){\rho_{\sst
UV}^{\ot^2}} \ar[r]^(0.4){\id\ot\det(q_{\sst VW})} &
*+[l]{\substack{\ts u^*\bigl(K_U^{\ot^2}\bigr)\ot\\
\ts i^*\bigl(\La^{\rm top}N_{\sst UV}^*\bigr){}^{\ot^2}\ot\\
\ts i^*\bigl(\La^{\rm top}N_{\sst VW}^*\bigr){}^{\ot^2}
\vert_{R^\red}}\!\!\!\!\!\!\!\!\!\!\!\!{}} \ar[d]_(0.55){\rho_{\sst
UV}^{\ot^2}\ot\id} \\
& i^*\bigl(K_V^{\ot^2}\bigr)\vert_{R^\red}
\ar[dr]_{J_\Psi\vert_{R^\red}} \ar[r]^(0.35){\id\ot\det(q_{\sst
VW})}
& *+[l]{\substack{\ts i^*\bigl(K_V^{\ot^2}\bigr)\ot\\
\ts i^*\bigl(\La^{\rm top}N_{\sst
VW}^*\bigr){}^{\ot^2}\vert_{R^\red}}\!\!\!\!\!\!\!\!\!\!\!\!{}}
\ar[d]_(0.45){\rho_{\sst VW}^{\ot^2}} \\
&& *+[l]{k^*\bigl(K_W^{\ot^2}\bigr)\vert_{R^\red},
\!\!\!\!\!\!\!\!\!\!\!\!{}}}$\end{footnotesize}}
\end{equation*}
which includes \eq{dc2eq27} for $\Phi,\Psi$ and $\Psi\ci\Phi$, where
the top left triangle commutes by \eq{dc2eq25}. This proves
Proposition \ref{dc2prop11} for complex analytic d-critical loci.

In the algebraic case, for $(X,s)$ an algebraic d-critical locus
over a field $\K$, the argument above shows that for each
scheme-theoretic point $x$ of the reduced $\K$-subscheme $R^\red$ we
have $J_\Phi\vert_x= J_{\smash{\ti\Phi}}\vert_x$, where
\eq{dc6eq7}--\eq{dc6eq11} are now diagrams of finite-dimensional
vector spaces over the residue field of $X$ at $x$, rather than over
$\C$. Since $R^\red$ is a reduced $\K$-scheme and $J_\Phi,
J_{\smash{\ti\Phi}}$ are isomorphisms of line bundles on $R^\red$,
$J_\Phi\vert_x=J_{\smash{\ti\Phi}}\vert_x$ for each $x\in R^\red$
implies that $J_\Phi= J_{\smash{\ti\Phi}}$. The rest of the proof is
as for the complex analytic case.

\subsection{Proof of Theorem \ref{dc2thm3}}
\label{dc63}

We first construct the line bundle $K_{X,s}$ in Theorem
\ref{dc2thm3}, and show it satisfies parts (i),(ii). Observe that
Theorem \ref{dc2thm3}(i),(ii) characterizing the sheaf $K_{X,s}$ are
similar in structure to Theorem \ref{dc2thm1}(i),(ii) characterizing
the sheaf $\cS_X$. We will follow the method of \S\ref{dc31} to
prove Theorem \ref{dc2thm3}(i),(ii), handling the complex analytic
and algebraic cases together. The analogues of Lemmas \ref{dc3lem1},
\ref{dc3lem2} and \ref{dc3lem3} are Proposition \ref{dc2prop11},
Theorem \ref{dc2thm2} and:

\begin{lem} Let\/ $(R,U,f,i),(S,V,g,j)$ be critical charts on
$(X,s)$. Then there exists a unique isomorphism
\e
J_{\sst R,U,f,i}^{\sst S,V,g,j}:i^*\bigl(K_U^{\ot^2}\bigr)
\vert_{R^\red\cap S^\red}\longra
j^*\bigl(K_V^{\ot^2}\bigr)\vert_{R^\red\cap S^\red}
\label{dc6eq12}
\e
such that if\/ $x,U'\subseteq U,V'\subseteq V,(T,W,h,k)$ and\/
$\Phi:(R',U',f',i')\hookra (T,W,h,k),$
$\Psi:(S',V',g',j')\hookra(T,W,h,k)$ are as in Theorem\/
{\rm\ref{dc2thm2},} and\/ $J_\Phi,J_\Psi$ are as in Definition\/
{\rm\ref{dc2def4}} for $\Phi,\Psi,$ then
\e
J_{\sst R,U,f,i}^{\sst S,V,g,j}\vert_{R^{\prime\red}\cap
S^{\prime\red}}=J_\Psi^{-1}\ci J_\Phi\vert_{R^{\prime\red}\cap
S^{\prime\red}}.
\label{dc6eq13}
\e

Also, if\/ $(T,W,h,k)$ is any other critical chart on $(X,s)$ then
\e
\begin{gathered} J_{\sst R,U,f,i}^{\sst
R,U,f,i}=\id_{i^*(K_U^2)\vert_{R^\red}}, \qquad J^{\sst
R,U,f,i}_{\sst S,V,g,j}=\bigl(J_{\sst R,U,f,i}^{\sst
S,V,g,j}\bigr)^{-1},\\
\text{and\/}\qquad J_{\sst S,V,g,j}^{\sst T,W,h,k}\ci J_{\sst
R,U,f,i}^{\sst S,V,g,j}\vert_{R^\red\cap S^\red\cap T^\red}=J_{\sst
R,U,f,i}^{\sst T,W,h,k}\vert_{R^\red\cap S^\red\cap T^\red}.
\end{gathered}
\label{dc6eq14}
\e
\label{dc6lem2}
\end{lem}

\begin{proof} Suppose $x,$ $U',$ $V',$ $\Phi,$ $\Psi,$ $(T,W,h,k)$
and $\check x,$ $\check U',$ $\check V',$ $\check\Phi,$
$\check\Psi,$ $(\check T,\check W,\check h,\check k)$ are two
possible choices in Theorem \ref{dc2thm2}. We will show that
\e
J_\Psi^{-1}\ci J_\Phi\vert_{R^{\prime\red}\cap S^{\prime\red}\cap
\check R^{\prime\red}\cap\check
S^{\prime\red}}=J_{\check\Psi}^{-1}\ci
J_{\check\Phi}\vert_{R^{\prime\red}\cap S^{\prime\red}\cap \check
R^{\prime\red}\cap\check S^{\prime\red}}.
\label{dc6eq15}
\e
Let $\ti x\in R^{\prime\red}\cap S^{\prime\red}\cap \check
R^{\prime\red}\cap\check S^{\prime\red}\subseteq T\cap\check T$.
Then applying Theorem \ref{dc2thm2} to $\ti x,(T,W,h,k),(\check
T,\check W,\check h,\check k)$ yields open $\ti x\in T'\subseteq T$,
$\ti x\in \check T'\subseteq\check T$, a critical chart $(\ti T,\ti
W,\ti h,\ti k)$ and embeddings $\Th:(T',W',h',k')\hookra(\ti T,\ti
W,\ti h,\ti k)$ and $\Xi:(\check T',\check W',\check h',\check
k')\hookra(\ti T,\ti W,\ti h,\ti k)$.

Set $U''=\Phi^{-1}(W')\cap\check\Phi^{-1}(\check W')$,
$V''=\Psi^{-1} (W')\cap\check\Psi^{-1}(\check W')$, and let
$(R'',\ab U'',\ab f'',\ab i'')\subseteq (R,U,f,i)$,
$(S'',V'',g'',j'')\subseteq (S,V,g,j)$ be the corresponding
subcharts. Then we have a diagram of embeddings of critical charts
\begin{equation*}
\xymatrix@C=50pt@R=4pt{ (R'',U'',f'',i'') \ar[r]_{\Phi\vert_{U''}}
\ar[ddr]_(0.25){\check\Phi\vert_{U''}} &
(T',W',h',k') \ar[dr]^{\Th} \\
&& (\ti T,\ti W,\ti h,\ti k). \\
(S'',V'',g'',j'') \ar[r]^{\check\Psi\vert_{V''}}
\ar[uur]_(0.75){\Psi\vert_{V''}} & (\check T',\check W',\check h',\check k')
\ar[ur]_\Xi }
\end{equation*}
Hence we have
\e
\begin{split}
J_\Th\ci J_\Phi\vert_{R^{\prime\prime\red}}=J_{\Th\ci\Phi}
\vert_{R^{\prime\prime\red}}=J_{\Xi\ci\check\Phi}\vert_{R^{\prime
\prime\red}}=J_\Xi\ci J_{\check\Phi}\vert_{R^{\prime\prime\red}},\\
J_\Th\ci J_\Psi\vert_{S^{\prime\prime\red}}=
J_{\Th\ci\Psi}\vert_{S^{\prime\prime\red}}=
J_{\Xi\ci\check\Psi}\vert_{S^{\prime\prime\red}}=J_\Xi\ci
J_{\check\Psi}\vert_{S^{\prime\prime\red}},
\end{split}
\label{dc6eq16}
\e
using \eq{dc2eq28} in the first and third steps of each line, and
the first part of Proposition \ref{dc2prop11} in the second steps.
Restricting \eq{dc6eq16} to $R^{\prime\prime\red}\cap
S^{\prime\prime\red}$, inverting the second line, and composing with
the first line gives
\begin{align*}
J_\Psi^{-1}&\ci J_\Phi\vert_{R^{\prime\prime\red}\cap
S^{\prime\prime\red}}=(J_\Th\ci J_\Psi)^{-1}\ci(J_\Th\ci
J_\Phi)\vert_{R^{\prime\prime\red}\cap S^{\prime\prime\red}}\\
&=(J_\Xi\ci J_{\check\Psi})^{-1}\ci (J_\Xi\ci
J_{\check\Phi})\vert_{R^{\prime\prime\red}\cap S^{\prime\prime\red}}
=J_{\check\Psi}^{-1}\ci J_{\check\Phi}\vert_{R^{\prime\prime\red}\cap
S^{\prime\prime\red}}.
\end{align*}
This proves \eq{dc6eq15} on the open neighbourhood
$R^{\prime\prime\red}\cap S^{\prime\prime\red}$ of $\ti x$ in
$R^{\prime\red}\cap S^{\prime\red}\cap \check
R^{\prime\red}\cap\check S^{\prime\red}$. As this works for all such
$\ti x$, equation \eq{dc6eq15} holds.

Thus, Theorem \ref{dc2thm2} shows that for each $x\in R^\red\cap
S^\red$ we can choose an open neighbourhood $R^{\prime\red}\cap
S^{\prime\red}$ on which the restriction of $J_{\sst R,U,f,i}^{\sst
S,V,g,j}$ is defined by \eq{dc6eq13}. These open neighbourhoods
$R^{\prime\red}\cap S^{\prime\red}$ form an open cover of
$R^\red\cap S^\red$, and on overlaps $(R^{\prime\red}\cap
S^{\prime\red})\cap(\check R^{\prime\red}\cap\check S^{\prime\red})$
the corresponding values of $J_{\sst R,U,f,i}^{\sst S,V,g,j}$ agree.
Therefore by properties of sheaves there is a unique morphism
$J_{\sst R,U,f,i}^{\sst S,V,g,j}$ in \eq{dc6eq12} such that
\eq{dc6eq13} holds for all applications of Theorem \ref{dc2thm2}.
Finally, we prove \eq{dc6eq14} by the method used for
\eq{dc3eq6}--\eq{dc3eq7} in Lemma~\ref{dc3lem3}.
\end{proof}

The existence of a line bundle $K_{X,s}$, unique up to canonical
isomorphism, satisfying Theorem \ref{dc2thm3}(i),(ii) now follows
from Lemma \ref{dc6lem2} in the same way that the first part of
Theorem \ref{dc2thm1} was deduced from Lemma \ref{dc3lem3}
in~\S\ref{dc31}.

For parts (iii),(iv), defining $\ka_x$ in \eq{dc2eq31} and showing
that \eq{dc2eq33} commutes for all $(R,U,f,i)$ with $x\in R$, first
note that if $(R,U,f,i)$ is a critical chart on $(X,s)$ with $x\in
X$, and $\al_{x,R,U,f,i}$ is as in Theorem \ref{dc2thm3}(iv), then
in the notation of \S\ref{dc62} we have
\e
\al_{x,R,U,f,i}=\rho_{XU}^{\ot^2}\big\vert_x\ci \bigl[\id\ot
\det(\Hess'_{i(x)}f)\bigr].
\label{dc6eq17}
\e
Thus, if $\Phi:(R,U,f,i)\hookra(S,V,g,j)$ is an embedding of
critical charts then equation \eq{dc6eq11} shows the following
commutes:
\begin{equation*}
\xymatrix@C=150pt@R=-3pt{ & K_U^{\ot^2}\vert_{i(x)}
\ar[dd]_{J_\Phi\vert_x} \\
\bigl(\La^{\rm top}T_x^*X\bigr)^{\ot^2} \ar[ur]^{\al_{x,R,U,f,i}}
 \ar[dr]_{ \al_{x,S,V,g,j}} \\
& K_V^{\ot^2}\vert_{j(x)}. }
\end{equation*}
It is now easy to see from Theorem \ref{dc2thm3}(i)--(ii) that there
is a unique isomorphism $\ka_x$ in \eq{dc2eq31} such that
$\al_{x,R,U,f,i}\ci\ka_x= \io_{R,U,f,i}\vert_x$ for all critical
charts $(R,U,f,i)$ on $(X,s)$ with $x\in R$. This proves
Theorem~\ref{dc2thm3}(iii)--(iv).

\subsection{Proof of Proposition \ref{dc2prop12}}
\label{dc64}

Let $\phi:(X,s)\ra(Y,t)$ be a morphism of d-critical loci with
$\phi:X\ra Y$ smooth, fix $x_0\in X$ with $\phi(x_0)=y_0\in Y$, and
set $m=\dim T_{x_0}X$ and $n=\dim T_{y_0}Y$, so that $\phi$ is
smooth of relative dimension $m-n\ge 0$ near $x_0$.

As in the proof of Proposition \ref{dc2prop3} in \S\ref{dc42}, we
may choose open $y_0\in S\subseteq Y$ and $x_0\in
R\subseteq\phi^{-1}(S)\subseteq X$, closed embeddings $i:R\hookra
U$, $j:S\hookra V$ for $U,V$ complex manifolds (or smooth
$\K$-schemes) with $\dim U=m$, $\dim V=n$, a morphism $\Phi:U\ra V$
smooth of relative dimension $m-n$ with $\Phi\ci
i=j\ci\phi\vert_R:R\ra V$, and holomorphic $g:V\ra\C$ and
$f=g\ci\Phi:U\ra\C$ (or regular $g:V\ra\bA^1$ and
$f=g\ci\Phi:U\ra\bA^1$) with $\io_{R,U}(s\vert_R)=
i^{-1}(f)+I_{R,U}^2$ and $\io_{S,V}(t\vert_S)=j^{-1}(g)+I_{S,V}^2$.
Since $\dim T_{x_0}X=\dim U$ and $\dim T_{y_0}Y=\dim V$, the second
part of Proposition \ref{dc2prop2} shows that making $R,S,U,V$
smaller, we can suppose that $(R,U,f,i)$ and $(S,V,g,j)$ are
critical charts on $(X,s)$ and~$(Y,t)$.

We have a diagram of coherent sheaves on $R$ with exact rows, where
the bottom row is \eq{dc2eq34} restricted to~$R$:
\e
\begin{gathered}
\xymatrix@C=15pt@R=15pt{
0 \ar[r] & {\begin{subarray}{l}\ts (j\ci\phi\vert_R)^*(T^*V)\\
\ts =(\Phi\ci i)^*(T^*V)\end{subarray}} \ar[d]^{\phi\vert_R^*(\d
j^*)} \ar[rr]_(0.6){i^*(\d\Phi^*)} && i^*(T^*U) \ar[d]^{\d i^*}
\ar[rr] && i^*(T^*_{U/V}) \ar@{.>}[d]^\al \ar[r] & 0 \\
0 \ar[r] & \phi\vert_R^*(T^*Y) \ar[rr]^(0.6){\d\phi^*\vert_R} &&
T^*X\vert_R  \ar[rr] && T^*_{X/Y}\vert_R  \ar[r] & 0. }
\end{gathered}
\label{dc6eq18}
\e
The left hand square commutes as $\Phi\ci i=j\ci\phi\vert_R$. So by
exactness, there is a unique morphism $\al$ as shown making
\eq{dc6eq18} commute. As $\Phi,\phi\vert_R$ are both smooth of
relative dimension $m-n$, $i^*(T^*_{U/V})$ and $T^*_{X/Y}\vert_R$
are both vector bundles on $R$ of rank $m-n$. But as $i$ is an
embedding, $\d i^*$ is surjective, so $\al$ is surjective, and thus
$\al$ is an isomorphism.

Taking top exterior powers in the top line of \eq{dc6eq18} gives an
isomorphism
\e
\be: (j\ci\phi\vert_R)^*(K_V)\ot i^*\bigl(\La^{\rm
top}T^*_{U/V}\bigr) \,{\buildrel\cong\over\longra}\,i^*(K_U).
\label{dc6eq19}
\e
Define an isomorphism $\Up_{\Phi,R}:\bigl[\phi\vert_{X^\red}^*
(K_{Y,t})\!\ot\!(\La^{\rm top}T^*_{X/Y})\vert_{X^\red}^{\ot^2}\bigr]
\big\vert_{\smash{R^\red}}\!\ra\! K_{X,s}\vert_{R^\red}$ by the
commutative diagram of isomorphisms
\e
\begin{gathered}
\xymatrix@C=160pt@R=17pt{
*+[r]{\bigl[\phi\vert_{X^\red}^*(K_{Y,t})\ot(\La^{\rm top}T^*_{X/Y})
\vert_{X^\red}^{\ot^2}\bigr]\big\vert_{R^\red}}
\ar[d]^{\phi\vert_{R^\red}^*(\io_{S,V,g,j})\ot (\La^{\rm
top}\al^{-1}\vert_{R^\red})^{\ot^2}} \ar[r]_(0.7){\Up_{\Phi,R}} &
*+[l]{ K_{X,s}\vert_{R^\red}} \ar[d]_{\io_{R,U,f,i}} \\
*+[r]{(j\ci\phi\vert_{R^\red})^*\bigl(K_V^{\ot^2}\bigr)\ot
i^*\bigl(\La^{\rm top}T^*_{U/V}\bigr)^{\ot^2}}
\ar[r]^(0.7){\be\vert_{R^\red}^{\ot^2}} &
*+[l]{i^*\bigl(K_U^{\ot^2}\bigr)\vert_{R^\red},\!\!{}} }
\end{gathered}
\label{dc6eq20}
\e
for $\io_{R,U,f,i},\io_{S,V,g,j}$ as in \eq{dc2eq29}, and $\al,\be$
as in~\eq{dc6eq18}--\eq{dc6eq19}.

We now claim that for each $x\in R$, equation \eq{dc2eq36} with
$\Up_{\Phi,R}$ in place of $\Up_\phi$ commutes. To see this, compare
\eq{dc2eq33} for $\ka_x,\ka_{\phi(x)}$, \eq{dc2eq36}, \eq{dc6eq17},
the restriction of \eq{dc6eq20} to $x$, and the commutative diagram
\e
\begin{gathered}
{}\!\!\!\!\!\!\xymatrix@C=159pt@R=15pt{
*+[r]{\bigl(\La^{\rm top}T_{\phi(x)}^*Y\bigr)^{\ot^2}\ot
\bigl(\La^{\rm top}T^*_{X/Y}\vert_x\bigr)^{\ot^2}}
\ar[d]^(0.4){\id\ot \det(\Hess'_{j\ci\phi(x)}g) }
\ar[r]_(0.62){\up_x^{\ot^2}} &
*+[l]{\bigl(\La^{\rm top}T_x^*X\bigr)^{\ot^2}}
\ar[d]_(0.4){\id\ot \det(\Hess'_{i(x)}f) } \\
*+[r]{{}\,\,\,\,\,\,\begin{subarray}{l}\ts \bigl(\La^{\rm top}
T_{\phi(x)}^*Y\bigr)^{\ot^2}\ot \bigl(\La^{\rm top}T^*_{X/Y}\vert_x
\bigr)^{\ot^2}\\ \ts \ot\bigl(\La^{\rm top}N_{\sst
YV}^*\vert_{\phi(x)}\bigr)^{\ot^2}\end{subarray}}
\ar[d]^(0.6){(\rho_{YV}\vert_{\phi(x)}^{\ot^2}) \ot(\La^{\rm
top}\al\vert_x^{-1})^{\ot^2}} \ar[r]^(0.64){ \up_x^{\ot^2}\ot
(\La^{\rm top}\ga_x)^{\ot^2} } &
*+[l]{\begin{subarray}{l}\ts\quad\,\,\bigl(\La^{\rm
top}T_x^*X\bigr)^{\ot^2}\\ \ts \ot\bigl(\La^{\rm top}N_{\sst
XU}^*\vert_x\bigr)^{\ot^2}\end{subarray}\,\,\,\,\,\,{}}
\ar[d]_(0.6){\rho_{XU}\vert_x^{\ot^2}} \\
*+[r]{K_V^{\ot^2}\vert_{j\ci\phi(x)}\ot(\La^{\rm top}T^*_{U/V}
\vert_{i(x)})^{\ot^2}} \ar[r]^(0.62){\be\vert_x^{\ot^2}} &
*+[l]{K_U^{\ot^2}\vert_{i(x)}.\!\!{}} }\!\!\!\!\!\!{}
\end{gathered}
\label{dc6eq21}
\e
Here to prove \eq{dc6eq21} commutes, consider the commutative
diagram, with exact rows and columns
\e
\begin{gathered}
\xymatrix@C=15pt@R=15pt{ & 0 \ar[d] && 0 \ar[d] \\
0 \ar[r] & N_{\sst YV}^*\vert_{\phi(x)} \ar[d]
\ar@{.>}[rr]^{\ga_x}_\cong && N_{\sst
XU}^*\vert_x \ar[rr] \ar[d] && 0 \ar[d] \\
0 \ar[r] & T_{\phi(x)}^*V \ar[d]^{\d j^*\vert_{\phi(x)}}
\ar[rr]_(0.6){\d\Phi^*\vert_{i(x)}} && T_{i(x)}^*U \ar[d]^{\d
i^*\vert_x}
\ar[rr] && T^*_{U/V}\vert_{i(x)} \ar[d]^{\al\vert_x} \ar[r] & 0 \\
0 \ar[r] & T_{\phi(x)}^*Y \ar[d] \ar[rr]^(0.6){\d\phi^*\vert_x} &&
T_x^*X \ar[d] \ar[rr] && T^*_{X/Y}\vert_x \ar[d] \ar[r] & 0 \\
& 0 && 0 && {0.\!\!{}} }
\end{gathered}
\label{dc6eq22}
\e

The bottom two rows of \eq{dc6eq22} are \eq{dc6eq18} restricted to
$x$. Exactness implies that there is a unique isomorphism $\ga_x$ as
shown in the first row. One can show that $\ga_x$ identifies the
nondegenerate quadratic forms $\Hess'_{j\ci\phi(x)}g$ and
$\Hess'_{i(x)}f$ on $N_{\sst YV}^*\vert_{\phi(x)}$ and $N_{\sst
XU}^*\vert_x$. This implies the upper rectangle of \eq{dc6eq21}
commutes. As in equations \eq{dc6eq8}, \eq{dc6eq19}, and
\eq{dc2eq36}, the maps $\rho_{YV}\vert_{\phi(x)},
\rho_{XU}\vert_x,\be\vert_x,\up_x$ in \eq{dc6eq21} are obtained by
taking top exterior powers in the first column, second column,
second row, and third row of \eq{dc6eq22}, respectively. Thus,
taking top exterior powers in \eq{dc6eq22} shows that the bottom
rectangle of \eq{dc6eq21} commutes.

To summarize the proof so far: for each $x_0\in X$ we have
constructed an open neighbourhood $x_0\in R\subseteq X$ and an
isomorphism $\Up_{\Phi,R}:\bigl[\phi\vert_{X^\red}^*(K_{Y,t})\ot
(\La^{\rm top}T^*_{X/Y})\vert_{X^\red}^{\ot^2}\bigr]
\big\vert_{\smash{R^\red}} \ra K_{X,s}\vert_{R^\red}$, such that for
each $x\in R^{\red}$, equation \eq{dc2eq36} with $\Up_{\Phi,R}$ in
place of $\Up_\phi$ commutes.

Suppose $x_0,R,\Up_{\Phi,R}$ and $x_0',R',\Up_{\Phi',R'}$ are two
possible choices above. Equation \eq{dc2eq36} implies that
$\Up_{\Phi,R}\vert_x=\Up_{\Phi',R'}\vert_x$ for all $x\in R^\red\cap
R^{\prime\red}$. As $R^\red\cap R^{\prime\red}$ is reduced, this
forces $\Up_{\Phi,R}\vert_{R^\red\cap
R^{\prime\red}}=\Up_{\Phi',R'}\vert_{R^\red\cap R^{\prime\red}}$.
Since such $R^\red$ form an open cover of $X^\red$, there is a
unique isomorphism $\Up_\phi$ in \eq{dc2eq35} with
$\Up_\phi\vert_{R^\red}=\Up_{\Phi,R}$ for all $x_0,R,\Up_{\Phi,R}$
as above. Proposition \ref{dc2prop12} follows.

\section{Equivariant d-critical loci}
\label{dc7}

Next we prove Propositions \ref{dc2prop14} and \ref{dc2prop15} from~\S\ref{dc26}.

\subsection{Proof of Proposition \ref{dc2prop14}}
\label{dc71}

For part (a), suppose $(X,s)$ is an algebraic d-critical locus over $\K$ equivariant under a good action $\mu:G\t X\ra X$ of an algebraic $\K$-torus $G$, with character $\chi:G\ra\bG_m$, and let $x\in X$. As $\mu$ is good, there exists a $G$-invariant affine open neighbourhood $R'$ of $x$ in $X$. Choose a closed embedding $(y_1,\ldots,y_K):R'\hookra\bA^K$. Then $y_a\ci\mu:G\t R'\ra\bA^1$ is regular, so as $H^0(\O_{G\t R'})\cong H^0(\O_G)\ot H^0(\O_{R'})$, we may write $y_a\ci\mu(\ga,r)=\sum_{b=1}^{L_a}x_a^b(\ga)y_a^b(r)$ with $x_a^b:G\ra\bA^1$ and $y_a^b:R'\ra\bA^1$ regular, for~$a=1,\ldots,K$.

Define $V=\langle y_a^b:a=1,\ldots,K$, $b=1,\ldots,L_a\rangle_\K\subseteq H^0(\O_Y)$ to be the finite-dimensional $\K$-vector subspace of
$H^0(\O_Y)$ generated by the $y_a^b$. Then $V$ is $G$-invariant, and contains $y_1,\ldots,y_K$ as $y_a=\sum_{b=1}^{L_a}x_a^b(1)\cdot y_a^b$.
As $G$ is a torus, $V$ decomposes as a direct sum of 1-dimensional $G$-representations, so we may choose a basis $v_1,\ldots,v_M$ for $V$ with
$v_a\ci\mu(\ga,r)=\ka_a(\ga)v_a(r)$ for all $\ga\in G$, $r\in R'$ and $a=1,\ldots,M$, where $\ka_a:G\ra\bG_m$ is a character of $G$. Then $i':=(v_1,\ldots,v_M):R'\hookra\bA^M=V^*$ is a closed embedding, since $(v_1,\ldots,v_K):R'\hookra\bA^K$ is and $v_1,\ldots,v_K\in V$, and is $G$-equivariant under the obvious linear $G$-action on $\bA^M$ given by~$\ga:(z_1,\ldots,z_M)\mapsto(\ka_1(\ga)z_1,\ldots,\ka_M(\ga)z_M)$.

Let $I\subset\K[z_1,\ldots,z_M]$ be the ideal of functions vanishing on $i'(R')$. Then $I$ is $G$-invariant, as $i'(R')$ is, so as $G$ is a torus we can choose a finite set of $G$-equivariant generators $h_1,\ldots,h_N$ for $I$, with $h_a(\ga\cdot r)=\la_a(\ga)h_a(r)$ for all $\ga\in G$, $r\in R'$ and $a=1,\ldots,N$, where $\la_a:G\ra\bG_m$ is a character of $G$. Then
\begin{equation*}
T_{i'(x)}(i'(R'))=\Ker\bigl[\ts\bigop_{a=1}^N\d h_a\vert_{i'(x)}:
T_{i'(x)}\bA^M\longra\bigop_{a=1}^N\d h_a^*(T_0\bA^1)\vert_{i'(x)}
\bigr].
\end{equation*}

Choose a minimal subset $\{h_{a_1},\ldots,h_{a_P}\}\subseteq\{h_1,\ldots,h_N\}$ such that
\begin{equation*}
T_{i'(x)}(i'(R'))=\Ker\bigl[\ts\bigop_{b=1}^P\d h_{a_b}\vert_{i'(x)}:
T_{i'(x)}\bA^M\longra\bigop_{b=1}^P\d h_{a_b}^*(T_0\bA^1)\vert_{i'(x)}
\bigr].
\end{equation*}
Then $P=N-\dim T_{i'(x)}i'(R')=N-\dim T_xX$, and $\d h_{a_1},\ldots,\d h_{a_P}$ are linearly independent at $i'(x)$. Define $U'$ to be the closed $\K$-subscheme of $\bA^M$ defined by $h_{a_1}=\cdots=h_{a_P}=0$. Then $i'(R')\subseteq U'$, as $h_{a_1},\ldots,h_{a_P}\in I$, and $U'$ is $G$-invariant as the $h_{a_b}$ are $G$-equivariant, and $U'$ is smooth of dimension $N-P=\dim T_xX$ near $i'(x)$ as $\d h_{a_1},\ldots,\d h_{a_P}$ are linearly independent at~$i'(x)$.

Let $\ti U\subseteq U'$ be the open $\K$-subscheme of points where $U'$ is smooth of dimension $\dim T_xX$. Then $\ti U$ is $G$-invariant, smooth of dimension $\dim T_xX$, and quasi-affine, so that $i'(x)\in\ti U$. As $\ti U$ is smooth it is normal, so by Lemma \ref{dc2lem} we can choose a $G$-invariant affine open neighbourhood $U''$ of $i'(x)$ in $\ti U$. Define $R''=i'^{-1}(U'')$ and $i'':=i'\vert_{R''}:R''\hookra U''$. Then $R''$ is an open neighbourhood of $x$ in $X$, as $R'$ is, $i'(R')\subseteq U'$, and $U''$ is an open neighbourhood of $i'(x)$ in $U'$. Also $i''$ is a closed embedding as $i'$ is, so $R''$ is affine. And $R''$ is $G$-invariant as $U''$ is and $i'$ is $G$-equivariant, so $i''$ is $G$-equivariant.

Write $I_{R'',U''}\subset H^0(\O_{U''})$ for the $G$-invariant ideal of functions on $U''$ vanishing on the closed subscheme $i''(R'')\subset U''$. Theorem \ref{dc2thm1}(i) gives a sheaf morphism $\io_{R'',U''}$. As $R'',U''$ are affine this descends to global sections, giving
\begin{equation*}
(\io_{R'',U''})_*:H^0(\cS_X\vert_{R''})\longra H^0(\O_{U''})/
I_{R'',U''}^2.
\end{equation*}
Thus $(\io_{R'',U''})_*(s\vert_{R''})\in H^0(\O_{U''})/I_{R'',U''}^2$, so we can choose regular $f'':U''\ra\bA^1$ with $(\io_{R'',U''})_*(s\vert_{R''})=
f''+I_{R'',U''}^2$. Now by assumption $s$ is $G$-equivariant with character $\chi:G\ra\bG_m$, so $(\io_{R'',U''})_*(s\vert_{R''})$ is also $G$-equivariant with character $\chi$ as $R''$ is $G$-invariant and $i''$ $G$-equivariant. By averaging $f''$ over the $G$-action twisted by $\chi$, we can suppose that $f''$ is also $G$-equivariant with character $\chi$, that is, $f''(\ga\cdot u'')=\chi(\ga)\cdot f''(u'')$ for all $\ga\in G$ and~$u''\in U''$.

Since $\dim U''=\dim T_xX$, Proposition \ref{dc2prop2} now shows that we can choose Zariski open $U\subseteq U''$ and $R=i^{\prime\prime -1}(U)\subseteq R''$ such that $(R,U,f,i)$ is a critical chart on $(X,s)$ with $x\in R$, where $f=f''\vert_U$ and $i=i''\vert_R$. The proof of this in \S\ref{dc41} works by showing that the closed subschemes $i''(R'')\subseteq\Crit(f'')\subseteq U''$ satisfy $i''(R'')=\Crit(f'')$ near $i''(x)$ in $U''$, and restricting to an open neighbourhood $U$ of $x$ in $U''$ such that $U\cap i''(R'')=U\cap\Crit(f'')$. Take $U$ to be the largest such neighbourhood, the union of all open $V\subseteq U''$ with $V\cap i''(R'')=V\cap\Crit(f'')$. Then $U$ is $G$-invariant, as $i''(R''),\Crit(f'')$ are. Writing $\rho$ for the $G$-action on $U$, it follows that $(R,U,f,i),\rho$ is a $G$-equivariant critical chart on $(X,s)$ with $x\in R$ and $\dim U=\dim T_xX$.  This proves Proposition~\ref{dc2prop14}(a).

For part (b), suppose that for all $x\in X$ there exists a $G$-equivariant critical chart $(R,U,f,i),\rho$ on $(X,s)$ with $x\in R$. For such $x,(R,U,f,i),\rho$, note that as $U$ is smooth it is normal, so by Lemma \ref{dc2lem} there is a $G$-invariant affine open neighbourhood $U'$ of $x$ in $X$. As $i:R\hookra U$ is a $G$-equivariant closed embedding, $R'=i^{-1}(U')$ is a $G$-equivariant affine open neighbourhood of $x$ in $X$. Since such $R'$ exist for all $x\in X$, $\mu$ is a good $G$-action.

\subsection{Proof of Proposition \ref{dc2prop15}}
\label{dc72}

Proposition \ref{dc2prop15} is a $G$-equivariant version of the $\K$-scheme case of Theorem \ref{dc2thm2}, which is proved in \S\ref{dc52}. The proof relies on Proposition \ref{dc2prop6}, which is proved in \S\ref{dc51}. We will explain how to modify the proofs in \S\ref{dc51}--\S\ref{dc52} to include $G$-equivariance throughout.

Suppose $(X,s)$ is an algebraic d-critical locus over $\K$ equivariant under the action $\mu:G\t X\ra X$ of an algebraic $\K$-torus $G$, with character $\chi:G\ra\bG_m$, and let $(R,U,f,i),\rho$ be a $G$-equivariant critical chart on $(X,s)$. For the equivariant version of Proposition \ref{dc2prop6}, we must show that for each $x\in X$ there exists a $G$-equivariant subchart $(R',U',f',i')\subseteq (R,U,f,i)$ with $x\in R'$, and a $G$-equivariant embedding $\Phi:(R',U',f',i')\hookra (S,V,g,j)$ into a $G$-equivariant critical chart $(S,V,g,j),\si$ on $(X,s)$, such that $G$ acts linearly on $\bA^n$, inducing the $G$-action $\si$ on the $G$-invariant Zariski open subset~$V\subseteq\bA^n$.

To prove this, modify \S\ref{dc51} as follows. Take the affine open neighbourhood $\ti U$ of $x$ in $U$ to be $G$-invariant, which is possible by Lemma \ref{dc2lem}. Then take $\ti\Phi:\ti U\ra\bA^n$ to be $G$-equivariant under a linear $G$-action on $\bA^n$, which is possible as in \S\ref{dc71}. Choose $V\subseteq\bA^n$ to be $G$-invariant, and the \'etale coordinates $(z_1,\ldots,z_n)$ on $V$ with $\ti\Phi(\ti U)\cap V=\{z_{m+1}=\cdots=z_n=0\}$ to be $G$-equivariant, that is, they should satisfy $z_a(\ga\cdot v)=\ka_a(\ga)z_a(v)$ for all $\ga\in G$, $v\in V$ and $a=1,\ldots,n$, where $\ka_a:G\ra\bG_m$ is a character of $G$. Choose $h:V\ra\bA^1$ to be $G$-equivariant with character $\chi$, which is possible as for $f''$ in~\S\ref{dc71}.

Now in the expression \eq{dc5eq1} for $g:V\ra\bA^1$, the terms $z_a^2$ for $a=m+1,\ldots,n$ are not $G$-equivariant. To deal with this, replace $V$ by $V'=V\t\bA^{n-m}$, where $\bA^{n-m}$ has coordinates $(w_{m+1},\ldots,w_n)$ and $G$ acts on $\bA^{n-m}$ by
\begin{equation*}
\ga:(w_{m+1},\ldots,w_n)\longmapsto\bigl(\ka_{m+1}(\ga)^{-1}\chi(\ga)w_{m+1},\ldots, \ka_n(\ga)^{-1}\chi(\ga)w_n\bigr).
\end{equation*}
Replace $\Phi:U\ra V$ by $\Phi'=\Phi\t 0:U\ra V'=V\t\bA^{n-m}$, and $g:V\ra\bA^1$ in \eq{dc5eq1} by $g':V'\ra\bA^1$ given by
\e
\begin{split}
g'\bigl(&v,(w_{m+1},\ldots,w_n)\bigr)=\ts h(v)-\sum_{a=m+1}^nz_a(v)\cdot\frac{\pd h}{\pd z_a}(v)\\
&+\ts\ha\sum_{a,b=m+1}^nz_a(v)z_b(v)\cdot\frac{\pd^2 h}{\pd z_a\pd
z_b}(v)+\sum_{a=m+1}^nz_a(v)w_a.
\end{split}
\label{dc7eq1}
\e
Then each term in \eq{dc7eq1} is $G$-equivariant with character $\chi$. The rest of the proof in \S\ref{dc51} may be made $G$-equivariant in a similar way to~\S\ref{dc71}.

To modify the $\K$-scheme case of \S\ref{dc52} to include $G$-equivariance, we choose $i(x)\in U'\subseteq U$ to be $G$-invariant and $\Th:U'\ra V$ with $\Th\ci i'=j\vert_{R'}$ to be $G$-equivariant, which is possible as $V\subseteq\bA^m$ is open and $G$-invariant with $G$-action induced from a linear $G$-action on $\bA^m$. Then $f'-g\ci\Th\in (I'_{R',U'})^2$ is $G$-equivariant with character $\chi$. So when as in \eq{dc5eq6} we write $f'=g\ci\Th+r_1s_1+\cdots+r_ns_n$, we choose the $r_a$ to be $G$-equivariant with some character $\la_a:G\ra\bG_m$, and the $s_a$ to be $G$-equivariant with the complementary character $\la_a^{-1}\chi$, for $a=1,\ldots,n$. Then $W=V\t\bA^{2n}$ has $G$-action
\begin{align*}
\ga:\,&\bigl(v,(r_1,\ldots,r_n,s_1,\ldots,s_n)\bigr)\longmapsto\\
&\bigl(\ga\cdot v,(\la_1(\ga)r_1,\ldots,\la_n(\ga)r_n,\la_1^{-1}(\ga)\chi(\ga)s_1,\ldots, \la_n^{-1}(\ga)\chi(\ga)s_n)\bigr),
\end{align*}
and the rest of the proof in \S\ref{dc52} may be made $G$-equivariant in a similar way to \S\ref{dc71}. This proves Proposition~\ref{dc2prop15}.

\section{Extension to Artin stacks}
\label{dc8}

Finally we prove Proposition \ref{dc2prop17} and Theorem
\ref{dc2thm4} from~\S\ref{dc28}.

\subsection{Proof of Proposition \ref{dc2prop17}}
\label{dc81}

The first part of Proposition \ref{dc2prop17} is standard stack
theory. Part (i) is a general property of global sections of sheaves
on Artin stacks, applied to the sheaf $\cS_X$. For (ii), the `only
if' and `and then' parts are immediate from Definition
\ref{dc2def10}. To prove the `if' part, suppose $s\in H^0(\cSz_X)$
with $\bigl(T,t^*(s)\bigr)$ an algebraic d-critical locus, and let
$v:V\ra X$ be smooth. Then $T\t_{t,X,v}V$ is equivalent to a
$\K$-scheme $W$ as $t$ is representable, so we have a $2$-Cartesian
diagram in~$\Art_\K$:
\begin{equation*}
\xymatrix@C=90pt@R=11pt{ *+[r]{W} \ar[d]^{\pi_T}
\ar[r]_(0.3){\pi_V} \drtwocell_{}\omit^{}\omit{^{\ze}} &
*+[l]{V} \ar[d]_v \\
*+[r]{T} \ar[r]^(0.7)t & *+[l]{X,\!\!{}} }
\end{equation*}
where $\pi_T:W\ra T$ and $\pi_V:W\ra V$ are smooth as $v,t$ are, and
$\pi_V:W\ra V$ is surjective as $t$ is.

We have $\pi_T^\star(s(T,t))=s(W,t\ci\pi_T)=s(W,v\ci\pi_V)=
\pi_V^\star(s(V,v))$ as $s\in H^0(\cS_X)$. Since
$\bigl(T,t^*(s)\bigr)=\bigl(T,s(T,t)\bigr)$ is an algebraic
d-critical locus and $\pi_T:W\ra T$ is smooth, the first part of
Proposition \ref{dc2prop3} shows that $\bigl(W,\pi_T^\star
(s(T,t))\bigr)$ is an algebraic d-critical locus. Thus
$\bigl(W,\pi_V^\star(s(V,v))\bigr)$ is an algebraic d-critical
locus, and $\pi_V:W\ra V$ is smooth and surjective, so the second
part of Proposition \ref{dc2prop3} shows that $\bigl(V,s(V,v)\bigr)$
is an algebraic d-critical locus. As this holds for all smooth
$v:V\ra X$, $(X,s)$ is a d-critical stack, proving the `if' part.

\subsection{Proof of Theorem \ref{dc2thm4}}
\label{dc82}

Let $u:U\!\ra\!X$ be a smooth atlas for $X$. Then $
U\!\t_{u,X,u}\!U$ and $\smash{U\!\t_{u,X,u}\! U\!\t_{u,X,u}\! U}$
are represented by $\K$-schemes $V,W$, with smooth projections
$\pi_1,\pi_2:V\ra U$ and $\pi_{12},\pi_{23}, \pi_{31}:W\ra V$ in
$\Sch_\K$ in a 2-commutative diagram in~$\Art_\K$:
\e
\begin{gathered}
\xymatrix@!0@C=50pt@R=25pt{ &  V \ar[drr]^(0.7){\pi_2}
\ar[dd]^{\pi_1} \\
W \ar[ur]^{\pi_{23}} \ar[drr]^(0.7){\pi_{31}}
\ar[dd]^(0.4){\pi_{12}} &&&  U \ar[dd]^u \\
&  U \ar[drr]^(0.7){u}  &  V \ar[ur]^{\pi_1}
\ar[dd]^{\pi_2} \\
V \ar[ur]^{\pi_2} \ar[drr]^(0.7){\pi_1} &&&  X. \\
&&  U \ar[ur]^u }
\end{gathered}
\label{dc8eq1}
\e
In $\Sch_\K$ we have $\pi_1\ci\pi_{12}=\pi_2\ci\pi_{31}$,
$\pi_1\ci\pi_{31}=\pi_2\ci\pi_{23}$,
$\pi_1\ci\pi_{23}=\pi_2\ci\pi_{12}$, so the same holds in $\Art_\K$
(with equalities of 1-morphisms, not just 2-isomorphisms), as
$F_\Sch^\Art$ is a strict 2-functor. From the fibre product $W=
U\t_{u,X,u}U$ we have a 2-morphism $\eta:u\ci\pi_1\Ra u\ci\pi_2$.

{\it Points\/} in $\K$-schemes and Artin $\K$-stacks were discussed
in \S\ref{dc27}. Points of $V$ correspond to triples
$(p_1,p_2,\th_{12})$ for $p_1,p_2\in U$, so that $u(p_1):=u\ci p_1$
and $u(p_2):=u\ci p_2$ are points in $X$, and $\th_{12}:u(p_1)\Ra
u(p_2)$ is a 2-isomorphism. On points we have
$\pi_1:(p_1,p_2,\th_{12})\mapsto p_1$,
$\pi_2:(p_1,p_2,\th_{12})\mapsto p_2$. Similarly, points of $W$
correspond to quintuples $(p_1,p_2,p_3,\th_{12},\th_{23})$ for
$p_1,p_2,p_3\in U$ and $\th_{12}:u(p_1)\Ra u(p_2)$,
$\th_{23}:u(p_2)\Ra u(p_3)$ 2-isomorphisms, and
$\pi_{12}:(p_1,p_2,p_3,\th_{12},\th_{23})\mapsto
(p_1,p_2,\th_{12})$,
$\pi_{23}:(p_1,p_2,p_3,\th_{12},\th_{23})\mapsto
(p_2,p_3,\th_{23})$, and $\pi_{31}:(p_1,p_2,p_3)\mapsto
(p_3,p_1,\th_{12}^{-1}\ci\th_{23}^{-1})$.

Taking reduced $\K$-subschemes and $\K$-substacks in \eq{dc8eq1}
gives another 2-commutative diagram in $\Art_\K$:
\e
\begin{gathered}
\xymatrix@!0@C=50pt@R=25pt{ &  V^\red \ar[drr]^(0.7){\pi_2^\red}
\ar[dd]^{\pi_1^\red}
\\
W^\red \ar[ur]^{\pi_{23}^\red} \ar[drr]^(0.7){\pi_{31}^\red}
\ar[dd]^(0.4){\pi_{12}^\red} &&&  U^\red \ar[dd]^{u^\red}
\\
&  U^\red \ar[drr]^(0.7){u^\red}  &  V^\red \ar[ur]^{\pi_1^\red}
\ar[dd]^{\pi_2^\red}
\\
V^\red \ar[ur]^{\pi_2^\red} \ar[drr]^(0.7){\pi_1^\red} &&& X^\red,
\\
&&  U^\red \ar[ur]^{u^\red} }
\end{gathered}
\label{dc8eq2}
\e
where again $u^\red:U^\red\ra X^\red$ is a smooth atlas,
$\pi_1^\red,\pi_2^\red:V^\red\ra U^\red$ are smooth, and so on.

Now line bundles on $X^\red$ are examples of {\it Cartesian\/}
(quasi-coherent) sheaves on $X^\red$ in the lisse-\'etale topology
\cite[Def.~12.3]{LaMo}, which in the notation of Proposition
\ref{dc2prop16} means that $\cA(\phi,\eta)$ is an isomorphism for
all diagrams \eq{dc2eq39}. As in \cite[Prop.~12.4.5]{LaMo},
Cartesian sheaves can be described completely in terms of the
diagram \eq{dc8eq2} for $u^\red:U^\red\ra X^\red$ a smooth atlas.
For line bundles  on $X^\red$, this means that the following functor
is an equivalence of categories:
\e
\begin{split}
F&:\bigl(\text{category of line bundles $\cL$ on
$X^\red$}\bigr)\longra \\
\bigl(&\text{category of pairs $(L,\la)$, where $L\ra U^\red$ is a
line bundle and}\\
&\text{$\la:(\pi^\red_1)^*(L)\ra (\pi_2^\red)^*(L)$ an isomorphism
of line bundles on $V^\red$,}\\
&\text{with $(\pi_{31}^\red)^*(\la)\ci(\pi_{23}^\red)^*(\la)\ci
(\pi_{12}^\red)^*(\la)=\id_{(\pi_1^\red\ci\pi_{12}^\red)^*(L)}$ on
$W^\red$}\bigr),\\
&\text{mapping}\;\> F:\cL\longmapsto
\bigl(\cL(U,u),\cL(\pi_2,\eta)^{-1}\ci\cL(\pi_1,\id_{u\ci\pi_1})\bigr)
\;\>\text{on objects}.
\end{split}
\label{dc8eq3}
\e

Fix a particular smooth atlas $u:U\ra X$ for the rest of the proof,
and use the notation $\pi_1,\pi_2,\pi_{12},\pi_{23},\pi_{31},\eta$
above. Define a line bundle $L$ on $U^\red$ by
\e
L=K_{U,s(U,u)}\ot \bigl(\La^{\rm top}T^{u\,*}_{
U/X}\bigr)\big\vert_{U^\red}^{\ot^{-2}},
\label{dc8eq4}
\e
as in \eq{dc2eq44}. In an analogue of \eq{dc2eq46}, for each point
$p\in U^\red\subseteq U$ define an isomorphism $\mu_p$ by the
commutative diagram
\e
{}\!\!\!\!\!\!\begin{gathered} \xymatrix@C=148pt@R=15pt{
*+[r]{L\vert_p} \ar[d]^{\mu_p} \ar@{=}[r] &
*+[l]{K_{U,s(U,u)}\vert_p\ot\! \bigl(\La^{\rm top}T^*_{
U/X}\bigr) \big\vert_p^{\ot^{-2}}} \ar[d]_{\ka_p\ot\id} \\
*+[r]{\bigl(\La^{\rm top}T_{u(p)}^*X\bigr)^{\ot^2}\!\!\ot\!\bigl(\La^{\rm
top}\fIso_{u(p)}(X)\bigr)^{\ot^2}} \ar[r]^(0.53){\al_p^2} &
*+[l]{\bigl(\La^{\rm top}T^*_pU\bigr)^{\ot^2}\!\!\ot\!
\bigl(\La^{\rm top}T^*_{U/X}\bigr)\big\vert_p^{\ot^{-2}},} }
\end{gathered}\!\!\!\!\!{}
\label{dc8eq5}
\e
where $\ka_p$ is as in \eq{dc2eq31}, and $\al_p$ as in
Theorem~\ref{dc2thm4}(c).

By Remark \ref{dc2rem10}(i)--(v) we have exact sequences of vector
bundles of mixed rank on $V,$ and an isomorphism~$\eta_*$:
\e
\begin{gathered}
\xymatrix@C=30pt@R=15pt{ 0 \ar[r] & \pi_1^*(T^{u\,*}_{U/X}) \ar[r] &
T^{u\ci\pi_1\,*}_{V/X} \ar[d]_\cong^{\eta_*}
\ar[r] & T^{\pi_1\,*}_{V/U} \ar[r] & 0 \\
0 \ar[r] & \pi_2^*(U^{u\,*}_{U/X}) \ar[r] & T^{u\ci\pi_2\,*}_{V/X}
\ar[r] & T^{\pi_2\,*}_{V/U} \ar[r] & 0.\!\!{} }
\end{gathered}
\label{dc8eq6}
\e
Define an isomorphism $\la:(\pi^\red_1)^*(L)\ra (\pi_2^\red)^*(L)$
by the commutative diagram of line bundles on $V^\red$: \vskip -15pt
\e
\begin{gathered}
{}\!\!\!\xymatrix@!0@C=135pt@R=55pt{
*+[r]{(\pi^\red_1)^*(L)}
\ar@{=}[d] \ar[r]_(0.6)\la & (\pi^\red_2)^*(L) \ar@{=}[r] &
*+[l]{\raisebox{-30pt}{$\begin{subarray}{l}\ts(\pi_2^\red)^*
\bigl(K_{U,s(U,u)}\bigr)\ot\\
\ts(\pi_2^\red)^*\bigl((\La^{\rm top}T^{u\,*}_{
U/X})^{\ot^{-2}}\bigr)\end{subarray}$\quad}}
\ar[d]_(0.55){\id\ot(\De_{\pi_2}\vert_{V^\red}^{\ot^{-2}})} \\
*+[r]{\begin{subarray}{l}\ts(\pi_1^\red)^*\bigl(K_{U,s(U,u)}\bigr)\ot\\
\ts(\pi_1^\red)^*\bigl((\La^{\rm top}T^{u\,*}_{
U/X})^{\ot^{-2}}\bigr)\end{subarray}}
\ar[d]^{\id\ot(\De_{\pi_1}\vert_{V^\red}^{\ot^{-2}})} &&
*+[l]{\begin{subarray}{l}\ts (\pi_2^\red)^*
\bigl(K_{U,s(U,u)}\bigr)\ot \\
\ts(\La^{\rm top}T^{\pi_2\,*}_{V/U})\vert_{V^\red}^{\ot^2}\ot
(\La^{\rm top}T^{u\ci\pi_2\,*}_{V/X})
\vert_{V^\red}^{\ot^{-2}}\end{subarray}}
\ar[d]_{\Up_{\pi_2}\ot\id} \\
*+[r]{{}\,\,\,\,\,\,\begin{subarray}{l}\ts (\pi_1^\red)^*
\bigl(K_{U,s(U,u)}\bigr)\ot \\
\ts(\La^{\rm top}T^{\pi_1,*}_{V/U})\vert_{V^\red}^{\ot^2}\ot
(\La^{\rm top}T^{u\ci\pi_1\,*}_{V/X})
\vert_{V^\red}^{\ot^{-2}}\end{subarray} }
\ar[rr]^(0.65){\Up_{\pi_1}\ot(\La^{\rm top}\eta_*)} &&
*+[l]{\begin{subarray}{l} \ts \qquad\,\,\, K_{V,s(V,u\ci\pi_1)}\ot \\
\ts \bigl(\La^{\rm top}T^{u\ci\pi_2\,*}_{V/X}
\bigr)\big\vert_{V^\red}^{\ot^{-2}},\end{subarray}\,\,\,\,{}}
}\!\!\!\!\!\!\!{}
\end{gathered}
\label{dc8eq7}
\e
where $\Up_{\pi_1},\Up_{\pi_2}$ are as in \eq{dc2eq35}, and $\eta_*$
as in \eq{dc8eq6}, and\/ $\De_{\pi_i}:\pi_i^*(\La^{\rm
top}T^{u\,*}_{U/X})\ra (\La^{\rm
top}T^{\pi_i,*}_{V/U})^{-1}\ot\La^{\rm top}T^{u\ci\pi_i\,*}_{V/X}$
for $i=1,2$ are induced by taking top exterior powers in the rows
of~\eq{dc8eq6}.

Let $(p_1,p_2,\th_{12})$ be a point in $V^\red\subseteq V$, so that
$p_1,p_2\in U$ with $\th_{12}:u(p_1)\Ra u(p_2)\in X$. We claim that
the following diagram commutes:
\e
\begin{gathered}
\xymatrix@C=43pt@R=11pt{*+[r]{(\pi^\red_1)^*(L)\vert_{(p_1,p_2)}}
\ar@{=}[r] \ar[d]^{\la\vert_{(p_1,p_2,\th_{12})}} & L\vert_{p_1}
\ar[rr]_(0.12){\mu_{p_1}} && *+[l]{\bigl(\La^{\rm
top}T_{u(p_1)}^*X\bigr)^{\ot^2}\!\!\ot\!\bigl(\La^{\rm
top}\fIso_{u(p_1)}(X)\bigr)^{\ot^2}} \ar[d]_{(\th_{12})_*} \\
*+[r]{(\pi^\red_2)^*(L)\vert_{(p_1,p_2)}}
\ar@{=}[r] & L\vert_{p_2} \ar[rr]^(0.12){\mu_{p_2}} &&
*+[l]{\bigl(\La^{\rm top}T_{u(p_2)}^*X\bigr)^{\ot^2}\!\!\ot\!
\bigl(\La^{\rm top}\fIso_{u(p_2)}(X)\bigr)^{\ot^2}.} }
\end{gathered}
\label{dc8eq8}
\e
To see this, combine \eq{dc8eq5}, the restriction of \eq{dc8eq7} to
$(p_1,p_2,\th_{12})$, equation \eq{dc2eq36} applied to give
expressions for $\Up_{\pi_1}\vert_{(p_1,p_2,\th_{12})}$ and
$\Up_{\pi_2}\vert_{(p_1,p_2,\th_{12})}$, and natural compatibilities
between the isomorphisms obtained by taking top exterior powers in
the sequences \eq{dc2eq34}, \eq{dc2eq45}, and~\eq{dc8eq6}.

Now let $(p_1,p_2,p_3,\th_{12},\th_{23})\in W^\red$. Using
\eq{dc8eq8} at $(p_1,p_2,\th_{12}),(p_2,p_3,\th_{23})$ and
$(p_3,p_1,\th_{12}^{-1}\ci\th_{23}^{-1})$ shows that
\begin{equation*}
\la\vert_{(p_3,p_1,\th_{12}^{-1}\ci\th_{23}^{-1})}\ci
\la\vert_{(p_2,p_3,\th_{23})}\ci\la\vert_{(p_1,p_2,\th_{12})}=
\id:L\vert_{p_1}\longra L\vert_{p_1}.
\end{equation*}
This is the restriction of
$(\pi_{31}^\red)^*(\la)\ci(\pi_{23}^\red)^*(\la)\ci
(\pi_{12}^\red)^*(\la)=\id_{(\pi_1^\red\ci\pi_{12}^\red)^*(L)}$ in
\eq{dc8eq3} to $(p_1,p_2,p_3,\th_{12},\th_{23})$. Since $W^\red$ is
reduced, the equation is implied by its restriction to each point of
$W^\red$. Thus $(\pi_{31}^\red)^*(\la)\ci(\pi_{23}^\red)^*(\la)\ci
(\pi_{12}^\red)^*(\la)=\id_{(\pi_1^\red\ci\pi_{12}^\red)^*(L)}$. So
by \eq{dc8eq3}, there exists a line bundle $K_{X,s}$ on $X^\red$,
unique up to canonical isomorphism, with an
isomorphism~$\chi:F(K_{X,s})\,{\smash{\buildrel\cong\over\longra}}\,
(L,\la)$.

We claim that $K_{X,s}$ is independent up to canonical isomorphism
of the choice of smooth atlas $u:U\ra X$ above. To see this, note
that if $u':U'\ra X$ is another atlas, yielding $K_{X,s}'\ra
X^\red$, then setting $U'':=U\amalg U'$ and $u'':=u\amalg u'$ gives
a third atlas $u'':U''\ra X$, yielding $K_{X,s}''\ra X^\red$. Now
$K_{X,s}''$ satisfies properties on the analogue of \eq{dc8eq2} for
$U''=U\amalg U'$. Restricting to the subdiagram generated by
$U\subseteq U''$ shows $K_{X,s}''$ satisfies the same properties as
$K_{X,s}$, and restricting to the subdiagram generated by
$U'\subseteq U''$ shows $K_{X,s}''$ satisfies the same properties as
$K_{X,s}'$. So we have canonical isomorphisms $K_{X,s}\cong
K_{X,s}''\cong K_{X,s}'$.

For Theorem \ref{dc2thm4}(a), let $x\in X$. As $u:U\ra X$ is
surjective, there exists $p\in U$ and a 2-isomorphism $\vp:x\Ra
u(p)$. Define an isomorphism $\ka_x$ as in \eq{dc2eq43} by the
commutative diagram of isomorphisms
\e
\begin{gathered}
\xymatrix@C=36pt@R=11pt{
*+[r]{K_{X,s}\vert_x} \ar[d]^{\vp_*} \ar[rrrr]_(0.39){\ka_x} &&&&
*+[l]{\bigl(\La^{\rm top}T_x^*X\bigr)^{\ot^2}\!\!\ot\!\bigl(\La^{\rm
top}\fIso_x(X)\bigr)^{\ot^2}} \ar[d]_{\vp_*} \\
*+[r]{K_{X,s}\vert_{u(p)}} \ar[r]^(0.7){\chi\vert_p} & L\vert_p
\ar[rrr]^(0.2){\mu_p} &&&
*+[l]{\bigl(\La^{\rm top}T_{u(p)}^*X\bigr)^{\ot^2}\!\!\ot\!\bigl(\La^{\rm
top}\fIso_{u(p)}(X)\bigr)^{\ot^2},} }
\end{gathered}
\label{dc8eq9}
\e
for $\mu_p$ as in \eq{dc8eq5}. To see $\ka_x$ is independent of the
choice of $p,\vp$, suppose $p',\vp'$ are alternate choices. Then
$(p,p',\vp'\ci\vp^{-1})\in V^\red$. Compare \eq{dc8eq9} for $p,\vp$
and $p',\vp'$ with \eq{dc8eq8} for $(p,p',\vp'\ci\vp^{-1})$, and use
the commutative diagram
\begin{equation*}
\xymatrix@C=160pt@R=13pt{
*+[r]{K_{X,s}\vert_{u(p)}} \ar[d]^{(\vp'\ci\vp^{-1})_*}
\ar[r]_(0.55){\chi\vert_{p}} & *+[l]{L\vert_{p}}
\ar[d]_{\la\vert_{(p,p',\vp'\ci\vp^{-1})}} \\
*+[r]{K_{X,s}\vert_{u(p')}} \ar[r]^(0.55){\chi\vert_{p'}} &
*+[l]{L\vert_{p'},\!\!\!{}} }
\end{equation*}
as $\chi:F(K_{X,s})\,{\smash{\buildrel\cong\over\longra}}\, (L,\la)$
is an isomorphism in the lower category in~\eq{dc8eq3}.

By the same argument, we can show that $\ka_x$ is independent of the
choice of atlas $u:U\ra X$ above. For if $u':U'\ra X$ is another
atlas, and $p\in U$, $\vp:x\Ra u(p)$ and $p'\in U'$, $\vp':x\Ra
u'(p')$ may be used to define $\ka_x$ using the atlases $u:U\ra X$,
$u':U'\ra X$, then passing to the atlas $u'':U''\ra X$ for
$U''=U\amalg U'$ and $u''=u\amalg u'$ as above, defining $\ka_x$
using $U,u,p,\vp$ or using $U',u',p',\vp'$ is equivalent to defining
$\ka_x$ using $p,\vp$ or $p',\vp'$ in the single atlas $U'',u''$,
which give the same answer as above. This proves
Theorem~\ref{dc2thm4}(a).

For part (b), suppose $t:T\ra X$ is a smooth 1-morphism, and set
$U':=T\amalg U$ and $u':=t\amalg u$. Then $u':U'\ra X$ is another
smooth atlas for $X$, so as above replacing $U,u$ by $U',u'$ yields
a canonically isomorphic line bundle $K_{X,s}'\cong K_{X,s}$, with
an isomorphism $\chi':F(K_{X,s}')\ra (L',\la')$ for $L',\la'$ given
by the analogues of \eq{dc8eq4}, \eq{dc8eq7} for $U',u'$, so that in
line bundles on $T^\red\amalg U^\red$ we have
\begin{align*}
\chi'&:K_{X,s}'(T\amalg U,t\amalg u)=
K_{X,s}'(T,t)\amalg K_{X,s}'(U,u)
\,{\buildrel\cong\over\longra}\\
&\bigl[K_{T,s(T,t)}\ot (\La^{\rm top}T^{t\,*}_{
T/X})\vert_{T^\red}^{\ot^{-2}}\bigr]\amalg
\bigl[K_{U,s(U,u)}\ot(\La^{\rm top}T^{u\,*}_{
U/X})\vert_{U^\red}^{\ot^{-2}}\bigr].
\end{align*}
Since $K_{X,s}$ was only determined up to canonical isomorphism
anyway, we may take $K_{X,s}'=K_{X,s}$, and
$\chi'\vert_{U^\red}=\chi$. Define $\Ga_{T,t}:=\chi'\vert_{T^\red}$,
as in \eq{dc2eq44}. It is a natural isomorphism, proving
Theorem~\ref{dc2thm4}(b).

For part (c), continuing to use the same notation, let $p\in
T^\red\subseteq T$, so that $t(p):=t\ci p\in X$. As above the
definition of $\ka_{t(p)}$ is independent of the choice of atlas
$u:U\ra X$, so we can define it using $u':U'\ra X$ for $U'=T\amalg
U$ and $u'=t\amalg u$, for the point $p\in T^\red\subseteq
T^\red\amalg U^\red=U^{\prime\red}$ and 2-morphism
$\vp=\id_{t(p)}:t(p)\Ra t(p)$. Combining \eq{dc8eq9} at $x=t(p)$
with $\Ga_{T,t}:=\chi'\vert_{T^\red}$ and the definition \eq{dc8eq5}
of $\mu_p$ shows that \eq{dc2eq46} commutes. This completes the
proof of Theorem~\ref{dc2thm4}.

\medskip

\noindent{\small\sc The Mathematical Institute, Radcliffe Observatory Quarter, Woodstock Road, Oxford, OX2 6GG, U.K.}

\noindent{\small\sc E-mail: \tt joyce@maths.ox.ac.uk}


\begin{thebibliography}{99}
\addcontentsline{toc}{section}{References}

\bibitem{Arti} M. Artin, {\it Grothendieck topologies}, Harvard
University Notes, 1962.

\bibitem{Behr} K. Behrend, {\it Donaldson--Thomas type invariants
via microlocal geometry}, Ann. of Math. 170 (2009), 1307--1338.
math.AG/0507523.

\bibitem{BeFa} K. Behrend and B. Fantechi, {\it The intrinsic normal
cone}, Invent. Math. 128 (1997), 45--88. alg-geom/9601010.

\bibitem{BBBJ} O. Ben-Bassat, C. Brav, V. Bussi, and D. Joyce,
{\it A `Darboux Theorem' for shifted symplectic structures on
derived Artin stacks, with applications}, arXiv:1312.0090, 2013.

\bibitem{BBDJS} C. Brav, V. Bussi, D. Dupont, D. Joyce, and B.
Szendr\H oi, {\it Symmetries and stabilization for sheaves of
vanishing cycles}, arXiv:1211.3259, 2012.

\bibitem{BBJ} C. Brav, V. Bussi and D. Joyce, {\it A Darboux
theorem for derived schemes with shifted symplectic structure},
arXiv:1305.6302, 2013.

\bibitem{Buss} V. Bussi, {\it Categorification of Lagrangian
intersections on complex symplectic manifolds using perverse sheaves
of vanishing cycles}, in preparation, 2013.

\bibitem{BJM} V. Bussi, D. Joyce and S. Meinhardt, {\it On motivic
vanishing cycles of critical loci}, arXiv:1305.6428, 2013.

\bibitem{DiSz} A. Dimca and B. Szendr\H oi, {\it The Milnor fibre
of the Pfaffian and the Hilbert scheme of four points on $\C^3$},
Math. Res. Lett. 16 (2009) 1037--1055. arXiv:0904.2419.

\bibitem{FOOO} K. Fukaya, Y.-G. Oh, H. Ohta and K. Ono,
{\it Lagrangian intersection Floer theory --- anomaly and
obstruction}, Parts I \& II. AMS/IP Studies in Advanced Mathematics,
46.1 \& 46.2, A.M.S./International Press, 2009.

\bibitem{Gome} T.L. G\'omez, {\it Algebraic stacks}, Proc. Indian Acad.
Sci. Math. Sci. 111 (2001), 1--31. math.AG/9911199.

\bibitem{GrRe} H. Grauert and R. Remmert, {\it Coherent analytic
sheaves}, Grundlehren der math. Wiss. 265, Springer-Verlag, Berlin,
1984.

\bibitem{GuRo} R.C. Gunning and H. Rossi, {\it Analytic functions of
several complex variables}, Prentice-Hall, Englewood Cliffs, NJ,
1965.

\bibitem{Hart} R. Hartshorne, {\it Algebraic Geometry}, Graduate
Texts in Math. 52, Springer, New York, 1977.

\bibitem{Illu1} L. Illusie, {\it Complexe cotangent et
d\'eformations. I}, Springer Lecture Notes in Math. 239,
Springer-Verlag, Berlin, 1971.

\bibitem{Illu2} L. Illusie, {\it Cotangent complex and deformations
of torsors and group schemes}, pages 159--189 in Springer Lecture
Notes in Math. 274, Springer-Verlag, Berlin, 1972.

\bibitem{JoSo} D. Joyce and Y. Song, {\it A theory of generalized
Donaldson--Thomas invariants}, Mem. Amer. Math. Soc. 217 (2012), no.
1020. arXiv:0810.5645.

\bibitem{KoSo1} M. Kontsevich and Y. Soibelman, {\it Stability
structures, motivic Donaldson--Thomas invariants and cluster
transformations}, arXiv:0811.2435, 2008.

\bibitem{KoSo2} M. Kontsevich and Y. Soibelman, {\it Cohomological
Hall algebra, exponential Hodge structures and motivic
Donaldson--Thomas invariants}, Commun. Number Theory Phys. 5 (2011),
231--352. arXiv:1006.2706.

\bibitem{LaMo} G. Laumon and L. Moret-Bailly, {\it Champs
alg\'ebriques}, Ergeb. der Math. und ihrer Grenzgebiete 39,
Springer-Verlag, Berlin, 2000.

\bibitem{LaOl1} Y. Laszlo and M. Olsson, {\it The six operations
for sheaves on Artin stacks. I\/} and {\it II}, Publ. Math. Inst.
Hautes \'Etudes Sci. 107 (2008), 109--168 and 169--210.
math.AG/0512097 and math.AG/0603680.

\bibitem{LaOl2} Y. Laszlo and M. Olsson, {\it Perverse
$t$-structure on Artin stacks}, Math. Z. 261 (2009), 737--748.
math.AG/0606175.

\bibitem{Maul} D. Maulik, {\it Motivic residues and
Donaldson--Thomas theory}, in preparation, 2013.

\bibitem{Olss} M. Olsson, {\it Sheaves on Artin stacks}, J. Reine
Angew. Math. 603 (2007), 55--112.

\bibitem{Pala1} V.P. Palamodov, {\it Deformations of complex
spaces}, Russian Math. Surveys 31 (1976), 129--197.

\bibitem{Pala2} V.P. Palamodov, {\it Cohomology of analytic
algebras}, Trans. Moscow Math. Soc. 1983, No. 2. (1983), 1--61.

\bibitem{Pala3} V.P. Palamodov, {\it The tangent complex of an
analytic space}, A.M.S. Translations 122 (1984), 119--171.

\bibitem{Pala4} V.P. Palamodov, {\it Deformations of complex
spaces}, pages 105--194 in S.G. Gindikin and G.M. Khenkin, editors,
{\it Several complex variables IV}, Encyclopaedia of Mathematical
Sciences 10, Springer-Verlag, Berlin, 1990.

\bibitem{PaTh} R. Pandharipande and R.P. Thomas, {\it Almost
closed\/ $1$-forms}, \hfil\break arXiv:1204.3958, 2012.

\bibitem{PTVV} T. Pantev, B. To\"en, M. Vaqui\'e and G. Vezzosi,
{\it Shifted symplectic structures}, Publ. Math. I.H.E.S. 117
(2013), 271--328. arXiv:1111.3209.

\bibitem{Serr} J.-P. Serre, {\it G\'eom\'etrie alg\'ebrique et
g\'eom\'etrie analytique}, Ann. Inst. Fourier 6 (1956), 1--42.

\bibitem{Sumi} H. Sumihiro, {\it Equivariant completion}, J.
Math. Kyoto Univ. 14 (1974), 1--28.

\bibitem{Thom} R.P. Thomas, {\it A holomorphic Casson invariant for
Calabi--Yau $3$-folds, and bundles on $K3$ fibrations}, J. Diff.
Geom. 54 (2000), 367--438. \hfil\break math.AG/9806111.

\bibitem{Toen} B. To\"en, {\it Higher and derived stacks: a global
overview}, pages 435--487 in {\it Algebraic Geometry --- Seattle
2005}, Proc. Symp. Pure Math. 80, Part 1, A.M.S., Providence, RI,
2009. math.AG/0604504.

\bibitem{ToVe1} B. To\"en and G. Vezzosi, {\it Homotopical Algebraic
Geometry II: Geometric Stacks and Applications}, Mem. Amer. Math.
Soc. 193 (2008), no. 902. math.AG/0404373.

\bibitem{ToVe2} B. To\"en and G. Vezzosi, {\it From HAG to DAG:
derived moduli stacks}, pages 173--216 in {\it Axiomatic, enriched
and motivic homotopy theory}, NATO Sci. Ser. II Math. Phys. Chem.,
131, Kluwer, Dordrecht, 2004. math.AG/0210407.

\bibitem{Vezz} G. Vezzosi, {\it Derived critical loci I: Basics},
arXiv:1109.5213, 2011.

\end{thebibliography}
\end{document}